\documentclass[12pt,reqno]{amsart}
\usepackage{cite}
\usepackage{mathtools}
\usepackage{amsmath,color}
\usepackage[v2,all]{xy}
\usepackage{amsfonts,amssymb}
\usepackage{mathrsfs}

\usepackage{xr-hyper}

\usepackage{hyperref}

\usepackage{amssymb}

\usepackage{bold}
 \def\unit{\Eins}
 \def\gh{\mathbbnew{\Gamma}}

\catcode `\@=11
\def\numberbysection{\@addtoreset{equation}{section}
         \renewcommand{\theequation}{\thesection.\arabic{equation}}}
\numberbysection
\def\subsubsection{\@startsection{subsubsection}{3}%
  \normalparindent{.5\linespacing\@plus.7\linespacing}{-.5em}%
  {\normalfont\bfseries}}
\newtheorem{thm}{Theorem}[section]

\newtheorem{lem}[thm]{Lemma}

\newtheorem{prop}[thm]{Proposition}

\newtheorem{cor}[thm]{Corollary}

\theoremstyle{definition}

\newtheorem{df}[thm]{Definition}

\newtheorem{rmk}[thm]{Remark}

\newtheorem{ex}[thm]{Example}

\def\C{\mathcal C}

\def\O{\mathcal{O}}
\def\Ab{\mathcal{A}b}
\def\Z{\mathbb{Z}}
\def\Q{\mathbb{Q}}
\def\Vect{\mathcal{V}ect}
\def\kVect{\Vect_k}

\def\Set{\mathcal{S}et}
\def\V{\mathcal V}
\def\H{\mathscr H}
\def\E{\mathcal{E}}
\def\Mor{\it Mor}
\newcommand{\leftsub}[2]{{\vphantom{#2}}_{#1}{#2}}

\def\SSigma{$\Sigma$}

\makeatother
\newcommand{\EZDIAG}[5]{\xymatrix
@C+=2.5cm{*+[r]{#1}
\ar@(u,l)_(0.62){\displaystyle #5}[]
\ar@<1ex>^-{#3}[r]&\ar@<1ex>^-{#4}[l]#2}}

\def\del{\partial}
\def\Coop{\check\O}
\def\id{{\mathrm{id}}}

\def\SS{\mathbb{S}}
\def\odo{\otimes \dots\otimes}

\def\kdk{,\dots,}
\def\eps{\epsilon}

\def\B{\mathscr{B}}

\def\nn{\nonumber}
\def\I{\mathscr {I}}
\def\J{\mathscr {J}}
\def\C{\mathscr {C}}

\def \colim {\mathop {\rm colim}}
\def\F{\mathcal{F}}
\def\FF{\mathfrak F}
\def\asts{{\mathcal V}}
\def\ff{\mathfrak{f}}

\def\clusters{{\mathcal F}}
\def\N{\mathbb{N}}
\def\dottree{\bullet\hskip-4.5pt |}
\def\simpcat{{\Delta}}

\def\egr{\unit_{\emptyset}}
\def\nn{\nonumber}

\def\eps{\epsilon}
\def\G{\Gamma}

\def\Vect{\mathcal{V}ect}

\def\la{\langle}
\def\ra{\rangle}

\def\CycAss{CAssoc}

\def\Graphs{{\mathcal G}\it raphs}
\def\CalC{{\mathcal C}}

\def\Agg{{\mathcal A}gg}
\def\Set{{\mathcal S}et}

\def\Crl{{\mathcal C}\it rl}

\def\CCyclic{\mathfrak  {C}}

\def\operads{{\mathfrak O}}

\def\F{\mathcal F}
\def\FF{\mathfrak F}

\def\GG{\mathfrak G}

\def\R{{\mathbb R}}
\def\C{\CalC}
\def\Z{{\mathbb Z}}
\def\N{{\mathbb N}}
\def\G{\Gamma}

\def\del{\partial}
\def\colim{\mathrm{colim}}

\def\FinSet{\mathcal{F}in\mathcal{S}et}
\def\FFinSet{\mathfrak{F}in\mathfrak{S}et}
\def\surj{FS}
\def\Surj{\mathfrak{FS}}
\def\Inj{\mathfrak{FI}}
\def\inj{FI}
\def\Int{\Inj_{*,*}}

\newcommand{\op}{\mathcal}

\def\O{{\mathcal O}}

\def\SS{{\mathbb S}}

\def\odo{\otimes \cdots \otimes}

\def\crl{*}

\def\deg{\mathrm{deg}}

\def\kill{\it trun}
\def\leaf{\it leaf}

\newcommand{\Hom}{\operatorname{Hom}}
\newcommand{\drpullback}[1][dr]{\save*!/#1-1.2pc/#1:(-1,1)@^{|-}\restore}

\def\V{\asts}
\def\asts{{\mathcal V}}
\def\F{\clusters}
\def\clusters{{\mathcal F}}

\def\Ab{{\mathcal Ab}}

\def\opcat{{\mathcal O }ps}
\def\op{{\mathcal O }p}

\def\I{{\mathcal I}}

\def\inthom{\underline{Hom}}

\def\ta{\twoheadrightarrow}

\def\triv{\underline{*}}

\newcommand{\X}{{\mathcal X}}

\def\Bquot{\B_\Q^{quot}}
\def\Hquot{\H_\Q^{quot}}
\def\epsquot{\eps^{quot}}
\def\Deltaquot{\Delta^{quot}}
\def\etaquot{\eta^{quot}}

\def\epsred{\eps^{red}}
\def\Deltared{\Delta^{red}}

\def\Biso{\B^{iso}}

\def\epsiso{\eps^{iso}}
\def\Deltaiso{\Delta^{iso}}
\def\etaiso{\eta^{iso}}

\def\otb{\leftsub{*}\otimes_*}

\def\ops{\opcat}
\def\op{op}

\def\FFdeco{\FF_{{\it dec}\O}}

\def\emptyset{\varnothing}

\numberbysection

\begin{document}

\title[Hopf Algebras II:  general categorical formulation]{Three Hopf algebras from number theory, physics \& topology, and their common background II: general categorical formulation}

\author[I.\ G\'alvez--Carrillo]{Imma G\'alvez--Carrillo}
\email{m.immaculada.galvez@upc.edu}
\address{Departament de Matem\`atiques, Universitat Polit\`ecnica de Catalunya,
Edifici TR1-Palau d'Ind\'ustries, Carrer Colom, 1.
08222 Terrassa,
Spain.}

\author[R.~M.~Kaufmann]{Ralph M.\ Kaufmann}
\email{rkaufman@math.purdue.edu}

\address{Purdue University Department of Mathematics, West Lafayette, IN 47907
}

\author[A.~Tonks]{Andrew Tonks}
\email{apt12@le.ac.uk}
\address{Department of Mathematics,
University of Leicester,
University Road,
Leicester, LE1 7RH, UK.}

\def\emptyset{\varnothing}

\begin{abstract}
We consider three {\it a priori} totally different setups for Hopf algebras from number theory, mathematical physics and algebraic topology.
These are the Hopf algebra of Goncharov for multiple zeta values, that of Connes--Kreimer for renormalization, and a Hopf algebra constructed by Baues to study double loop spaces.
We show that these examples can be successively unified by considering simplicial objects,
co--operads with multiplication and Feynman categories at the ultimate level. These considerations open the door to new constructions and reinterpretations of known constructions in a large common framework which is presented step--by--step with examples throughout. In this second part of two papers, we give the general categorical formulation.
\end{abstract}

\maketitle

\tableofcontents

\section*{Introduction}
In this sequence of two papers we provide a common background for Hopf algebras that appeared prominently in vastly different areas of mathematics. Standout examples are the Hopf algebras of Goncharov \cite{Gont} in number theory, those of Connes and Kreimer \cite{CK,CK1,CK2} in mathematical physics and that of Baues \cite{BauesHopf} in algebraic topology.  There are several Hopf algebras of Connes and Kreimer and variations of these which are of great interest in physics and number theory, e.g.\ \cite{brown,brownann}. The most basic ones being those for rooted trees, see \cite{foissyCR1,foissyCR2}. These algebras and those of Baues and Goncharov have  been identified as examples of universal constructions  stemming from simplicial and operadic setups in the first part \cite{HopfPart1}.
The next level of complexity is represented by the Connes--Kreimer Hopf algebras for renormalization defined on graphs.
These Hopf algebras are most properly discussed on the categorical level. This is the purview of this second part.

The natural setup for this  definitive source of Hopf algebras of this type are the Feynman categories of \cite{feynman}. The results of \cite{HopfPart1}  can be re--derived  by restricting to special types of Feynman categories. As in the first part,  the key observation is that the  Hopf algebras are quotients of bi--algebras. Here we add that these bi--algebras have a natural origin coming from Feynman categories, which can be seen as  a special type of monoidal category. This allows us to uncover  the ``raison d'\^etre'' of the co--product  simply as  the dual to the  partial product given by the composition in the Feynman category.  The quotient is furthermore identified as the natural quotient making the bi--algebras connected.

In particular, we show that under reasonable assumptions a  Feynman category gives rise to a Hopf algebra formed by the free Abelian group of its morphisms. Here the co--product,  motivated by a discussion with D.\ Kreimer, is defined by deconcatenation. With hindsight, this type of co--product goes back at least as far as \cite{JR} or \cite{Moebiusguy}, who considered a deconcatenation co--product from a combinatorial point of view. Feynman categories are monoidal, and this monoidal structure yields a product. Although it is not true in general for any monoidal category that the multiplication and comultiplication are compatible and form a bi--algebra, it is for Feynman categories, and hence also for their opposites. This also gives a new understanding for the axioms of a Feynman category. The case relevant for co--operads with multiplication, treated in the first part, is the Feynman category of finite sets and surjections and its enrichments by operads.
The constructions of the bi--algebra then correspond to the pointed free case considered in \cite{HopfPart1} if the co--operad is the dual of an operad. Invoking opposite categories, one can treat the co--operads appearing in \cite{HopfPart1} directly. For this one notices that the opposite Feynman category, that for co--algebras, can be enriched by co--operads.
 It is here that we can also say that the two constructions of Baues and Goncharov are related by Joyal duality to the operad of surjections.

The quotients  are obtained   by ``dividing out isomorphisms''. This amounts to taking co--invariants or alternatively dividing out by certain co--ideals. This  allows us to distinguish the levels between planar, symmetric, labeled and unlabeled versions analogous to the tree Hopf algebras of Connes and Kreimer and \cite{HopfPart1}. To actually get the Hopf algebras, rather than just bi--algebras, one  has to take  quotients and require certain connectedness assumptions. Here the conditions become very transparent. Namely, the unit, hidden in the three examples \cite[\S\ref{P1-hopfpar}]{HopfPart1} by normalizations, will be given by the unit endomorphism  of the monoidal unit $\unit$ of the Feynman category, viz.\  $id_{\unit}$. Isomorphisms keep the co--algebra from being co--nilpotent. Even if there are no isomorphisms, still all identities are group--like and hence the co--algebra is not connected. This explains the necessity of taking quotients of the bi--algebra to obtain a Hopf algebra. We give the technical details of the two quotients, first removing isomorphisms and then identifying all identity maps.

There is also a distinction here between the non--symmetric and the symmetric case. While in the non--symmetric case, there is a Hopf structure before taking the quotient, the passing to the quotient, viz.\ coinvariants is necessary in the symmetric case.

The categorical constructions are more general than those in \cite{HopfPart1} as there are other Feynman categories besides those which yield  co-operads with multiplication. One of the most interesting examples of a Feynman category which yields a deep connection to  mathematical physics is the Feynman category $\GG$ whose ``morphisms are graphs'', see \cite[\S2]{feynman} and \S\ref{graphexpar}.   This is the medium which allows us to obtain  graph Hopf algebras of Connes and Kreimer,  those based on 1-PI graphs and motic graphs, the latter yielding the new Hopf algebras of Brown \cite{brown}.  Mathematically $\GG$ is also at the heart of the whole zoo of operad--types \cite{feynman,decorated,BergerKaufmann}. Consequentially, there are also the Hopf algebras  corresponding to cyclic operads, modular operads, etc.. To obtain these examples several general constructions on Feynman categories, such as enrichment, decoration, universal operations, and free construction come into play. These constructions also give interrelations between the examples.
Among the examples discussed are the examples of \cite{HopfPart1} which are analyzed in various contexts that provide new depths of understanding.
In particular, we revisit operads, simplicial structures and Joyal duality in the context of Feynman categories, decorations and enrichments.

\subsubsection*{Organization of the paper and overview of results}

We start by giving an overview of the results of the rest of the paper in \S\ref{previewpar}. We then treat the non--symmetric case, where the bi--algebra equation follows directly from the conditions on a Feynman category, viz.\ Theorem \ref{f-to-b-nonsym}. With more work, there is a version for symmetric Feynman categories; see Theorem \ref{Bisothm}. Under certain conditions, there is again a Hopf quotient, see \S \ref{Hopfpar}. In order to get a practical handle, we consider graded Feynman categories.
The result is  Theorem \ref{HopfFeythm}. We conclude the section with a discussion of functoriality in \S\ref{functpar}. This analysis explains why there is no Hopf algebra map from the Hopf algebra of Connes--Kreimer to that of Goncharov.

The shorter \S\ref{alternatepar} gives further constructions and twists. It contains the original construction on indecomposables as well as a different quotient construction.

Having the whole theory at hand, we give a detailed discussion of a slew of examples in \S\ref{constexpar}, including previously undiscussed ones. The reader is encouraged to skip ahead to these examples at any time for concreteness, and some references to these examples are given throughout. Here we first treat the examples introduced in \cite[\S\ref{P1-hopfpar}]{HopfPart1} as well as the Connes--Kreimer category for graphs.
This discussion also identifies the construction of \cite[\S\ref{P1-operadpar} and \S\ref{P1-gencooppar}]{HopfPart1} as the special case of Feynman categories with trivial vertex set.
We then review constructions from \cite{feynman} to put these special cases into a larger context. These include decorations (\S\ref{decopar}), enrichments
(\S\ref{enrichedpar}) and universal operations (\S\ref{universalpar}).
These explain the underlying mechanisms and allow for alterations for future applications. Among the special cases of these general construction is the motic Hopf algebra of Brown. The enrichment adds another layer of technical sophistication and is kept short  referring to \cite{feynman} for additional details.
We also consider colored operads, which naturally appear in this situation and show that the formulas for Goncharov's Hopf algebra become apparent in the colored context. This also gives a bridge back to the simplicial setting, as the nerve of a category is naturally at the same time simplicial and a colored operad,
see Proposition \ref{catprop}.

The subsection \S\ref{simpfeypar} also contains a detailed discussion of simplicial structures and the relationship with Joyal duality.
The latter is of independent interest, since this duality explains
the ubiquitous occurrence of two types of formulas, those with repetition and those without repetition, in the contexts of number theory, mathematical physics and algebraic topology.
This duality also explains the two graphical versions used in this type of calculations, polygons vs.\ trees, which are now just Joyal duals of each other, see especially \S\S\ref{joyal2par}--\ref{joyal2par}. The presentation of Joyal duality is novel, both graphically and combinatorially.

In \S\ref{summarypar}, we give  a short summary of the given constructions in both parts, their interrelations and specializations to the original examples and end with an outlook to further results.

To be self-contained the paper also has an appendix on graphs and their formalization due to \cite{BorMan,feynman} we use.
Again, this can serve as an independent guide to a very useful tool as this particular presentation of graphs is ``just right'' in terms of complexity to handle
all combinatorial intricacies, such as those appearing when considering auto-- and isomorphisms.

\subsection*{Acknowledgments}
We would like to thank  D.~Kreimer, F.~Brown, P.~Lochack,  Yu.~I.~Manin, H.~Gangl, M.~Kapranov, JDS Jones, P.~Cartier and A.~Joyal for enlightening discussions.

RK gratefully acknowledges support from  the Humboldt Foundation and the Simons Foundation, the
Institut des Hautes Etudes Scientifiques and the Max--Planck--Institut
for Mathematics in Bonn and the University of Barcelona for their support. RK also thanks the Institute for Advanced Study in Princeton and the Humboldt University in Berlin for their hospitality and
 their support.

IGC was partially supported by Spanish Ministry of Science and Catalan government grants
MTM2012-38122-C03-01, MTM2013-42178-P, 2014-SGR-634, MTM2015-69135-P, MTM2016-76453-C2-2-P (AEI/ FEDER, UE), MTM2017-90897-REDT, and 2017-SGR-932,
and AT by
MTM2013-42178-P, and MTM2016-76453-C2-2-P (AEI/FEDER, UE)
all of which are gratefully acknowledged.

We also thankfully acknowledge the Newton Institute where the idea for this work was conceived during the activity on ``Grothendieck-Teichm\"uler Groups, Deformation and Operads''.
Last but not least, we thank the  Max--Planck--Institut
for Mathematics for the activity on ``Higher structures in Geometry and Physics'' which allowed our collaboration to put  major parts of the work  into their current form.

\subsection*{Notation}
As usual for a set $X$ with an action of a group $G$, we will denote the invariants by $X^G=\{x|g(x)=x\}$
and the  co--invariants by $X_G=X/\sim$ where $x\sim y$ if and only if there exists a $g\in G:g(x)=y$.

For an object $V$ in a monoidal category, we denote by $TV$ the free unital algebra on $V$, that is $TV=\bigoplus_n V^{\otimes n}$, in the case of an Abelian monoidal category,  and by $\bar TV$ the free algebra on $V$, that is reduced the tensor algebra on $\bar TV=\bigoplus_{n\geq 1}V^{\otimes n}$  in the case of an Abelian monoidal category. Similarly $SV=\bigoplus_{n\geq 0}V^{\odot n}$ denotes the free symmetric algebra and $\bar SV$ the free non--unital symmetric algebra. We use the notation $\odot$ for the symmetric aka.\ symmetrized, aka.\ commutative tensor product: $V^{\odot ^n}=(V^{\otimes n})_{\SS_n}$ where $\SS_n$ permutes the tensor factors.

Furthermore, we use $\underline{n}=\{1,\dots,n\}$ and  denote by $[n]$ to be the category with $n+1$ objects $\{0,\dots, n\}$ and morphisms generated by the chain $0\to 1 \to \dots \to n$.

Given two functors $f:\F'\to \F$ and $g:\F''\to \F$, we denote the comma category by $(f\downarrow g)$, or ---if the functors are clear form the context--- simply by $(\F'\downarrow \F'')$. The objects are triples $(X,Y,\phi:f(X)\to g(Y))$ with $X\in Obj(\F'), Y\in Obj(\F'')$. Morphisms from $(X,Y,\phi)$ to $(X',Y',\phi')$ are pairs $(\psi:X\to X',\psi':Y\to Y')$ such that $g(\psi')\circ\phi=\phi'\circ f(\psi)$.

\section{The general case: Bi-- and Hopf algebras from Feynman categories}
\label{feynmanpar}

\subsection{Preview}
\label{previewpar} As a paradigm, let us consider the Connes--Kreimer Hopf algebra of graphs, in particular the core Hopf algebra \cite{KreimerCore}, see also \S\ref{graphexpar}. The key point to understand this example is that the graphs (along with extra data) form the morphisms of a special type of monoidal
category, a Feynman category as introduced in \cite{feynman}, also see \S\ref{feydefpar} below.
 The graphs in particular appear in the Feynman category $\GG$ defined in \cite[\S2]{feynman}, see \S\ref{graphexpar} below. The pairs of sub-- and quotient graphs appearing in the co--product  \eqref{graphcoprodeq} then represent factorizations of morphisms; see  Figure \ref{phi3fig} for an example. This exhibits the co--product as deconcatenation. The product is induced by the monoidal product which in this case is simply disjoint union.

More generally for a Feynman category we will use deconcatenation and monoidal product to construct a bi--algebra and then show that under certain natural conditions  a quotient of this bi--algebra
yields a Hopf algebras.
This makes the theory particularly transparent and allows to recover the previous constructions of \cite{HopfPart1} by specialization. Or, taking the reverse perspective, we can generalize the constructions appearing in \cite{HopfPart1} by lifting them to the categorical level.

More specifically, in a monoidal category there are two products on morphisms, the tensor product $\otimes$ and the partially defined product of
composition $\circ$.
The product for morphisms will just be their tensor product. The co--product with be dual to partial composition product $\circ$. Unlike the composition, deconcatenation is not a partial operation, but
 rather unconditionally defined. The compatibility, viz.\ bi--algebra equation,  is guaranteed by the axioms of a Feynman category.

 There are three main types of examples for Feynman categories, the first are  of combinatorial type and are based on sets. The second are those of graph type, where the graphs are a structure of the morphisms. These also appear in physics in the form of Feynman diagrams, whence the name. The last type are the enriched Feynman categories. These will be discussed in \S\ref{enrichedpar}.

The Hopf algebras of Goncharov and Baues are combinatorial as are the tree Hopf algebras of Connes and Kreimer.
The graph Hopf algebra of Connes and Kreimer is of graph type.  The Hopf algebras from co--operads more generally are of enriched type, however,  they still have a description of combinatorial type if the co--operad is in $\Set$.

There are also two flavors, depending on whether one is working in symmetric or simply monoidal (non-$\Sigma$) categories.
 We preview the
results of this section:

\begin{thm}
 Let $\FF$ be a non--$\Sigma$ decomposition finite strict monoidal Feynman category. Set $\B=\Z\Mor(\F)$. Let $\mu=\otimes$, $\eta(1)=id_{\unit}$,  set $\Delta(\phi)=$ $\sum_{(\phi_0,\phi_1): \phi=\phi_0\circ \phi_1}\phi_0\otimes \phi_1$ and define $\eps(\phi)=1$ if $\phi=id_X$ for some $X$, else $\eps(\phi)=0$   then $(\B,\mu,\eta,\Delta,\eps)$ is a bi--algebra.

Let $\FF$ be a factorization finite Feynman category. Let $\B^{iso}$ be the free Abelian group on the isomorphisms classes of morphisms. Then there is a bi--algebra structure on $\B^{iso}$ given by $(\mu,\etaiso,\Deltaiso, \epsiso)$ where  $\mu$ is the tensor product on classes $\etaiso=[id_\unit]$, $\Deltaiso$ is the co--product induced on
co--invariants, and $\epsiso$ evaluates to $1$ precisely on the isomorphism classes  of identities.

If $\FF$ is almost connected then there is a bi--ideal $\I$ spanned by $[id_X]-[id_\unit]$
 and the quotient $\H=\B^{iso}/\I$ is connected and Hopf.
\end{thm}

For the notion of ``almost connected'' in this context,  see Definition \S\ref{almostconnecteddef}.

In the next section, we give alternative descriptions in terms of indecomposables \S\ref{indecomppar} and in the non--$\Sigma$ case we have a different construction for taking isomorphism classes using a quotient rather than co--invariants, cf.\ \S\ref{quotpar}.

\begin{prop} The relation of being in the same isomorphism class gives rise to a co--ideal $\C$ spanned by $f-g$ for any two  morphisms that are isomorphic in the arrow category.
In the non-$\Sigma$ case $\Bquot:=\B/\C\otimes \Q$ is a bi--algebra for a normalized $\epsquot$.
If $\FF$ is almost--connected then there is an ideal $\J$ and the quotient $\B^{quot}/\J$ is connected and a Hopf algebra over $\Q$ in general.
\end{prop}
Here the ideal $\J$ is spanned by $|Aut(X)||Iso(X)|id_X-|Aut (Y)||Iso (Y)|$  $id_Y$, where $|Aut (X)|$ is the cardinality of the automorphism group and $|Iso (X)|$ is the number of objects isomorphic to $X$. Both are finite if $\F$ is decomposition finite.
If $\F$ is skeletal the $|Iso(X)|=1$, and  if $\V$ is furthermore discrete, the ideal is simply $(id_X-id_Y)$. This is the case for non--sigma co--operads, in which case the two constructions coincide.
  For the symmetric case, it is possible to twist the co--multiplication in certain cases, so that the bi--algebra equation holds, see Theorem \ref{sumthm} for a summary.

 In order to recover the previous cases, one has to use several constructions defined in \cite[\S3, \S4]{feynman}.
 This is done in \S\ref{constexpar}.
  In particular, case I corresponds to the Feynman level category $\FF^+$ and its relation to enriched Feynman categories,  see \cite[\S3.6,\S4]{feynman} and \cite{Frep} for more details, applied to the Feynman category of surjections $\Surj$, that is the Feynman categories $\Surj_\O$, where $\O$ is an operad.
 The  generalization comes from the nc--construction \cite[\S3.2]{feynman} applied to the Feynman category for operads $\operads$ and a $B_+$ operator as given in \cite[Example 3.5.2]{feynman}.
 The construction of simplicial strings is captured by the nc--construction applied to the Feynman category $\Delta_{*,*}$ together with a decoration, that is the construction of $\FF_{dec\O}$, see \cite{decorated} and \cite[\S3.3]{feynman}, see \S\ref{constexpar} in particular \S\ref{enrichedpar} and \ref{anglepar}. Finally, universal operations \ref{universalpar} explain the amputation mechanism.

 We will begin by considering algebra and co--algebra structures for morphisms and isomorphisms classes of morphisms. We then introduce the notion of a Feynman category in the symmetric and non-$\Sigma$ version. Thus allows us to prove the bi--algebra structures under standard assumptions.
Afterwards, we turn to the Hopf algebras and functoriality.

\subsection{Algebra and co-algebra structures for morphisms}
Given a category $\F$ let $\B=\Z [Mor(\F)]\subset Hom(Mor(\F),\Z)$ be the free Abelian group on the morphisms of $\F$.

 \subsubsection{Isomorphism classes}
\label{isopar}
 Set $\B^{iso}=\B/ \sim$ where $\sim$ is the equivalence relation on morphisms  given by isomorphisms in $(\F\downarrow \F)$. In particular,
the equivalence relation $\sim$, which exists on any category, means that for given $f$ and $g$: $f\sim g$ if there is a commutative diagram with isomorphisms as vertical morphisms.
$$
\xymatrix{X\ar[r]^{f}\ar[d]^{\simeq}_{\sigma}&Y\ar[d]_{\simeq}^{ \sigma'}\\
X'\ar[r]^{g}&Y'\\
}
$$
i.e.: $f=\sigma^{\prime -1}\circ g\circ \sigma$.

 $\B^{iso}$ is the free Abelian group on isomorphism classes.
Fixing a skeleton $\F^{sk}$ of $\F$, $\B^{iso}=\Z[\amalg_{X,Y\in Obj(\F^{sk})} {}_{Aut(Y)}Hom(X,Y)_{Aut(X)}]$, that is the free Abelian group of the co--invariants of the left $Aut(Y)$ and right $Aut(X)$ action of the Hom sets of $\F^{sk}$. In general $\B^{iso}(\F)\simeq\B^{iso}(\F^{sk})$.

\begin{rmk}
The morphisms of $\F$ together with these isomorphisms are also precisely the groupoid of vertices $\V'$ of the iterated Feynman category $\FF'$, cf.\ \cite[\S3.4]{feynman}.
\end{rmk}

\begin{lem}
\label{isolem}
$id_X\sim g$ if and only if $g:X'\to Y'$ is an isomorphism and $X\simeq X' \simeq Y'$.
\end{lem}
\qed

\subsubsection{Algebra of morphisms of a  (strict) monoidal category}
We refer to \cite{Kassel} for details on monoidal categories. An introduction can be found in \cite{Frep}.
\begin{prop}
Let $(\F,\otimes)$ be a strict monoidal category. Then $\B$ is a unital algebra with multiplication $\mu(\phi,\psi)=\phi\otimes \psi$ and unit $1=id_\unit$.

If $(\F,\otimes)$ is a monoidal category then $\B^{iso}$ is a unital algebra with multiplication $\mu([\phi],[\psi])=[\phi \otimes \psi]$ and unit $1=[id_\unit]$.
If $(\F,\otimes)$ is symmetric monoidal then  $\B^{iso}$ is a commutative unital algebra.

\end{prop}

\begin{proof}
Recall that strict monoidal  means that in  the unit constraints and associativity constraints are identities.  Thus  $X\otimes(Y\otimes Z)=(X\otimes Y)\otimes Z$ which guarantees the associativity $(\phi_1\otimes\phi_2)\otimes\phi_3 = \phi_1\otimes(\phi_2\otimes\phi_3)$. Likewise
$X\otimes \unit=X=\unit\otimes X$ shows that indeed $\phi\otimes id_\unit=\phi=\phi\otimes id_\unit$.

The product is well defined on isomorphism classes, since if $\phi'\simeq \phi,\psi'\simeq\phi$ then $\phi'=\sigma^{-1}\phi\sigma'$ and $\psi'=\tau^{-1}\psi\tau'$ for isomorphisms, $\sigma,\sigma',\tau,\tau'$ and $\phi'\otimes\psi'=(\sigma^{-1}\otimes \tau^{-1})(\phi\otimes \psi)(\sigma'\otimes \tau')$, so that $[\phi\otimes \psi]=[\phi'\otimes\psi']$.
  Without the assumption of strictness,  if $\phi_i:X_i\to Y_i, i=1,2,3$ we have $(\phi_1\otimes\phi_2)\otimes\phi_3 =A(\phi_1\otimes(\phi_2\otimes\phi_3))$ in $\B$ where $A$ is given by pre- and post-composing with associativity isomorphisms $a_{X_1,X_2,X_3}$ and  $a_{Y_1,Y_2,Y_3}^{-1}$. Thus when one passes to isomorphism classes, the algebra structure is strict. In the same way, the unit constraints provide the isomorphism, which make the unit strict on $\B^{iso}$.
If $\F$ is symmetric, then the commutativity constraints $C_{X,Y}$ give the isomorphisms, proving that $[\phi] \otimes[\psi]=[\psi]\otimes[\phi]$.

\end{proof}

\begin{rmk}
The condition of being strict is not severe as by using Mac Lane's coherence theorem \cite{MacLane} one can pass from any monoidal category to an equivalent strict one.
We make this assumption, so that the algebra structure will be unital and associative rather than only weakly unital and weakly associative. After taking isomorphism classes the algebra structure is strict even if the monoidal category is not.
Note that if we are working in the enriched version $Hom(\unit,\unit)=K$ will play the role of a ground ring.
\end{rmk}

\subsubsection{The decomposition co--product}
Suppose that $\F$
is a decomposition finite category. This means
that  for each morphism $\phi$ of $\F$ the set $\{(\phi_0,\phi_1):\phi=\phi_0\circ\phi_1\}$
is finite.
Then, $\B$ carries a co--associative co--product given
by the dual of the composition. On
generators
it is given
by the sum over factorizations:
\begin{equation}
\label{facteq}
\xymatrix{
X\ar[rr]^{\phi}\ar[rd]_{\phi_1}&&Z\\
&Y\ar[ru]_{\phi_0}&
}
\end{equation}
that is
\begin{equation}\label{Delta-decomp}
\Delta(\phi)=\sum_{\{(\phi_0,\phi_1):\phi=\phi_0\circ\phi_1\}}\phi_0 \otimes \phi_1.
\end{equation}
where a morphism  $\phi$  is identified with  its characteristic morphisms $\delta_{\phi}$  that evaluates to $1$ on $\phi$ and zero on all other generators, as an element in $\Hom(\Mor(\F),\Z)$ $=\Z[\Mor(\F)]$

A co--unit is defined on the generators by:
\begin{equation}
\eps(\phi)=\begin{cases} 1 &\text{ if for some object $X$}: \phi=id_X\\
0& \text{else}
\end{cases}
\end{equation}
The co--unit axioms are readily verified and the co--associativity follows from the associativity of composition.

\begin{rmk}
One can enlarge
the setting to the situation in which
the sets of morphisms are graded and composition preserves the grading.
In this case, one only needs the condition: degree--wise composition finite.
This will be the case for any graded Feynman category \cite{feynman}. See also Example \cite[\ref{P1-semiinfex}]{HopfPart1}.
\end{rmk}

\begin{rmk}
We realized with hindsight that the co--product we constructed on indecomposables, guided by remarks from D.\ Kreimer given below \S\ref{indecomppar}, is equivalent to the co--product above. A little bibliographical sleuthing revealed that the
the co--product for any
finite decomposition category appeared already in
\cite{Moebiusguy} and was picked up later in \cite{JR}.
\end{rmk}

\subsubsection{Co--product of the identity morphisms}

\begin{rmk}

\begin{equation}
\label{Duniteq}
\Delta(id_X)=\sum_{(\phi_L,\phi_R): \phi_L\circ \phi_R=id_X}\phi_L\otimes \phi_R
\end{equation}
where  $id_X:X\stackrel{\phi_R}{\to}X'\stackrel{\phi_L}{\to}X$.
This mean that each $\phi_L$ has a right inverse $\phi_R$,
and each $\phi_R$ has a left inverse $\phi_L$. They do not have to be invertible in general.

\end{rmk}

\begin{cor}
In a decomposition finite category the automorphism groups $Aut(X)$ are finite for all objects $X$, as are the classes $Iso(X)$ of objects isomorphic to $X$.
\end{cor}

\begin{proof}
For each automorphism $\phi$ of $X$ and for each isomorphism $\phi:X\to X'$ there is a factorisation $id_X=\phi^{-1}\circ\phi$, and there are only finitely many such factorisations.
\end{proof}

\begin{lem}
\label{idfinitelem}
If $\F$ is decomposition finite, if the identity of an object has a factorization
$id_X:X\stackrel{\phi_R}{\to}X\stackrel{\phi_L}{\to}X$ then both $\phi_R$ and $\phi_L$ are invertible.
\end{lem}
\begin{proof}
Using the powers of $\phi_L$ and $\phi_R$, there are decompositions of $\phi_L=\phi_L^l\circ (\phi_R^l\circ \phi_L)$. Since $\F$ is decomposition finite, we have to have that $\phi_L^l=\phi_L^k$ for some $k>l$. Applying $\phi_R^l$ from the right, we see that $\phi_L^{k-l}=id_X$. That is $\phi_L$ is unipotent and hence an isomorphism.
\end{proof}

 \subsubsection{Co--algebra on isomorphisms classes}

The set $Hom(X,Y)$ has a natural action of $Aut(Y)\times Aut(X)$: $\phi\stackrel{\lambda_{\sigma_Y}\rho_{\sigma^{-1}_X}}{\longrightarrow} \sigma_Y\circ \phi\circ \sigma_X^{-1}$. We let $Aut(\phi)\subset Aut(Y)\times Aut(X)$ be the stabilizer group of $\phi$.
There is an action  of $Aut(Y)$ on
$Hom(Y,Z) \times  Hom(X,Y)$ given by $\bar d:(\sigma)(\phi_0,\phi_1)=(\phi_0\circ\sigma^{-1} ,\sigma\circ \phi_1)$, which leaves the composition map invariant:
$\phi_0\circ\phi_1=\phi_0\circ \sigma^{-1}\circ \sigma\circ \phi_1$.

There is also an action on decompositions which is a specialization of the actions of $I_{X,X'}=Iso(X,X')$, $I_{Y,Y'}=Iso(Y,Y')$ and
$I_{Z,Z'}=Iso(Z,Z')$ that maps $Hom(Y,Z) \times  Hom(X,Y)\to Hom(Y',Z') \times  Hom(X',Y'):$
 $(\phi_0,\phi_1)\mapsto (\phi_0',\phi_1')= (\sigma_Z\phi_0\sigma_Y^{-1},\sigma_Y\phi_1\sigma_X^{-1})$.

 There is an equivalence relation on factorizations $(\phi_0,\phi_1)$ of the type \eqref{facteq}  given by the action of $I_{Y,Y'}$, namely we set:
 $(\phi_0,\phi_1)\sim (\phi_0',\phi_1')$ if $\phi'_0=\phi_0\sigma_Y^{-1}$ and $\phi'_1=\sigma_Y\phi_1$.
For a given class $c$ under this equivalence choose a representative $c=[f]=[(\phi_0,\phi_1)]$  and consider the corresponding summand $\Delta_f$ of $\Delta$ together with  the $I_{X,X'},I_{Y,Y'}$
and $I_{Z,Z'}$ actions and co--invariants on this decomposition.

 \begin{equation}
\label{phicoprodeq}
\resizebox{\textwidth}{!} {
{
\xymatrix{
\phi\ar@{|->}[rr]^{\Delta_f}\ar@{|->}[d]\ar@{|->}@(dl,ul)@<-2pc>[dd]_{\lambda_{\sigma_Z}\rho_{\sigma_X^{-1}}}
\ar@{|->}[dr]^{\pi\circ p\circ \Delta_f}&&(\phi_0, \phi_1)\ar@{|->}[d]^p
\ar@{|->}@{|->}@(dr,ur)@<5pc>[dd]^{\lambda_{\sigma_Z}\rho_{\sigma_Y^{-1}}\otimes \lambda_{\sigma_Y}\rho_{\sigma_X^{-1}}}\\
[\phi]\ar@{|->}[r] &([\phi_0],[\phi_1])&[(\phi_0, \phi_1)]\ar@{|->}[l]_--{\pi}\\
\sigma_Z\phi\sigma_X^{-1}\ar@{|->}[rr]_{\Delta_{f'}}\ar@{|->}[u]
\ar@{|->}[ur]_{\pi\circ p\circ\Delta_{f'}}&&(\sigma_Z\phi_0\sigma_Y^{-1}, \sigma_Y\phi_1\sigma_X^{-1})\ar@{|->}[u]_p\\
}
}
}
\end{equation}
here $f=(\phi_0,\phi_1)$ is a factorization $f'=(\phi'_0,\phi'_1)$ is a different representative of the same class $[(\phi_0,\phi_1)]$ under the action of $\lambda\times \bar d\times \rho$ of
$I_{Z,Z'}\times I_{Y,Y'}\times I_{X,X'}$.

For simplicity assume that $\F$ is skeletal.  To shorten notation, we let $\F(X,Y)=Hom_\F(X,Y)$,
 and $A_X=Aut(X)$.

Fix a representative of the intermediate space $Y$ in its isomorphism class and a choice of representative decompositions $F$,
one for each class of the $\phi\in\F(X,Z)$ this
fixes  a   system of representatives obtained by conjugation $F'$ .

Assume that $\F$ is finite if for any morphism $\phi$ and fixed space $Y$.
Under this condition, the map $\Delta^{iso}_{X,Y,Z}$ in the diagram below is well defined due to the properties of co--limits and the finiteness assumption., we obtain a diagram of the type \cite[\eqref{P1-coprodsqeq}]{HopfPart1}.
 \begin{equation}
\label{coprodfeyeq}
\resizebox{\textwidth}{!} {
{
\xymatrix{
\F(X,Z)\ar[rr]^{\Delta_F}\ar[d]
\ar@(dl,ul)@<-2pc>[dd]_{\lambda_{\sigma_Z}\rho_{\sigma_X^{-1}}'}
\ar[dr]^{\pi\circ p \circ\Delta_F}&&\F(Y,Z)\times\F(X,Y)\ar[d]^p
\ar@(dr,ur)@<5pc>[dd]^{\lambda_{\sigma_Z}\rho_{\sigma_Y^{-1}}\otimes \lambda_{\sigma_Y}\rho_{\sigma_X^{-1}}}\\
{}_{A_Z}\F(X,Z)_{A_X} \ar[r]^-{\Delta^{iso}_{X,Y,Z}} &{}_{A_Z}\F(Y,Z)_{A_Y}\times{}_{A_Y}\F(X,Y)_{A_Z}&
{}_{A_Z}\F(Y,Z)\times_{A_Y}\F(X,Y)_{A_X}\ar[l]_--{\pi}\\
\F(X',Z')\ar[rr]^{\Delta_{F'}}\ar[u]\ar[ur]_{\pi\circ p\circ \Delta_{F'}}&&\F(Y',Z')\times\F(X',Y')\ar[u]_p\\
}
}
}
\end{equation}
Assume further that for any $\phi$ only finitely many isomorphism classes $Y$ appear in the decompositions of $\phi$.
If both finiteness assumptions are satisfied, we call $\F$ {\em factorization finite}.
Fixing a representative $\phi$ and summing over isomorphism classes of $Y$,
 we then obtain  $\Delta^{iso}([\phi])$ if $\F$ is factorization finite.
\begin{lem}
If $\F^{sk}$ is decomposition finite, then $\F$ is factorization finite.\qed
\end{lem}
 \begin{prop}
\label{Bisoprop}
If $\F$ is factorization finite, then
the decomposition co--product $\Delta$ and the co--unit $\eps$ descend to a co--product
$\Deltaiso$ and co--unit $\epsiso$
on $\Biso$
as co--invariants and
$(\Biso,\Deltaiso,\epsiso)$ is a co--unital co--algebra.
 \end{prop}

 \begin{proof}

Co--associativity  follows readily. The co--unit $\eps$ is invariant under the left and right actions by automorphisms and descends to $\eps^{iso}$
\begin{equation}
\eps^{iso}([\phi])=\begin{cases}1&\text{ if } [\phi]=[id_X]\\
0 & \text{else}\\
\end{cases}
\end{equation}

\end{proof}
\subsubsection{Direct formula for $\Delta^{iso}$}
There is a direct way to describe the co--product, by analyzing the image of $\Deltaiso_{X,Y,Z}$.

We call a pair $(\phi_0,\phi_1)$ of  morphisms {\em weakly composable}, if there is an isomorphism $\sigma$, such that $\phi_1\circ\sigma\circ\phi_0$ is composable. A weak decomposition of a morphism $\phi$ is
a  pair of morphisms  $(\phi_0,\phi_1)$ for which there exist isomorphisms $\sigma,\sigma',\sigma''$  such that
$\phi=\sigma\circ \phi_1\circ \sigma' \circ \phi_0 \circ \sigma''$.
In particular, a decomposition $(\phi_0,\phi_1)$ is weakly composable.
We introduce an equivalence relation on weakly composable morphisms, which says that $(\phi_0,\phi_1)\sim (\psi_0,\psi_1)$ if they are weak decompositions of the same morphism. An equivalence class of weak decompositions will be called a decomposition channel. The notation will be $([\phi_0],[\phi_1])$. In this notation, we have that
$\pi([(\phi_0,\phi_1)])=([\phi_0],[\phi_1])$ and the image of $\Deltaiso_{X,Y,Z}$  are precisely the decomposition channels.
 These may, however, appear with multiplicities.

\begin{prop}
For an element/equivalence class $[\phi]\in \Biso$.

\begin{equation}
\label{Deltaiso}
 \Delta^{iso}([\phi])=\sum_{[(\phi_0,\phi_1)]}[\phi_0] \otimes [\phi_1]
\end{equation}
where the sum is over a complete system  of decompositions for a fixed representative $\phi$.
\qed
\end{prop}

\begin{rmk}

Notice that there are many ways in which two weakly composable morphisms are composable and hence may yield different compositions. Thus the right hand side may have terms that can be collected together.
To obtain a representative one has to again  rigidify by ``enumerating everything'', as in was done
in \cite[Remark \ref{P1-rigidrmk}]{HopfPart1}.

This is also  the reason that the composition dual to decomposition in $\B^{iso}$ in \cite[\eqref{P1-symopcompeq}]{HopfPart1} is on invariants. A similar phenomenon is known in physics, when composing graphs \cite{KreimerGauge}. For Feynman categories of graph type such aspect have been previously discussed in \cite[\S2.1]{feynman}.
\end{rmk}

\begin{ex}
In particular in
 \S\ref{graphexpar} the construction is made concrete for the Connes-Kreimer Hopf algebra of graphs. An isomorphism class of a morphism is fixed by the ghost graph. The ghost graph of $\phi_1$ is naturally a subgraph of the ghost graph of $\phi$.
  The action on the intermediate space allows to ``forget'' the target of $\phi$ up to isomorphism and identify the ghost graph of $\phi_0$ with the quotient graph.
In the co--product one forgets the structure of being a subgraph, which is also what leads to multiplicities, cf.\ Example \ref{subgraphex}.
\end{ex}
\subsubsection{Bi--algebra structure conditions}
By the above, in any  strict monoidal category with finite decomposition $\B$ has a unital product and a co--unital co--product. However, the compatibility axioms of a bi--algebra {\em do not hold in general}.
 For instance, one needs to check
$$\Delta\circ\mu=(\mu\otimes \mu)\circ \pi_{2,3}\circ (\Delta\otimes \Delta)$$
where $\pi_{2,3}$ switches the 2nd and 3rd tensor factors.
Each side of the equation is represented by a sum over diagrams.

For  $\Delta\circ\mu$ the sum is over diagrams of the type
\begin{equation}
\label{deltamueq}
\xymatrix{
X\otimes X'\ar[rr]^{\Phi=\phi\otimes \psi}\ar[dr]^{\Phi_1}&&Z\otimes Z'\\
&Y\ar[ur]^{\Phi_0}&\\
}
\end{equation}
where $\Phi=\Phi_0\circ\Phi_1$.

When considering $(\mu\otimes \mu)\circ\pi_{23}\circ (\Delta\otimes \Delta)$
the diagrams are of the type
\begin{equation}
\label{mudeltaeq}
\xymatrix{
X\otimes X'\ar[rr]^{\phi\otimes \psi}\ar[dr]^{\phi_0\otimes \psi_0}&&Z\otimes Z'\\
&Y\otimes Y'\ar[ur]^{\phi_1\otimes \psi_1}&\\
}
\end{equation}
where $\phi=\phi_0\circ\phi_1$
and
$\psi=\psi_0\circ\psi_1$. And there is no reason for there to be a bijection of such diagrams.

The compatibility  {\em does hold} when dealing with  Feynman categories; as we now show.
\subsection{Feynman categories and bi--algebra structures}
Here we give the definition of the various Feynman categories and prove the theorems previewed above. Examples can be found in \S\ref{constexpar}.
\subsubsection{Definition of  a Feynman category}
\label{feydefpar}
Consider the following data:

\begin{enumerate}
\item $\V$ a groupoid, with $\V^{\otimes}$ the free symmetric monoidal category on $\V$.
\item $\F$ a  symmetric monoidal category, with monoidal structure
denoted by $\otimes$.
\item $\imath: \V\to \F$ a functor, which by freeness extends to a monoidal functor $\iota^\otimes$ on $\V^{\otimes}$,
\end{enumerate}
$$\xymatrix@C+2em{
\V\ar[d]_{\jmath}\ar^{\imath}[r]&\F\\
\V^{\otimes}\ar_{\imath^\otimes}[ur]\ar[r]&Iso(\F)\ar@{^{(}->}[u]
}$$
where $Iso(\F)$ is the maximal (symmetric monoidal) sub--groupoid of $\F$.

Consider the comma categories $(\F\downarrow \F)$ and $(\F\downarrow \V)$ defined by  $(id_{\F},id_{\F})$ and $(id_{\F},\imath)$.
\begin{df}
A triple $\FF=(\V,\F,\imath)$ as above is called a \emph{Feynman category}
if

\begin{enumerate}
\renewcommand{\theenumi}{\roman{enumi}}

\item \label{objectcond}
 $\imath^{\otimes}$ induces an equivalence of symmetric monoidal groupoids between $\V^{\otimes}$ and $Iso(\F)$.

\item \label{morcond} $\imath$ and $\imath^{\otimes}$ induce an equivalence of symmetric monoidal groupoids\\ $Iso(\F\downarrow \V)^{\otimes}$ and  $Iso(\F\downarrow \F)$.

\item For any object $\ast_v$ of $\asts$, $(\clusters\downarrow\ast_v)$ is
essentially small.

\end{enumerate}
\end{df}
The first condition says that $\V$ knows all about the isomorphisms. The third condition is technical to guarantee that certain colimits exist. The second condition,
also called the \emph{hereditary condition}, is the key condition. It can be understood as follows: any morphism in $\F$ is isomorphic, up to unique isomorphism, to a tensor product of basic morphisms, which are those in $(\F\downarrow \V)$ (aka.\ one-comma generators). Viz.:
\begin{enumerate}
\item
For any morphism
$\phi: X\to X'$, if we choose $X'\simeq \bigotimes_{v\in I} \imath(\ast_v)$ by (i),  there are $X_v$ and $\phi_v:X_v\to\iota(\ast_v)$ in $\F$ such that $\phi$ is isomorphic to $\bigotimes_{v\in I}\phi_v$,
\begin{equation}
\label{morphdecompeq}
\xymatrix
{
X \ar[rr]^{\phi}\ar[d]_{\simeq}&& X'\ar[d]^{\simeq} \\
 \bigotimes_{v\in I} X_v\ar[rr]^{\bigotimes_{v\in I}\phi_{v}}&&\bigotimes_{v\in I} \imath(\ast_v).
}
\end{equation}
\item
For any two such decompositions $\bigotimes_{v\in I} \phi_v$ and $\bigotimes_{v'\in I'}\phi'_{v'}$
there is a bijection
$\psi:I\to I'$ and isomorphisms
$\sigma_v:X_v\to X'_{\psi(v)}$
such that $P\circ\bigotimes_v\phi_v=\bigotimes_v(\phi'_{\psi(v)}\circ\sigma_v)$
where $P$ is the permutation corresponding to $\psi$.

\item These are the only isomorphisms between morphisms.
\end{enumerate}

\subsubsection{Non-symmetric version}
Now let $(\V,\F,\imath)$ be as above with the exception that $\F$ is only a monoidal category, $\V^{\otimes}$ the free monoidal category, and $\imath^{\otimes}$ is the corresponding morphism of monoidal groupoids.

\begin{df}
A triple $\FF=(\V,\F,\imath)$ as above is called a \emph{non-\SSigma{} Feynman category} if

\begin{enumerate}
\renewcommand{\theenumi}{\roman{enumi}}

\item \label{nonsymobjectcond}
 $\imath^{\otimes}$ induces an equivalence of  monoidal groupoids between $\V^{\otimes}$ and $Iso(\F)$.

\item \label{nonsymmorcond} $\imath$ and $\imath^{\otimes}$ induce an equivalence of  monoidal groupoids $Iso(\F\downarrow \V)^{\otimes}$ and  $Iso(\F\downarrow \F)$.

\item For any object $\ast_v$ in $\asts$, $(\clusters\downarrow\ast_v)$ is
essentially small.

\end{enumerate}

\end{df}

\subsubsection{Strict Feynman categories}
We call a Feynman category {\em strict} if the monoidal structure on $\F$ is strict, $\iota$ is an inclusion, and $\V^\otimes=Iso(\F)$ where we insist on using the strict free monoidal category, see e.g.\ \cite{matrix} for a thorough discussion. Up to equivalence in $\V$, $\F$ and  in $\FF$ this can always be achieved.

In the strict case, one can assume that the right vertical arrow in \eqref{morphdecompeq} is an identity,  thus for any morphism $\phi$ we have $\phi=\bigotimes \phi_v\circ P$. Here  $\phi_v:X_v\to\iota(*_v)$ and $P$ is an isomorphism in $\V^\otimes=Iso(\F)$, which we can fix to be simply a permutation $P:X\stackrel{\sim}\to \bigotimes_v X_v$ after absorbing possible isomorphisms $\tau_v:X_v\stackrel{\sim}{\to}X'_v$ into the $\phi_v$ by pre-composition.  The permutation is by definition trivial in the non-\SSigma\ case.

\subsubsection{Bi--algebra structure for non--$\Sigma$ Feynman categories}

\begin{lem}
\label{deltaunitlem}
In a strict decomposition finite Feynman category  $\Delta(id_{\unit})$ is group--like, i.e.: $\Delta(id_\unit)=id_\unit\otimes id_\unit$
\end{lem}

\begin{proof} By
\eqref{Duniteq} $\Delta(id_\unit)=id_{\unit}\otimes id_{\unit}+ \sum \phi_L\otimes \phi_R$ with $\phi_R:\unit \to X$ and $\phi_L:X\to \unit$ with $\phi_L\circ \phi_R=id_\unit$.
It follows from axiom (ii) that there are only morphisms $X\to \unit$ for $X=\unit$ and thus $\phi_L,\phi_R:\unit \to \unit$. By Lemma \ref{idfinitelem} they have to be isomorphisms,
 $\unit\to \unit$ and thus by axiom (i) $\phi_L=\phi_R=id_\unit$. Hence
 $\Delta(id_{\unit})$ only has one summand corresponding to $id_\unit\otimes id_\unit$.
\end{proof}

\begin{thm}\label{f-to-b-nonsym}
For any  strictly monoidal, finite decomposition,
non-\SSigma{}
Feynman category $\FF$
the tuple  $(\B,\otimes,\Delta,\eps,\eta)$ defines a  bi--algebra over $\mathbb{\Z}$.
\end{thm}

\begin{proof}
We check the compatibility axioms:

(1) The co--unit is multiplicative  $\eps(\phi\otimes \psi)=\eps(\phi)\eps(\psi)$.
First, $id_X\otimes id_Y=id_{X\otimes Y}$, since $\F$ is strict monoidal. Because of axiom (i) this is then the unique decomposition of $id_{X\otimes Y}$, and hence both sides are either zero or $\phi=n\,id_X$ and $\psi=m\,id_Y$,
in which case both sides equal to $nm$.

(2) The unit is co--multiplicative: by Lemma \ref{deltaunitlem}, $\Delta(id_\unit)=\id_\unit\otimes id_\unit$, so $\Delta\circ \eta=\eta\otimes \eta$.

(3) Compatibility of unit and co--unit: $\eps(1)=\eps(id_\unit)=1$ and hence $\eps\circ\eta=id$.

(4) Bi--algebra equation:
In order to prove that $\Delta$ is an algebra morphism, we consider the two sums over the diagrams \eqref{deltamueq}  and \eqref{mudeltaeq} above and show that they
 coincide. First, it is clear that all diagrams of the second type appear in the first sum. Vice--versa, given a diagram of the first type, we know that $Y\simeq \hat Y\otimes \hat Y'$,
since $\Phi_1$ has to factor by axiom (ii) and the Feynman category is strict. Then again by axiom (ii) $\Phi_0$ must factor.
We see that we obtain a diagram:
\begin{equation}
\label{deltamuisoeq}
\xymatrix{
\hat X\otimes \hat X'\ar[r]^{\simeq}_{\sigma=\sigma_1\otimes \sigma_2}\ar[dddrr]_{\hat\phi_0\otimes\hat\psi_0}&X\otimes X'\ar[rr]^{\Phi=\phi\otimes \psi}\ar[dr]^{\Phi_0}&&Z\otimes Z'\\
&&Y=Y'\otimes Y''
\ar[ur]^{\Phi_1}\ar[dd]_{\sigma'=\sigma'_1\otimes \sigma'_2 }^{\simeq}&\\
\\
&&\hat Y\otimes \hat Y'\ar[uuur]_{\hat\phi_1\otimes\hat\psi_1}&\\
}
\end{equation}
Now since the Feynman category is strict and non-symmetric, the two isomorphisms also decompose as
$\sigma=\sigma_1 \otimes \sigma_2,$ and $\sigma'=\sigma'_1\otimes \sigma'_2$, for a splitting  $Y=Y'\otimes Y''$ so that $\Phi_0=\sigma_1^{\prime -1}\circ \hat\phi_0\circ\sigma^{ -1}_1\otimes \sigma_2^{\prime -1}\circ \hat\psi_0\circ\sigma^{-1}_2$ and
$\Phi_1=\hat\phi_1\circ\sigma'_1\otimes \hat\psi_2\circ\sigma'_2:Y=Y'\otimes Y''\to Z\otimes Z'$ and one obtains that both diagram sums
agree.

\end{proof}

Examples are discussed in detail in \S\ref{constexpar}.

\subsubsection{Bi--algebra structure on $\B^{iso}$ for a Feynman category}

\begin{thm}
\label{Bisothm}
Given a factorization finite Feynman category $\FF$,
$(\B^{iso}$, $\otimes,\eta,\Delta^{iso},\eps^{iso})$ is a bi--algebra, both in the symmetric and the non--$\Sigma$ case.
\end{thm}

\begin{proof}
We can retrace the steps in  the proof of Theorem \ref{f-to-b-nonsym}, up until the decomposition of $\sigma$ and $\sigma'$ into tensor products. Even without this assumption,
the diagram \eqref{deltamuisoeq} clearly shows that $[(\Phi_0,\Phi_1)]=
[(\hat \phi_0\otimes \hat \psi_0),(\hat \phi_1\otimes \hat \psi_1)]$, so that there is indeed a bijection of the equivalence classes and hence the bi--algebra equation holds.
The compatibilities for the
unit and co--unit are simple computations along the lines of the proof of Theorem \ref{f-to-b-nonsym}.
\end{proof}

\begin{rmk} \mbox{}
\begin{enumerate}
\item If $\V$ is discrete in the non--$\Sigma$ case, then $\B^{iso}=\B$.
\item In the symmetric case, there is a difference in the count of diagrams in $\B$, which is controlled by the action $\bar d$ and the symmetric group actions. This is made precise in \S\ref{alternatepar}.
\item We have so far considered Feynman categories over $\Set$. The theorems also hold in the case of enriched Feynman categories such as $\FF_\O$, see (\S\ref{hyppar}) and \cite{Frep} for more details.
 The enrichment can be over a tensor category $\mathcal E$ which has a faithful functor to $\Ab$, e.g. $\kVect$. In this case one should work over the ring $K=Hom_{\F}(\unit,\unit)$, see \S\ref{hyppar} and \cite[\S4]{feynman} for more details.
\item
This co--product actually corresponds to the category $\FF'_{\V'}$ of universal operations \cite[\S6]{feynman}. Here all channels with $[\phi_1]=[\psi]$ corresponds to the class of morphisms in $Hom_{\F'}(\phi,\psi)$. That means that each class of such a morphism under isomorphism corresponds to a channel and contributes a term to the sum.  The associativity of the co--product is then just the associativity of the composition in $\FF'_\V$.
\end{enumerate}
\end{rmk}

\subsection{Co--module structure}
Let $\B_1=\B_1=\Z[Ob(\imath^\otimes \downarrow \imath)]$  be the free Abelian group on the basic morphisms ---see also \S\ref{indecomppar} below.
\begin{prop}
For a decomposition finite (non--$\Sigma$) Feynman category the set of basic morphisms, that is objects of $(\F\downarrow\V)$ form a co--module, viz.\ $\rho:=\Delta|_{\B_1}:\B_1\to \B_1\otimes \B$ is a co--module for $\Delta$.
For a factorization finite Feynman category the analogous statement holds true for $\Biso_1=\B_1/\sim$ and $\Biso$.
\end{prop}
\begin{proof}
If $\phi\in Ob(\imath^\otimes \downarrow \imath) $,
then $\Delta(\phi)\in \B_1\otimes \B$, since the target of $\phi_0$, which is the target of $\phi$, is an object of $\B$ for any factorization $\phi=\phi_0\circ\phi_1$.
\end{proof}
See \S\ref{indecomppar}  for more details on this point of view.
\subsubsection{$B_+$ operator}
\label{B+par}
The definition of $B_+$ operators in general is quite involved, see \cite[\S3.2.1]{feynman}.
Functorially, such an operator gives a morphism $\B\to \B_1$.
Without going into a full analysis, which will be done in  \cite{B+paper}. At this points, we simply make the following definition.
\begin{df} A $B_+$ multiplication or $B_+$--operator for $\FF$ is a morphism $B_+:\B\to \B_1$ such that
$\Delta_1:=(id\otimes B_+)\circ \Delta:\B_1\to \B_1\otimes \B_1$ and $\mu_1:=\B_+\circ \mu_\otimes$ together with the unit and co--unit, yield
a unital, co--unital bi--algebra structure on $\B_1$.
\end{df}
The multiplication for a co--operad with multiplication of \cite[\S\ref{P1-gencooppar}]{HopfPart1} is an example, see  \S\ref{B+parII}.

This description also links the $B_+$ operator to Hochschild homology, as considered in \cite{CK}.

\subsection{Opposite Feynman category yields the co-opposite bi--algebra}
\label{oppositepar}
Notice that usually the opposite category of a Feynman category is not a Feynman category, but it still defines a bi--algebra. Namely,
the constructions above just yield the co-opposite bi--algebra structure $\B^{co-op}$.
This means, the multiplication is unchanged but the co--multiplication is switched. That is $\Delta(\phi^{op})=\sum_{\phi_1\circ \phi_0=\phi}\phi_1^{op}\otimes \phi_0^{op}$.

 The same holds for quotient and Hopf algebra structures discussed below, i.e.\
  $\H$ is replaced by $\H^{co-op}$.

\subsection{Hopf algebras from Feynman categories}
\label{Hopfpar}
The above bi--algebras are usually not connected.
There are several obstructions. Each identity morphism of an object $X$ potentially gives a group--like element. Additionally, unless $\V$ is discrete, there are isomorphisms which are not co--nilpotent. At a deeper level, they can prevent the identities of the different $X$ from being group-like elements and hence keep the putative Hopf quotient form being connected. There are several other obstructions to co--nilpotence, which one has to grapple with in the general case.  We will now formalize this and give checkable criteria that are met by the main examples.

\subsubsection{Almost group--like identities and the putative Hopf quotient}
A Feynman category has {\em almost group--like} identities if each of the $\phi_L$ and hence each of the $\phi_R$ appearing in a co--product of any $id_X$  \eqref{Duniteq}  is an isomorphism.

\begin{ex}
A counter--example, that is a Feynman category that does not have group--like identities, is $\FinSet_<$ or its skeleton $\Delta_+$.  In this case,
the category is also not decomposition finite. The reason is that each $id:\underline n\to \underline n$ factors as $\underline n\hookrightarrow \underline m\twoheadrightarrow \underline n$ for all $m\geq n$. Both $\Surj$ and $\Inj$ as well as all the graphical examples have group--like identities.
For these notations and a detailed analysis of the examples, see \S\ref{constexpar}, especially \S\ref{simpfeypar}, Table \ref{table1} and Table \ref{table2}.
\end{ex}

The assumption of almost group--like identities is, however, very natural and is often automatic.
The example above is symptomatic.

\begin{lem}
 If $\FF$ is  decomposition finite and has almost group--like identities then both in the symmetric and non--$\Sigma$ case:
 \begin{enumerate}
 \item \label{grouplike} The classes $[id_X]$ are group--like in $\Biso$ that is $\Deltaiso([\id_X])=[id_X]\otimes [id_X]$.
\item The two--sided ideal $\I=\langle [id_X]-[id_Y]\rangle$ in $\Biso$ is also a co--ideal.
\end{enumerate}
\end{lem}
\begin{proof}
Under the assumption of almost group--like identities:
\begin{equation}
\label{coprodcalceq}
    \Delta(id_X)=\sum_{X', \sigma \in Iso(\F)(X,X')}\sigma \otimes \sigma^{-1}
\end{equation}
thus there is only one decomposition channel with multiplicity $1$,
since $[(\sigma,\sigma^{-1})]=[(id_X\otimes id_X)]$.

Using \eqref{grouplike}: $\Delta^{iso}([id_X]-[id_Y])=[id_X]\otimes [id_X]-[id_Y]\otimes [id_Y]=
([id_X]-[id_Y])\otimes [id_X]+ [id_Y]\otimes([id_X]-[id_Y])$
and $\epsiso([id_X]-[id_Y])=1-1=0$,
so that $\I$ is a co--ideal.
\end{proof}

\begin{df}
 If $\FF$ is  factorization finite and has almost group--like identities then both in the symmetric and non--$\Sigma$ case, we set $\H=\B^{iso}/\I$.
 We call $\FF$ Hopf, if it satisfies the stated conditions and the bi--algebra $\H$ has an antipode, i.e.\ $\H$ is a Hopf algebra.
\end{df}

\begin{thm}
A  Hopf  Feynman category  yields a Hopf algebra
$\H:=\B^{iso}/\I$, both in the symmetric and non--symmetric case. $\H$  is commutative in the symmetric case and not necessarily commutative in the non--symmetric case.
\end{thm}

\begin{proof}
The only new claim is the commutativity in the symmetric case. This is due to the fact that the commutativity constraints are isomorphisms and these become identities already in $\B^{iso}$.
\end{proof}

In general, the existence of  an antipode is complicated. We do know that for graded connected bi--algebras an antipode exists. In terms of Feynman categories this situation can be achieved by looking at definite Feynman categories.

\subsubsection{Graded Feynman categories}
\label{reflengthpar}
One thing that helps to check  connectedness and co--nilpotence is a grading. Each Feynman category has a {\em native length} for objects and morphisms.
Due to condition (i)  for a Feynman category every object $X$ has a unique  length $|X|$ given by the tensor word length of any object of $\V^{\otimes}$ representing it.
We  define the length decrease (or just length) of a morphism $\phi:X\to Y$ as $|\phi|=|X|-|Y|$. This is additive under composition and tensor. Isomorphic objects have the same length, so isomorphisms have length zero. Morphisms can also increase length, that is, have negative length (decrease), as one may have a morphism $\unit\to \imath(\ast)$ which increases length by one and hence has  length $-1$, see \cite[Remark 1.4.3]{feynman}.

An {\em integer degree function} for a Feynman category is a function $\deg:Mor(\F)\to \Z$ which is additive under composition and tensor product:
 $\deg(\phi\circ\psi)=\deg(\phi\otimes \psi)=\deg(\phi)+\deg(\phi)$, with the additional condition that isomorphisms have degree $0$. Thus the length function $|\,.\;|$ is a degree function.

A graded Feynman category with an integer degree function is {\em non-negative} or {\em non-positive} if all morphisms have non-negative or  non-positive degree respectively.
We call a graded Feynman category {\em definite} if it is non--positive or non--negative. Of course by changing $\deg$ to $-\deg$, one can change from non--positive to non--negative. One has extra structure in the definite case, which allows one to define the condition of almost connected, see Definition \ref{almostconnecteddef} and Lemma \ref{definitelem}.
All the main examples are definite.

 \begin{rmk}
 In \cite[Definition 7.2.1]{feynman}, similar notions were introduced:  a {\em degree function} has the two additional conditions: (1) to have positive values and (2) all the morphisms are generated by degree $0$ and $1$ morphisms by composition and tensor product.
  It is called a {\em proper}, if all morphisms of degree $0$ are isomorphisms.
 Many, but not all, Feynman categories have a proper degree function.  Proper implies definite.
\end{rmk}

 \begin{ex}\mbox{}
 \begin{enumerate}
\item In the set based examples: for $\FFinSet$, $|\phi|$ is a proper integer degree function. On $\Surj$, $|\phi|$ is a proper degree function and on $\Inj$, $\deg(\phi)=-|\phi|$ is a proper degree function.
 \item

In the case of graphs of higher genus ($b_1>0$), loop contractions are of native length $0$. It is more natural, to have a different grading, in which both loop and edge contractions have degree $1$ and mergers have degree $0$. This makes the
relations homogeneous, cf.\ \cite[\S5.1]{feynman}. For $\Agg^{ctd}$ this is a proper degree function. The degree of a morphism $\phi$ is the number of edges of $\gh(\phi)$.

In most practical examples, mergers are excluded, making life simple. This  includes the Feynman categories for operads, colored operads, modular operads, etc., however this excludes PROPs and other ``disconnected'' types.
In all of $\Agg$ it actually suffices to have the generators (a) isomorphisms, (c) simple loop contractions and (d) mergers. In this setting a proper degree function is given by assigning isomorphisms degree $0$, and loop contractions and mergers degree $1$.

\item
All three main examples are definite. The Feynman categories $\Surj_\O$ for algebras over operads, see \S\ref{hyppar} and \cite[\S4]{feynman} are precisely non--negative, if there is no $\O(0)$;  the length of elements of $\O(n)$ is $n-1$. They are proper if $\O(1)=\unit$ is reduced.  In the split unital case, $\O(1)=\unit \oplus \O^{red}(1)$, if $\O^{red}(1)$ has no invertible morphisms, they are have group--like identities.
Surjections are also non--negative. Dually, regarding only injections is an example of a non--positive Feynman category. All graph examples ---without extra morphisms, see \cite{feynman}--- are also non--negative.

 \end{enumerate}
 \end{ex}

\begin{prop}
\label{degreeprop}
Given a factorization $id_X:X\stackrel{\phi_R}{\to}Y\stackrel{\phi_L}{\to}X$ it follows that
\begin{enumerate}
\item For any integer degree function $\deg(\phi_R)=-\deg(\phi_L)$.
\item $|\phi_R|\leq 0$ and $|\phi_L|\geq 0$.
\item If $\FF$ has a definite integer degree function then $\deg(\phi_R)=\deg(\phi_L)=0$. I.e.\ any morphism with a left or right inverse has degree $0$.
\item If $\FF$ is definite and if  the only morphisms of $\FF$ with length $0$, which have  one--sided inverses are isomorphisms, then $\FF$ has group--like identities.
\item \label{properitem} If $\FF$ has a  proper degree function then $\FF$ has almost group--like identities.
\item  If $\FF$ is decomposition finite, then  the identity of any object $X$ does not have  a factorization
$id_X:X\stackrel{\phi_R}{\to}X\otimes Y\stackrel{\phi_L}{\to}X$ with $|\phi_R|<0$.
\end{enumerate}
\end{prop}

\begin{proof} (1): $\deg(\phi_L)+\deg(\phi_R)=\deg(\phi_L\circ \phi_R)=\deg(\id_X)=0$.
(2): Decomposing the morphisms for $X=\bigotimes_v\ast_v$ according to (ii) we end up with  sequences
$$\ast_v \stackrel{\phi_{R,v}}{\to}Y_v\stackrel{\phi_{L,v}}{\to}\ast_v$$ with $\phi_{L,v}\circ \phi_{R,v}=id_{\ast_v}$. This follows from decomposing $\phi_L$ and $\phi_R$ and then comparing to the decomposition of the isomorphism $\phi$. We see that $|Y_v|\geq 1$ since there are no morphisms from  any $X$ of length greater or equal to one to $\unit$. Thus $|\phi_{R_v}|\leq 0$  and hence $|\phi_R|=\sum_v |\phi_{R_v}|\leq 0$. (3) follows from (2) and (4) and (5) follow from (3).

(6): Define $\phi_R^{(1)}=\phi_R$ and for $n\geq 2$: $\phi_R^{(n)}= \phi_R \otimes id \circ \phi_R^{( n-1)} : X\to X\otimes Y^{\otimes n}$
and likewise set $\phi_L^{(1)}=\phi_L$ and for $n\geq 1$: $\phi_L^{(n)}=\phi_L^{(n-1)}\circ \phi_L\otimes id: X\otimes Y^{\otimes n}\to X$. These satisfy $\phi_L^{(n)}\circ \phi_R^{(n)}=id_X$ and there will be infinitely many possible decompositions of $id_X$, one for each $n$, and, hence, we arrive at a contradiction.
\end{proof}

\subsubsection{Morphisms of degree $0$ and almost--connectedness  in the definite case}
We can reduce the question of the existence of an antipode further in the case of a definite Feynman category to the connectedness of the degree $0$ morphisms.
Let $\B_0$ be the span of the  morphisms of degree $0$ and set $\B_\V=\Z[Hom_{\F}(\imath(\V),\imath(\V))]$.

\begin{lem}\label{definitelem} Assume that $\FF$ is decomposition finite, strict and definite w.r.t. $\deg$, then
\begin{enumerate}
\item
  $\B_0$ together with the restriction of the co--unit $\eps|_{\B_0}$ are a sub--co--algebra of $\B$. Together  with $\otimes$ the unit $\eta$, $\B_0$ is  a sub--algebra.
\item  $\B_0$ is isomorphic to the  symmetric tensor algebra on morphisms $\phi_v:X\to \imath(\ast_v)$ of degree $0$.
\end{enumerate}
If the length function $|\;.\;\,|$ is definite, then
\begin{enumerate}
\setcounter{enumi}{2}

\item$\B_{\V}$ together with
the co--unit $\eps|_{\B_\V}$ and the unit $\eta$ form a pointed co--algebra.
\item \label{isoitem} $\B_\V=Hom_{\F}(\iota(\V),\iota(\V))^{\otimes}=\B_\V^{\otimes}$.
Thus any morphism of length $0$ has a decomposition into morphisms of $\B_\V$
up to permutations in the symmetric case.
\end{enumerate}
\end{lem}

\begin{proof}
Suppose $\phi:X\stackrel{\phi_R}{\to}Z\stackrel{\phi_L}{\to}Y$  has degree $0$, then $\deg(\phi_R)+\deg(\phi_L)=\deg(\phi)=0$. In the definite case this implies  $\deg(\phi_R)=\deg(\phi_L)=0$, which shows that $\B_0$ is a sub--co--algebra and since $\otimes$ has is additive in degree, $\B_0$ is also subalgebra. The co-unit restricts and the unit is of length $0$.
Also if $\deg(\phi)=0$ as $\phi\simeq \bigotimes_{v\in V}\phi_v$ with $\deg(\phi_v)\geq 0$ (or $\leq 0$) and $\sum_{v\in V}|\phi_v|=0$, which means that  $\forall v\in V:\deg(\phi_v)=0$.
In particular, if $|Y|=1$, we see that $|X|=1$ as $|\phi|=0$ and since $|\phi_L|=|\phi_R|=0$, also $|Z|=1$ so that $\B_\V$ is a sub--co--algebra. The image of $\eta$ is in $\B_\V$ and $\eps$ restricts as the $id_X\subset \B_V$.
The last statement follows from (2) by the observation that if $|\phi_v|=0$, then $|X_v|=1$ and hence $X_v=\imath(*_v)$.

\end{proof}

\begin{rmk}
The elements of $\B_\V$ split according to whether they are  isomorphisms or not. That is, whether or not they lie in $Mor(\V)$.
\end{rmk}

 By induction, one can see that what can keep things from being connected is $\B_0$ or in the case of $\deg=|\;.\;|$ being definite $\B_{\V}$. This is analogous to the situation for co--operads with multiplication, where, $\V$ is trivial and $\B_\V=\O(1)$ is the pointed co--algebra as in Definition \ref{almostconnecteddef}.

\begin{cor}
Assume that $\FF$ has almost group--like identities. If $\FF$ has a definite degree function
then
\begin{enumerate}
\item $\B_{0}^{iso}:=(\B_0/\sim)$  is a sub--bi--algebra with the induced unit and co--unit.
 \item If  $\I_0$  is the restriction  of the ideal and co--ideal $\I=\langle[id_X]-[id_Y]\rangle$ to $\B_0$ then $(\H_0=\B^{iso}_0,\eta,\eps)$ is a sub--bi--algebra of $\H$.

\item If $|\;.\;|$ is a definite degree function then  $\Biso_{\V}:=(\B_\V/\sim)$ is a sub--bi--algebra with the induced unit and co--unit.
\item Let $\I_\V$ be $\I$ restricted to $\Biso_{\V}$ then $(\H_\V:=\Biso_\V/\I_\V,\Deltaiso,\epsiso,\etaiso)$ is a sub--bi--algebra of $\H$.
\end{enumerate}
\end{cor}

\begin{proof}
Immediate from the above.
\end{proof}

\begin{df}
\label{almostconnecteddef}
We call $\FF$ almost connected with respect to a given definite degree function if
\begin{enumerate}
\renewcommand{\theenumi}{\roman{enumi}}
\item $\FF$ is factorization finite.
\item $\FF$ has almost group--like identities.
\item $(\H_0^{iso},\Delta^{iso},\eps,\eta)$ is connected as a pointed co--algebra
\end{enumerate}
\end{df}

\begin{lem} Assume $\FF$ is factorization finite, has almost group--like identities and $|\;.\;|$ is a definite degree function. Then:
if $(\H_V,\eps,\eta)$ is almost connected,  $(\H_0,\eps,\eta)$ is as well and hence $\FF$ is almost connected w.r.t.\ $|\;.\;|$.
\end{lem}

\begin{proof}
This follows from Lemma \ref{definitelem} (\ref{isoitem}) by applying  the bi--algebra equation.
\end{proof}

\begin{thm}
\label{HopfFeythm}
 If $\FF$ is almost connected then $\FF$ is  Hopf.
\end{thm}
\begin{proof}
We show that $\H$ is co--nilpotent and hence connected. WLOG we assume $\deg$ is non--negative. Any decomposition of a morphisms $\phi$ into $(\phi_0,\phi_1)$  for which $\deg(\phi_0),\deg(\phi_1)\neq 0$ has $\deg(\phi_0),\deg(\phi_1)<\deg(\phi)$ due to the additivity of $\deg$. These terms of $\Delta^{iso}$ of  lesser degree are taken care of by induction. The terms with degree $0$ factors are taken care of by the almost connectedness of $\B_0$ and co--associativity.
\end{proof}

\begin{prop}
If $\deg$ is a proper degree function for a factorization finite $\FF$, then $\FF$ is almost connected and hence Hopf.
\end{prop}

\begin{proof}
By Proposition \ref{degreeprop} (\ref{properitem}) $\FF$ has almost group--like identities. $\H_0=[id_{\unit}]$ and is connected.
\end{proof}

\begin{rmk}
Any morphism $\phi:X\to Y$ satisfies $\Delta(\phi)=id_X\otimes \phi+ \phi\otimes id_Y + \dots$. In the case of almost group--like identities, the $id_X$  are  group--like elements in $\B^{iso}$. Hence it is interesting to study the co--radical filtration and the $([id_X],[id_Y])$--primitive elements in $\B$. They correspond to the generators for morphisms in Feynman categories \cite{feynman}. In the main examples they are all tensors of elements of length $1$.

\end{rmk}

\subsection{Functoriality}
\label{functpar}
Let $\ff:\FF\to \FF'$ be a morphism of Feynman categories. In the strict case, this is  a pair of functors $\ff=(v,f)$: $v:\V\to \V$ and
$f:\F\to \F'$, strict symmetric monoidal, compatible with all the structures, see \cite[Chapter 1.5]{feynman}. In general, one allows strong monoidal functors.
In the non-$\Sigma$ case, the functor $f$ is required to be strict, resp.\ strong monoidal.
For a morphism $\phi\in Mor(\F')$ thought of as a characteristic function $\phi(\psi)=\delta_{\phi,\psi}$ one calculates

\begin{equation}
\label{pullbackeq}
f^*(\phi):=\phi \circ f=\sum_{\hat\phi\in Mor(\F):f(\hat\phi)=\phi}\hat\phi
\end{equation}
This induces a pull--back operation under $\ff$.
The pull--back descends to isomorphism classes. We set $[\phi]([\psi])=1$ if $\phi$ and $\psi$ are in the same class and $0$ otherwise.
The lift is defined by $f^*([\phi])([\hat \psi]):=[\phi]([f(\hat \psi)])$.
\begin{prop}
\label{bialgfunprop}
For non--$\Sigma$ Feynman categories:
\begin{enumerate}
\item Via $f^*$, $\ff$ induces a morphism of unital algebras $\B_{\FF'}\to \B_{\FF}$ .
\item If $f$ is injective on objects, then $f^*$ induces a morphism of co--algebras $\B_{\FF'}\to \B_{\FF}$.
\item If $f^*$ is bijective
on objects, it induces a morphism of co--unital co--algebras $\B_{\FF'}\to \B_{\FF}$.
\end{enumerate}
For Feynman categories:
\begin{enumerate}
\setcounter{enumi}{3}
\item
Via $f^*$, $\ff$ induces a morphism of unital algebras $\B^{iso}_{\FF'}\to \B^{iso}_{\FF}$.
\item If $f$ is essentially injective on objects, then $f^*$ induces a morphism of co--algebras $\B^{iso}_{\FF'}\to \B^{iso}_{\FF}$.
\item If $f^*$ is essentially bijective
on objects, it induces a morphism of co--unital co--algebras $\B^{iso}_{\FF'}\to \B^{iso}_{\FF}$.
\end{enumerate}
\end{prop}

\begin{proof}
In the non--$\Sigma$ case:
For a strictly monoidal $f:\F\to \F'$, $f^*$ is  functorial in the algebra structure: using \eqref{pullbackeq}. Consider $\phi:X\to Y, \psi:X'\to Y'$, then
$(f^*\otimes f^*)(\phi\otimes \psi)=(\phi\circ f)\otimes (\psi\circ f)=(\phi\otimes \psi)\circ f=f^*(\phi\otimes \psi)$. Here for the penultimate equality:
let $\hat X,\hat X',\hat Y$ and  $\hat Y'$ be lifts of $X,X',Y$ and $Y'$ and let
$\hat \Phi:\hat X\otimes \hat X'\to \hat Y\otimes \hat Y'$ then $\hat \Phi$
decomposes as $\hat \Phi^1\otimes \hat \Phi^2$, $\hat\Phi^1:\hat X\to \hat Y, \hat\Phi^2:\hat X'\to \hat Y'$, since we are in the non--$\Sigma$ case, and thus $f(\hat \Phi)=f(\hat\Phi^1\otimes \Phi^2)=f(\hat\Phi^1)\otimes f (\Phi^2)$.

For the co--product one calculates:
\begin{eqnarray}
\Delta(f^*\phi)&=&\sum_{\hat\phi\in Mor(\F):f(\hat\phi)=\phi} \sum_{(\hat\phi_0,\hat\phi_1):\hat\phi_1\circ \hat \phi_0=\hat \phi}\hat\phi_0\otimes \hat\phi_1\\
(f^*\otimes f^*)\Delta(\phi)&=& \sum_{(\phi_0,\phi_1):\phi_1\circ\phi_0=\phi}
\sum_{\hat\phi_0, \hat\phi_1\in Mor(\F):f(\hat\phi_0)=\phi_0, f(\hat\phi_1)=\phi_1}\hat\phi_0\otimes \hat\phi_1\nn
\end{eqnarray}
We now check that the sums coincide. Certainly for any term in the first sum corresponding to  decomposition $\hat\phi=\hat\phi_1\circ\hat\phi_0$ appears in the second sum, since $f$ is a functor: $
f(\hat\phi_1)\circ f(\hat \phi_0)=f(\hat\phi_0\circ\hat\phi_1)=f(\hat\phi)=\phi$.
The second sum might be larger, since the lifts need not be composable. If, however, $f$ is injective on objects, then all lifts of a composition are composable and the two sums agree.
The unit  agrees,  because of the injectivity and uniqueness of the unit object and the triviality of $Hom(\unit,\unit)$. For the co--unit, we need bijectivity.
Namely, $1=\eps(id_X)$, but if $f$ is not surjective, then some $f^*(id_X)=0$ and $\eps(f^*(id_X))=0\neq1$. If $f$ is not injective, then as all the $f(id_{\hat X})=id_X$
for all $\hat X: f(\hat X)=X$,
$\eps(f^*(id_X)= \sum_{\hat X:f(\hat X)=X}$ and the sum is $>2$ for some $X$. Thus the condition is necessary. It is also sufficient. If $f$ is bijective on objects, then, $\widehat{id_X}=id_{\hat X} +T$, with $\eps(T)=0$.
This implies that $\eps(f^*(id_X))=\eps(id_{\hat X})=1$ and $\eps(f^*(\phi))=0$ if $\phi\neq id_{\hat X}$.
as there is no $id_{\hat X}$ in the fiber over $\phi$ if $\phi$ is not an identity. If the functor is not injective, we might have more objects in the fiber and if it is not surjective  $f^*(id_X)$ can be $0$.

In the symmetric situation, the arguments are analogous using isomorphism classes. Although one cannot guarantee the decomposition of $\Phi$ as above, there is a decomposition up to isomorphism $\hat\Phi=\hat\Phi^1\otimes \hat\Phi^2\circ \sigma$. Likewise, the essential injectivity ensures that the lifts are composable as classes and the essential surjectivity is needed to preserve the co--unit.
\end{proof}

\begin{rmk}
A functor between Feynman categories is an indexing if it is bijective on objects, cf.\ \cite{Frep} for more details on indexings.
\end{rmk}

\begin{df} We call a functor $\ff$ as above Hopf compatible if it is essentially bijective and $f^*( \I_{\FF'})\subset I_{\FF}$.
\end{df}
The following is straightforward.
\begin{prop} If $\FF$ and $\FF'$ are  Hopf, a Hopf compatible functor induces a morphism of
Hopf algebras $\H_{\FF'}\to \H_{\FF}$. \qed
\end{prop}
The following is a useful criterion:
\begin{prop} If in addition to being essentially bijective $f$ does not send any non--invertible elements of $\Mor(\F)$ to invertible elements in\\ $\Mor(\F')$, then $\ff$ is Hopf compatible.
\end{prop}

\begin{proof}
That the condition is necessary is clear. Fix $X$, then up to isomorphism there is a unique lift $\hat X$ of $X$.
Any lift of $id_X$ is then an isomorphism $\hat \phi:\hat Y \to \hat Y'$  with both $\hat Y$ and $\hat Y'$
 being isomorphic to $\hat X$ which means that $[\hat \phi]=[\id_{\hat X}]$ and $f^*[id_X]=[id_{\hat X}]$.
\end{proof}

These criteria reflect that  Hopf algebras are very sensitive to invertible elements. It says that we can identify isomorphisms and are allowed to identify morphisms, but only in each class separately.

\begin{ex}
\label{operadex1}
An example is provided by the map of operads: rooted 3--regular forests $\to$ rooted corollas.
This give a functor of  Feynman categories enriching $\Surj$ or in the planar version of $\Surj_{<}$, see \ref{hyppar}.
This functor is Hopf compatible thus induces a  map of Hopf algebras which is the morphism considered by Goncharov in \cite{Gont}.
\end{ex}
\begin{ex}
\label{operadex2}

Another example is given by the map of rooted forests with no binary vertices $\to$ corollas.
The corresponding morphisms of Feynman categories is again Hopf compatible.

However, if we consider the functor of Feynman categories induced by rooted trees $\to$ rooted corollas is not Hopf compatible. It sends all morphisms corresponding to binary trees to the identity morphism of the corolla with one input. Thus is maps non--invertible elements to invertible elements. The presence of these extra morphisms in $\H_{CK}$ is what makes it especially interesting. They also correspond to a universal property, see  \cite{MoerdijkKreimer} and \cite[Example \ref{P1-moeex}]{HopfPart1}.
\end{ex}

\section{Variations on the bi-- and Hopf algebra structures}
\label{alternatepar}
Here we will give some variations of the structures above. The first is an analysis of the role of basic morphisms as indecomposables.
The second is the possibility to modify the bi--algebra structure and how to twist by co--cycles.
For the latter there are two relevant constructions. The first involves quotienting by  isomorphisms
and the second uses co--cycles to twist the co--multiplication.

The need to regard twists stems from the fact, that in the symmetric case the bi--algebra equation fails on the level of morphisms, i.e.\ without passing to the isomorphism classes. The reason for this is that $Aut(X)\times Aut(X')\subset Aut(X\otimes X')$ is a proper subset due to the permutation symmetries. To remedy this one can twist in certain situations, for example  if $\bar d$ is a free action.

There would be a third alternative, which is to use representations, in the spirit in which they appear in the fusion rules in physics. But, we will not delve into this further technical complication at this point and leave it for future study.

\subsection{Bi-algebra structure induced from indecomposables}
\label{indecomppar}
For a strict Feynman category  $Mor(\F)=Obj(\imath^\otimes \downarrow \imath)^\otimes$ and hence $\B$ is the  strictly associative free  monoid on $\B_1=\Z[Ob(\imath^\otimes \downarrow \imath)]\subset \B$ with additional symmetries possibly given by the commutativity constraints induced by $\F$.

\begin{lem}
\label{indecomplem}
If $\F$ is strict and non--$\Sigma$, $\B_1$ is the set of indecomposables.
\end{lem}
\begin{proof} By axiom (ii) any morphism with target of length greater or equal to $2$ is decomposable.
If the target of a morphism $\phi$ has length $1$,  it can only decompose as $\phi=\hat\phi\otimes_\Z \lambda$ with $\lambda\in \Z[Hom(\unit,\unit)]=\Z id_\unit$, since the only object of length $0$ is unit $\unit$ and $\FF$ was taken to be strict. Hence $\lambda=\pm id_\unit$ is itself a unit in the algebra
and $\phi=\pm \hat\phi$.
\end{proof}

We now suppose that $\B_1$ is decomposition finite, which means that the sum in \eqref{Delta-decomp-basic} is finite.
Consider the one--comma generators $\B_1$ and  define
\begin{equation}\label{Delta-decomp-basic}
\Delta_{indec}(\phi)=\sum_{\{(\phi_0,\phi_1):
\phi=\phi_0\circ\phi_1\}}\phi_0 \otimes \phi_1
\end{equation}
here $\phi_0\in \B_1$ and $\phi_1=\bigotimes_{v\in V} \phi_v$ for $\phi_v\in \B_1$.
We extend the definition of $\Delta_{indec}$ to all of $\B$ via the bi-algebra equation.
\begin{equation}
\Delta_{indec}(\phi\otimes \psi):=
\sum (\phi_0\otimes \psi_0)\otimes (\phi_1\otimes \psi_1)
\end{equation}
where  we used Sweedler notation.\footnote{If there is a non--trivial commutativity constraint, we take this to mean $\sigma_{23}\circ \Delta\otimes \Delta$}
$$\eps(\phi)=
\begin{cases} 1& \text{if } \phi=id_X\\
0&\text{else }
\end{cases}
$$

In this case there is a direct proof of the bi--algebra structure.  {\it A posteriori} using Lemma \ref{indecomplem} it follows that this bi--algebra structure coincides with the decomposition bi--algebra structure.

\begin{prop}
 With the assumptions on $\F$ as above and that $\B_1$ is decomposition finite, the tuple $(\B,\otimes_\F,\Delta_{indec},\unit,\eps)$ is a bi--algebra.  A posteriori $\Delta=\Delta_{indec}$.
\end{prop}

\begin{proof} The multiplication is unital and associative. That the co--product is co--associative and $\eps$ is a co-unit is a straightforward check. The latter follows from the decomposition $id_X=\otimes_v id_{\ast_v}$ if $X=\otimes_v \ast_v$. The fact that the bi--algebra equation holds, follows from the fact that all elements in $\B_1$ are indecomposable with respect to this product. For the co--associativity, we notice that in both iterations we get sum over decomposition diagrams $\phi=\phi'''\circ\phi''\circ\phi'$.
\begin{equation}
    \xymatrix{
    X=\bigotimes_v \bigotimes_{w\in V_v} X_w\ar@{=}[r]\ar[d]_{\phi'=\otimes_w\phi_w}&\bigotimes_v \bigotimes_{w\in V_v}\bigotimes_{u\in V_w}*_u\ar[rr]^{\phi=\bigotimes_u \phi_u}&&\ast\\
    Z_1=\bigotimes_v \bigotimes_{w\in V_v} \ast_w\ar@{=}[r]&\bigotimes_v Z_v \ar[rr]^{\phi''=\bigotimes_v \phi''_v}&&Z_2=\bigotimes_v \ast_v \ar[u]_{\phi'''=\bigotimes \phi'''_v}
    }
\end{equation}
where the order of the factors is fixed and the sum is over the possible morphisms and bracketings.
That $\Delta=\Delta_{indec}$ follows from the equality of the co--products on indecomposables for the bi--algebra which by Lemma \ref{indecomplem} are precisely $\B_1$.
\end{proof}

\begin{rmk}
This two step process corresponds to the free construction $\check \O^{nc}$ in Chapter 1.
A prime example is the bi--algebra of rooted planar trees aka.\ bi--algebra of forests of Connes and Kreimer \cite{CK}. The usual way this is defined is to give the co--product on indecomposable, viz.\ trees, and then extend using the bi--algebra equation.
\end{rmk}

\subsection{Isomorphisms, quotients and twists}
\label{quotpar}
We collect more precise information about the isomorphisms and their role in order to make the more specialized constructions. The first is a quotient by the co--ideal of isomorphisms in the non--$\Sigma$ case. In the symmetric case, although we have a co--ideal to divide by, there is a problem with the bi--algebra equation already on the level of the morphisms. Note, we are not taking isomorphisms yet. To remedy the situation, one can introduce twists in certain situations.

\subsubsection{Iso-- and Automorphisms}
\label{isoformulapar}
By the  conditions of a Feynman category for  $X=\bigotimes_{i=1}^k\ast_i$.
In the non--symmetric case, any automorphism factors, so
$$Aut(X)\simeq Aut(\ast_1)\times \dots\times Aut(\ast_k) \text{ in the non--symmetric case.}$$
In the symmetric case its automorphisms group is the wreath product $$Aut(X)\simeq(Aut(\ast_1)\times \dots\times Aut(\ast_k))\wr \SS_k \text{ in the non--symmetric case.}$$

\subsubsection{The co--ideal generated by the isomorphisms relation}

Recall that $f\sim g$ if they are isomorphic, c.f.\ \S\ref{isopar}.

\begin{prop}
Let $\C$ be the ideal generated by elements $f-g$ with $f\sim g$. Then
\begin{equation}
\Delta(\C)\subset \B\otimes \C + \C\otimes \B
\end{equation}
and hence $\B/\C$ is a unital algebra and (non--co--unital) co--algebra.

Extending scalars to $\Q$,  there is a co--unit on $\Bquot=\B/\C\otimes_\Z \Q$
\begin{equation}
\epsquot([f]):=\begin{cases}\frac{1}{|Iso(X)| |Aut(X)|}&
 \text{if } [f]=[id_X]\\
0 & \text{else}\\
\end{cases}
\end{equation}

\end{prop}

\begin{proof}
To compute the co--product,
we break up the sum over the factorizations of $f$ and $g$ with $f\sim g$  into the pieces that correspond to a factorization through a fixed space $Z$.
\begin{equation}
\label{coidealeq}
\xymatrix{
&Z\ar[dr]^{f_2}&\\
X\ar[rr]^{f}\ar[d]^{\simeq}_{\sigma'}\ar[ur]^{f_1}&&Y\ar[d]_{\simeq}^{\sigma}\\
X'\ar[rr]^g\ar[dr]_{g_1}&&Y'\\
&Z\ar[ur]_{g_2}&\\
}
\end{equation}
Now the term in $\Delta(f-g)$ corresponding to $Z$ is
$\sum_i f^i_2\otimes f^i_1-\sum_j g^j_2\otimes g^j_1$. Re--summing using  the identification $g^i_1:=f^i_1\circ \sigma^{\prime -1}$ and
$g^i_2:=\sigma\circ f^i_2$ this equals to
$$\sum_i (f^i_2\otimes f^i_1-g^i_2\otimes g^i_1)=
 \sum_i (f^i_2-g^i_2)\otimes g^i_1+ \sum_i f^i_2\otimes (f^i_1-g^i_1)$$

For the co--unit, notice that $\Delta([f])=[\Delta(f)]$ is a sum of terms factoring through an intermediate space $Z$. If $Z\not\simeq X,Y$ then these terms are killed by $\epsquot$ on either side, since there will be no isomorphism in the decomposition. If $Z\simeq X$, then  any factorization $ f\circ \sigma^{-1}\otimes\sigma $ with $\sigma\in Iso(X,Z)$ descends to
$[f\circ\sigma^{-1}]\otimes[\sigma] = [f]\otimes[id_X]$. Since $Iso(X,Z)$ is a left $Aut(X)$ torsor, there are exactly $|Aut(X)||Iso(X)|$ of these terms  and
$\epsquot\otimes id$ evaluates to $1\otimes [f]$ on their sum.
By Lemma \ref{isolem}, all other decompositions will evaluate to $0$ and we obtain that $\epsquot$ is a left co--unit. Likewise $\epsquot$ is a right co--unit by considering
the terms which factor through $Y'\in Iso(Y)$.

\end{proof}

\begin{rmk}\mbox{}
\begin{enumerate}
\item Note that $\C$ is not a co--ideal in general, since for any automorphism $\sigma_X\in Aut(X):[\sigma_X]=[id_X]$ and hence $\eps(\C)\not\subseteq ker(\epsilon)$.  Likewise if $X\simeq Y$ and $\phi:Y\stackrel{\sim}{\to} Y'$ then
$[id_X]=[\phi]$ from Lemma \ref{isolem}. This is why we need a new definition for the co--unit.
If there are no automorphisms and the underlying category is skeletal, then $\eps$ descends as claimed in \cite{JR}.

\item The equivalence relation $\sim$ is coarser
than the equivalence studied in \cite{JR} for the standard reduced incidence category.
\item Extending scalars from $\Z$ all the way to  $\Q$ may not be necessary; we only need that $|Iso(X)|$ and  $|Aut(X)|$ are invertible for all $X$. Although in the symmetric case, the automorphisms groups will contain all $\SS_n$ and hence $\Q$ is necessary.
\item One can get rid of the terms $X'\in Iso(X)$ in $\Delta(id_X)$ and the factor $|Iso(X)|$ by considering a skeletal version.
Recall that skeletal means that there is only one object per isomorphism class.
\item Although in the symmetric case, the bi--algebra equation does not hold on $\B$,  it does on a non--$\Sigma$ Feynman category. The difference is due to \S\ref{isoformulapar}.
 The failure in the symmetric case  is analyzed in detail in \S \ref{symfailpar} below.
\end{enumerate}
\end{rmk}

\begin{thm}
\label{coprodthm}
Let $\F$ be a decomposition finite non--$\Sigma$ Feynman category set  $\Bquot$ with the induced product, unit, co--product and co--unit $\epsquot$ is a bi--algebra.
\end{thm}

\begin{proof}

In the non--symmetric case, the compatibility of product and co--product descend as does the compatibility of the unit.
For the co--unit, we notice that $\epsquot([\phi\otimes \psi])$ as well as $\epsquot([\phi])\epsquot([\psi])$ are $0$ unless $[\phi]=\lambda [id_X]$ and $[\psi]=\mu[id_Y]$. If this is satisfied, by the conditions of a non--symmetric Feynman category $|Aut(X)||Aut(Y)|=|Aut(X\otimes Y)|$ as well as $|Iso(X)||Iso(Y)|=|Iso(X\otimes Y)|$ so that $\epsquot([id_X]\otimes[id_Y])=\epsquot([id_X])$ $\epsquot([id_Y])$.
\end{proof}

We define the ideal $\bar\J=\langle |Aut(X)||Iso(X)|id_X-|Aut(Y)||Iso(Y)|id_Y\rangle$ of $\Bquot$, and then consider $\Hquot=\Bquot/\bar\J$.

\begin{thm}
 Assume that $\FF$ is decomposition finite non--$\Sigma$ and has almost group--like identities,
 then, $\bar\J$ is a co--ideal in $\Bquot$ and $\Hquot=\Bquot/\bar\J$ is a bi--algebra with co--unit induced by $\epsquot$ and unit $\eta_{\Hquot}(1)=[\id_{\unit_{\F}}]$. If $\Hquot$ is connected, then it is a Hopf algebra.
\end{thm}
\begin{proof}
In $\Bquot$, \eqref{coprodcalceq}
reads $\Delta([id_X])=|Aut(X)||Iso(X)|[id_X]\otimes[id_X]$,
so that
\begin{multline*}
\Delta(|Aut(X)||Iso(X)|[id_X]) -|Aut(Y)||Iso(Y)|[id_Y]\\
=(|Aut(X)||Iso(X)|)^2[id_X]\otimes[id_X]- (|Aut(Y)||Iso(Y)|)^2 [id_Y]\\
= (|Aut(X)||Iso(X)|[id_X]-|Aut(Y)||Iso(Y)|[id_Y]\otimes |Aut(X)||Iso(X)|[id_X]+\\
|Aut(Y)||Iso(Y)|[id_Y]\otimes
(|Aut(X)||Iso(X)|[id_X]-|Aut(Y)||Iso(Y)|[id_Y])
\end{multline*}

Hence, the ideal $\bar \J$ is
generated by elements $|Aut(X)||Iso(X)|[id_X]-|Aut(Y)||Iso(Y)|[id_Y]$ is also a co--ideal, as these also satisfy $$\epsquot(|Aut(X)||Iso(X)|[id_X]-|Aut(Y)||Iso(Y)|[id_Y])=1-1=0$$
It is easy to check that $\eta_{\Hquot}$ yields a split co--unit.
\end{proof}
\begin{rmk}\mbox{}
\begin{enumerate}
\item One can use a notions of grading and almost connectedness here as in previous analysis of connectedness.
This is entirely analogous to \cite[\S\ref{P1-hopfcondpar}]{HopfPart1}.
\item   If $\V$ is also discrete and  $\F$ skeletal then $\bar \J=\la [id_X]-[id_Y]\ra$ and $\Bquot=\B^{iso}\otimes_\Z\Q$.
This is the case for the non--$\Sigma$ operads, see \S\ref{enrichedpar}.
\end{enumerate}
\end{rmk}

\subsubsection{The symmetric case: a careful analysis of the two sides of the bi--algebra equation}
\label{symfailpar}
The following proposition a finer version of Proposition \ref{f-to-b-nonsym} which also holds in the symmetric case.
\begin{prop}
For any factorization of $\Phi=\phi\otimes \psi:X\times X'\to Z\otimes Z'$ as $\Phi_0\circ \Phi_1:X\times X'\to Y\to Z\otimes Z'$ there exists a decomposition $\sigma': Y\simeq \hat Y\otimes \hat Y'$ and a factorization $(\phi_0\otimes \psi_0,\phi_1\otimes\psi_1)$ factoring through  $\hat Y\otimes \hat Y'$ such that $(\Phi_0,\Phi_1)=\bar d(\sigma')(\phi_0\otimes \psi_0,\phi_1\otimes\psi_1)= ( \phi_0\otimes \psi_0\circ\sigma^{\prime -1},\sigma'\circ \phi_1\otimes\psi_1)$. Furthermore, all such factorizations are in 1--1 correspondence with the cosets $Iso(Y,\hat Y\otimes \hat Y')/Aut(\hat Y)\times Aut(\hat Y')$.
\end{prop}

\begin{proof}
Given a decomposition of $\Phi$ as $(\Phi_0,\Phi_1)$, we can follow the argument of the proof of Theorem \ref{f-to-b-nonsym} up until the discussion of the isomorphisms $\sigma$ and $\sigma'$.

In the symmetric case, {\it a priori} there could be permutations involved for $\sigma$ and $\sigma'$.  This is, however, not the case for $\sigma$, and  we can absorb it to get decompositions of $\Phi$. More precisely, the isomorphism $\sigma$ has to be a block isomorphism as axiom (ii) applies to the two decompositions $\Phi=\phi\otimes \psi$ and $\Phi\simeq \hat\phi_0\circ\hat\psi_0\otimes \hat\phi_1\circ\hat\psi_1$. This means that  $\sigma$ in \eqref{deltamuisoeq} is uniquely a  tensor product of isomorphisms $\sigma=\sigma_1\otimes \sigma_2$,  since both decompositions have the same target decomposition $Z\otimes Z'$. By pre-composing, we get the tensor decomposition $\Phi= (\hat\phi_0\otimes \hat\psi_0)\circ  (\hat\phi_1\otimes \hat\psi_1)\circ(\sigma^{-1}_1\otimes \sigma^{-1}_2) $.

Continuing with the decomposition of this form, we turn to  $\sigma'$. We know that by (ii) that $\sigma'$ can be written as a tensor product decomposition preceded by a permutation. If $\sigma'=\sigma'_1\otimes \sigma'_2$, we have that $Y=Y'\otimes Y''$ and $(\Phi_0,\Phi_1)$ appears  as a tensor product. Again absorbing the tensor decomposition means that the remaining terms corresponding to non--tensor decomposable permutations, and hence to a sum over the respective cosets.
\end{proof}

Notice that fixing any isomorphism in $Iso(Y,\hat Y\otimes \hat Y')$ identifies it with $Aut(\hat Y\otimes \hat Y')$
so that the quotient group
$Iso(Y,\hat Y\otimes \hat Y')/[Aut(\hat Y)\times Aut(\hat Y')]$ becomes identified with $Aut(\hat Y\otimes \hat Y')/
[Aut(\hat Y)\times Aut(\hat Y')]\simeq Aut(Y)/(Aut(Y')\times Aut(Y))$. Using this identification, we can see that if $\F$ is a Feynman category, then in the proof of Theorem \ref{f-to-b-nonsym} the sets of diagrams  agree  up to a choice of cosets of isomorphisms of $\sigma'$ in \eqref{deltamuisoeq}, that is the difference in the count of diagrams will result from  the cosets
$Aut(\hat Y\otimes \hat Y')/(Aut(\hat Y)\times Aut(\hat Y'))$. More precisely:

\begin{cor}
\label{actioncor}

Splitting the sum $\Delta\circ \mu$  into sub--sums over a fixed  decomposition of  $Y=Y'\otimes Y''$, $\Delta\circ \mu= \sum_Y (\Delta\circ \mu)_Y$, we have
\begin{multline}
\label{actioneq}
\sum_Y (\Delta\circ \mu)_Y =\\\sum_{Y=Y'\otimes Y''}\sum_{[\sigma']\in Aut(Y)/(Aut(Y')\times Aut(Y''))} \bar d(\sigma')\left(\mu\otimes \mu\circ\pi_{23}\circ \Delta\otimes \Delta\right)_{\hat Y\otimes \hat Y'}
\end{multline}
where we have fixed a decomposition $\hat Y\otimes \hat Y'\simeq Y$ and used the identification above.

In the non--$\Sigma$ case, $Aut(Y)\simeq Aut(Y')\times Aut(Y'')$, which implies  that \\ $Aut(Y)/(Aut(Y')\times Aut(Y''))$ is trivial and we recover Theorem \ref{f-to-b-nonsym}.\end{cor}

Thus, in the symmetric case,  the bi--algebra equation fails on $\B$.
An interesting aspect is the possibility to twist  the co--multiplication by a co--cycle, to make it hold on $\B/\C$,
which in certain cases  leads to a bi--algebra structure.

\begin{ex}
In the case of trivial $\V$, in the symmetric case, we have $Aut(n)\times Aut(m)=\SS_n\times \SS_m\subset \SS_{n+m}=Aut(n+m)$ in $\V^{\otimes}$.
 Let us consider the trivial Feynman category with trivial $\V$, that is $\F=\SS$, the skeletal version of $\V^{\otimes}$, which has the natural numbers as objects and only isomorphisms as morphisms, where $Hom(n,n)=Aut(n,n)=\SS_n$.
We will consider $\Delta(id_n\otimes id_m)= \Delta(id_{n+m})=\sum_{\sigma\in \SS_{n+m}} \sigma\otimes \sigma^{-1}$. We analyze the possible diagrams \eqref{deltamuisoeq} for the summand $\sigma\otimes \sigma^{-1}$ in the proof of Theorem \ref{f-to-b-nonsym}.
\begin{equation}
\label{symmexeq}
\xymatrix{
n\otimes m\ar[r]^-{\sigma'}\ar[ddrr]_{\hat\sigma_n\otimes\hat\sigma_m}&n\otimes m=n+m \ar[rr]^{id_{n+m}=id_n\otimes id_m}\ar[dr]^{\sigma}&&n+m\\
&&n\otimes m
\ar[ur]^{\sigma^{-1}}\ar[d]_{\sigma_n\otimes \sigma_m\circ\sigma^{-1}}&\\
&&n\otimes m\ar[uur]_{\sigma^{-1}_n\otimes \sigma^{-1}_m}&\\
}
\end{equation}
And we see that $\sigma'=\sigma^{-1}_n\otimes \sigma^{-1}_m\circ \hat\sigma_n\otimes \hat\sigma_m=
\sigma^{-1}_n\circ\hat\sigma_n\otimes \sigma^{-1}_m\circ \hat\sigma_m$ absorbing this block isomorphism into $\hat\sigma_n\otimes \hat\sigma_m$, we get the diagram.

\begin{equation}
\label{symmextwoeq}
\xymatrix{
n\otimes m=n+m \ar[rr]^{id_{n+m}=id_n\otimes id_m} \ar[dr]^{\sigma}\ar[ddr]_{\sigma_n\otimes\sigma_m}&&n+m\\
&n\otimes m
\ar[ur]^{\sigma^{-1}}\ar[d]_{\sigma_n\otimes \sigma_m\circ\sigma^{-1}}&\\
&n\otimes m\ar[uur]_{\sigma^{-1}_n\otimes \sigma^{-1}_m}&\\
}
\end{equation}
If $\sigma$ is of the form $\sigma_n\otimes \sigma_m$, then the term appears in $\Delta(id_n)\otimes \Delta(id_m)$. Otherwise, the action of $Aut(Y)$ on $Hom(X,Y)\otimes Hom(Y,Z)$ with $X=Y=Z=n+m$, on the decompositions appearing in $\Delta(id_n)\otimes \Delta(id_m)$ and moreover, picking representatives $\sigma^r$ of
$Aut(n+m)/(Aut(n)\times Aut(m))$ and summing over their action, we get an equality
$$
\Delta(id_n\otimes id_m)=\sum_{\sigma^r\in \SS_{n+m}/(\SS_n\otimes \SS_m)} \rho(\sigma^r) \Delta(id_n)\otimes \Delta(id_m)
$$
In particular for equivalence classes in $\B/C$, we get
$$\Delta([id_n]\otimes [id_m])=\frac{(n+m)!}{n!m!} \Delta([id_n])\otimes \Delta([id_m])
$$
which shows the failure of the bi--algebra equation.

However,  the difference can be absorbed by a co--cycle:
Set $\beta(\sigma_n,\sigma_n^{-1})=\frac{1}{|Aut(n)|}=\frac{1}{n!}$.
Define a new co--multiplication: $\Delta_\beta(id_n)= \beta(\sigma_n,\sigma_n^{-1})\, \sigma_n\otimes \sigma_n^{-1}$
then $\otimes$ and $\Delta_\beta$ on $\Bquot$ satisfy the bi--algebra equation.
\end{ex}

\subsubsection{Actions and  co--cycles}

Recall that there is an  $Aut(Z)$ action on
$Hom(Z,Y) \times  Hom(X,Z)$ given by $\bar d(\sigma)(\phi_0,\phi_1)=( \phi_0 \circ \sigma^{-1},\sigma\circ \phi_1)$.

By a twisting co--cycle for the co--product, we mean a morphism $\B\to Hom(\B\otimes \B,K)$ that is a linear collection of bilinear morphisms $\beta_\phi$, s.t.\
$\Delta_\beta(\phi)=\sum_{(\phi_0,\phi_1)}\beta_\phi(\phi_0,\phi_1)\phi_0\otimes \phi_1 $ is still co--associative. Such a co--cycle is called multiplicative if  $\beta_{\phi\otimes \psi}=\beta_\phi\beta_\psi$ on decomposables. $\beta$ is called co--unital, if there exists a co--unit $\eps_\beta$ for $\Delta_\beta$.

\begin{prop} Assuming for simplicity that we are in the skeletal case.
If the  $Aut(Z)$ action is free on all decompositions, then  we can define a modified co--product
$\Delta_\beta$ on $\B$, defined
the multiplicative co--cycle, which is given by a co--cycle $\beta(\phi_0,\phi_1)=\frac{1}{|Aut(Z)|}$ for a factorization $\phi:X\stackrel{\phi_1}{\to}Z\stackrel{\phi_0}{\to} Y$.  In the  non-$\Sigma$ case the co--cycle is multiplicative, and, if the identities are almost group--like, co--unital.

This co--algebra structure descends to  $\B/\C$ furnishing a bi--algebra structure:
\begin{equation}
    \Deltared([\phi]):= [\Delta_\beta](\phi)= \sum_Z \sum_{i_r}[\phi_1^{i_r}]\otimes [\phi_0^{i_r}]
\end{equation}
where the sum  runs over representatives of the $Aut(Z)$ action.
There is a co--unit
\begin{equation}
\epsred([\phi])=\begin{cases}1&\text{ if } [\phi]=[id_X]\\
0 & \text{else}\\
\end{cases}
\end{equation}
This is true, both in the non--$\Sigma$ and the symmetric case.
\end{prop}
\begin{proof}
The fact that this is co--associative is a straightforward calculation given that the action is free and that the $Aut(Z_1)$ and $Aut(Z_2)$ actions on decompositions $X\to Z_1\to Z_2\to Y$ commute. The co--unit in the skeletal case is simply $\eps_{\beta}(\phi)=1$ if $\phi=id_X$ and $0$ else. The multiplicativity in the non--$\Sigma$ case corresponds to the fact that $Aut(Y\otimes Y')\simeq Aut(Y)\otimes Aut(Y')$.

On $\B/\C$ one calculates:
\begin{eqnarray*}
\Deltared[(\phi])&=&[\Delta_\beta(\phi)]
=
\sum_Z \sum_i \beta(\phi_0,\phi_1) [\phi^i_0\otimes \phi^i_1]\\
&=&\sum_Z \sum_{i_r} \sum_{\sigma\in Aut(Z)}  \frac{1}{|Aut(Z)|} [\phi_0^{i_r}\circ \sigma^{-1}]\otimes[\sigma\circ \phi_1^{i_r}] \\
&=&\sum_Z \sum_{i_r} [\phi_0^{i_r}]\otimes [\phi_1^{i_r}]\\
\end{eqnarray*}

For the bi--algebra equation in the symmetric case:
Inspecting the proof of Corollary \ref{actioncor}, we get an additional factor of $\frac{1}{|Aut(Y)|}$ for each summand in $\Delta\circ \mu$, while on the other side of the equation the factor is $\frac{1}{|Aut(\hat Y)||Aut(\hat Y')|}$. These  cancel with the additional factor of $\frac{|Aut(Y)|}{|Aut(\hat Y)||Aut(\hat Y')|}$ in \eqref{actioneq}.
\end{proof}

\subsubsection{Balanced actions}

  More generally, one could define the putative co--cycle $\beta(\phi_1^i,\phi_0^i)=\frac{1}{|Or(\phi_1,\phi_0)|}$ where $Or(\phi_0,\phi_1)$ is the orbit under the $Aut(Z)$ action.
  If this is indeed a co--cycle  then we say that $\FF$ has  a  {\em balanced} action by automorphisms.
The trivial and free actions are balanced.

\begin{prop} If $\FF$ is non--symmetric, skeletal in the above sense, and decomposition finite with  balanced actions as above
 then tuple $(\B,\otimes,\Delta_{\beta},\eta,\eps_\beta)$ is also a bi--algebra.
\end{prop}
\begin{proof}
The fact that we have an algebra remains unchanged. For the co--algebra, we have to check co--associativity, which is guaranteed by the assumption that the action is balanced. The bi--algebra equation still holds, since the co--cycle is multiplicative: $\beta(\phi_1\otimes \psi_1,\phi_0\otimes \psi_0)=\beta(\phi_1,\psi_1)\beta(\phi_0,\psi_0)$. This follows from the fact that in the non--$\Sigma$ case: $Aut(Z\otimes Z')=Aut(Z)\otimes Aut(Z')$.
\end{proof}

\begin{rmk} \mbox{}
\begin{enumerate}
\item The reduced structure is available for the non--skeletal version. Here, for instance in the free action case, one obtains factors $|Iso(Z)||Aut(Z)|$ which again constitutes a multiplicative co--cycle.

\item {\it A priori} It seems that the two bi--algebra structures $\Delta_{\beta}$ for a balanced action and $\Deltaiso$ may differ. We conjecture that they do coincide for all  Feynman categories of crossed type \cite[\S5.2]{feynman}.
\end{enumerate}
\end{rmk}

\subsubsection{Summary}
Since there are many constructions at work here, we will collect the results for the bi--algebras in an overview theorem:

\begin{thm}
\label{sumthm}
Fix a decomposition finite  non--$\Sigma$ or a factorization finite Feynman category $\FF$, let  $\B=\Z[Mor(\F)$ and $\B^{sk}:=\B_{\F_{sk}}$ based on the skeletal version of $\F$.
 Let $\C$ be the ideal generated by $\sim$ in $\B$ and $\C^{sk}$  the respective ideal in $\B^{sk}$.
 Set $\Biso=\B/\C$, $\Bquot=\Biso\otimes_\Z\Q$.

\begin{enumerate}

\item   Both $\B$ and $\B^{sk}$ are unital algebras with $\otimes$ as product and $id_{\unit}$ as the unit. They are  Morita equivalent as algebras
\item   Both $\B$ and $\B^{sk}$
are co--unital co--algebras with respect to the deconcatenation co--product with co--unit $\eps$.
\item If $\FF$ is a non--$\Sigma$ Feynman category: $\B$ and $\B^{sk}$  are   unital, co--unital bi--algebras.

\item   $\Biso \simeq \B^{sk}/C^{sk}$ as  algebras and there is a bi--algebra structure $(\Biso,\otimes,$ $\etaiso,\Deltaiso,\epsiso)$ defined via co--invariants as in Theorem \ref{Bisothm}.

\item If $\FF$ is non--$\Sigma$ then there is a unital, co--unital quotient bi--algebra
$(\Bquot,\otimes,\etaquot,\Deltaquot,\epsquot)$ as defined in \S\ref{quotpar}.
\item If the action of $Aut(Z)$
on $Hom(X,Z)\times Hom(Z,Y)$ is trivial, free or in general balanced for all $X,Y,Z$, then the twisted $\B$ descends to a bi--algebra
$(\Bquot,\otimes,\etaquot,\Deltared,\epsred)$.
\item All the structures above are graded by the length of a morphism or the degree of a morphism if there is an integer  degree function.
\end{enumerate}

\qed
\end{thm}

\subsection{ Feynman categories, groupoids and de--compositions}
\label{GKTpar}
The co--nilpotence of the deconcatenation is related to iterated factorizations, which appear in \cite[\S3.3]{feynman} in the form of iterated Feynman categories $\FF',\dots \FF^{(n)}, \dots$. The associated maximal sub--groupoids $\V^{(n)\, \otimes}, \dots$  form a simplicial groupoid:  objects at level $n$ are factorizations of morphisms into $n$ chains, with the isomorphisms between these chains. In operad theory this type of groupoid explicitly appeared already in \cite{GKmodular} in the context of (twisted) modular operads, cf.\ also \cite{MSS}.

More explicitly, consider the `fat nerve' $\X=\X(\F)$ of any category $\F$, the simplicial groupoid with $\X_n$ the groupoid of $n$-chains $$\alpha_n=(X_0\to X_1\to\dots\to X_n)\text{ in }\F$$
and the isomorphisms between such chains, and $\X_0=Iso(\F)$.
The simplicial operator $d_1:\X_2\to\X_1$ is composition in $\F$. Its homotopy fiber over an object $\phi:X\to X'$ in $\X_1$ is thus the groupoid $\text{Fact}(\phi)$ of factorizations $\phi\simeq \phi_1\circ\phi_2$.

In a special situation, one can use the theory of decompositions which was developed after \cite{feynman} and the beginning of this paper, cf. \cite[\S3.3]{feynmanarxiv}.

In the transition to decomposition spaces, one however looses the simplicity that the co--product was initially just the dual of the composition.

Suppose $\F$ is any Feynman category such that the factorisations of the identity on the monoidal unit form a contractible groupoid. Then it can be shown that in fact $\X(\F)$ is a {\em symmetric monoidal decomposition groupoid} \cite[\S9]{GKT-decomp1}. The tensor and unit of $\F$ clearly define $\eta:*\to\X$, $\mu:\X\times\X\to\X$, but it is the key hereditary condition of a Feynman category that shows that tensor and composition are compatible: they form a homotopy pullback square
$$
\xymatrix@C+1em@R-1ex{
\text{Fact}(\phi)\times\text{Fact}(\phi')\ar[d]^\simeq_-{\otimes}\rto&\X _2\times \X _2\drpullback
\rto^-{\circ\times\circ}\dto_-{\otimes}
&\X _1\times \X _1\ar[d]^-{\otimes}
&{\hspace*{-3em}}{\ni (\phi,\phi')}\\
\text{Fact}(\phi\otimes\phi')\rto&\X _2\ar[r]_-{\circ}&\X _1&{\hspace*{-5em}}{\ni \phi\otimes\phi'},
}
$$
for all $\phi:X \to Y $ and $\phi':X '\to Y '$, that is, $\otimes:\text{Fact}(\phi)\times\text{Fact}(\phi')\to\text{Fact}(\phi\otimes\phi')$ is a groupoid equivalence.

From \cite[Theorem 7.2 and \S9]{GKT-decomp1} we see that $\X(\F)$ induces a bi--algebra in the symmetric monoidal category of  comma categories of groupoids and linear functors between them, and in \cite{GKT-decomp2} the finiteness conditions necessary and sufficient to pass to bi--algebras in the category of $\Q$-vector spaces are studied.

\section{Constructions and Examples}
\label{constexpar}
The main  examples are already directly accessible via the formalism above. However, more context is provided,
by using several universal constructions on Feynman categories from \cite[\S3]{feynman}, see also \cite{Frep} for more details.

We will go through the examples starting with the basic ones, which contain the three main examples, and then introduce further complexity to provide better insight and further examples.

\subsection{Examples with trivial $\V$ a.k.a.\ Operads and the three main examples}

\label{trivVpar}

Let $\V=\underline{*}$ be the trivial category with one object $*$ and its identity morphism $id_*$.
In the {\em non--symmetric case}, there is an equivalence $\V^{\otimes}\simeq {\bf N}_0$  with the discrete category whose objects are the natural numbers, with $n$ representing $*^{\otimes n}$.  The monoidal structure is given by addition. Here $0=*^{\otimes 0}=\emptyset$.
In the {\em symmetric monoidal} case there is an equivalence $\V^{\otimes}=\SS$, which  again has the natural numbers as objects, but with $Hom_\SS(n,m)=\emptyset$ for $n\neq m$ and $Hom_\SS(n,n)=\SS_n$, the symmetric group. This category is sometimes also denoted by $\Sigma$ and it is the skeleton of $Iso(\FinSet)$, where $\FinSet$ is the category of finite sets with set maps. For more details, see \cite{matrix}, especially \S 2.4.

Consider a strict Feynman category $\FF=(\underline{*},\F,\imath)$  with $Obj(\F)=\N_0$. The monoidal unit is $\unit=0$.
 The basic morphisms will be $\F(n,1):=\O(n)$. Since $\Hom_\F(n,n)=\SS_n$,
the collection $\O(n)$ has an action of $\SS_n$ in the symmetric case.
By the hereditary condition \eqref{morcond}:
\begin{multline}
\Hom_\F(n,k)=\amalg_{(n_1,\dots,n_k:\sum n_i=n}\Hom_\F(n_1,1)\amalg\cdots\amalg \Hom_\F(n_k,1)\\=\amalg_{(n_1,\dots,n_k:\sum n_i=n}\O(n_1)\amalg\cdots\amalg \O(n_k)
\end{multline}
 and the composition
 $\circ:\Hom_\F(k,1)\times \Hom_\F(n,k)\to \Hom_\F(n,1)$ will be given by
  \begin{equation}
\circ:  \O(k)\times(\amalg_{(n_1,\dots,n_k:\sum n_i=n}\O(n_1)\amalg\cdots\amalg \O(n_k))\to \O(n)
 \end{equation}
or in components by
\begin{equation}
\gamma_{k;n_1 \kdk n_k}: \O(k)\times (\O(n_1)\amalg\cdots\amalg \O(n_k))\to \O(n)
\end{equation}
The fact that $\circ$ is associative together with the properties of $id_1$ implies that the $\gamma$ give the collection $\O(n)$ the structure of an operad with unit $u=id_1$.

Furthermore, because of axiom \eqref{objectcond}, we see that $Aut(1)=id_1$, so that $\O(1)$ only has $id_1$ as an invertible element.
In principle, there can be morphisms in $\Hom_\F(0,1)=\O(0)$.
The length of a morphisms in $\Hom_\F(n,k)=n-k$.

This recovers \cite[\S\ref{P1-freecooppar}]{HopfPart1} for the duals of operads in $\Set$, which contains the
 examples of rooted trees. For operads in other categories, see \S\ref{hyppar}.

\begin{prop}
The strict  Feynman categories whose underlying $\V$ is trivial are in 1--1 correspondence with set--operads, whose $\O(1)$ splits as $\O(1)=id_1\amalg \O(1)^{red}$ where no element in $\O^{red}(1)$ is invertible.
They are  non--negative with respect to length, if $\O(0)=\emptyset$
and are non--positive w.r.t.\ length, if $\O(i)=\emptyset$ for $i>0$.

The construction of bi--algebras and conditions for Hopf algebras coincide
in both formulations under this translation.

\qed
\end{prop}

\subsection{Connes--Kreimer tree algebras}

Let $\FF_{CK}$ be the Feynman category with trivial $\V$, $\F$ having objects $\N_0$ and morphisms given by
rooted forests: $Hom_{\F_{CK}}(n,m)$ is the set of $n$-labeled rooted forests with $m$ roots. The composition is given by gluing the roots to the leaves. This is the twist of $\Surj$ by the
 operad of leaf-labeled rooted trees, see \S \ref{hyppar}.

In the non-\SSigma{} version, one uses planar forests/trees and omits labels or equivalently uses orders on the sets of labels. Here this is the twist by the non--\SSigma{} operad of planar forests of $\Surj_<$.

Here there is non--trivial $\O(1)$. This is basically the difference of the $+$ and the $hyp$ construction, see \S\ref{hyppar}. The grading $n-p$ is the native grading and the co--radical length
is the word length of a morphism and is given by the number of vertices.

\subsubsection{Leaf labeled and planar version of Connes--Kreimer}
We now give complete details.
Let  $\O$ be the operad of leaf labeled rooted trees or planar planted trees.
Here $\O(1)$ has two generators: $id_1$ which we denote by $|$, and $\dottree$, the rooted tree with one binary non--root vertex. Composing $\dottree$ with itself $n$ times
will result in $\dottree \, n$, the rooted tree with $n$ binary non--root vertices, aka.\ a ladder with $n$ vertices.
We also identify $\dottree \,0=|$.
Taking the dual, either as the free Abelian group of morphisms, or simply the dual as a co--operad, we obtain a co--operad and the multiplication is either $\otimes$ from the Feynman category or $\otimes$ from the free construction. That these two coincide follows from condition \eqref{morcond} of a Feynman category.
$\eta$ is given by $|=id_1$. The Feynman category and the co--operad are almost connected, since $\Delta(\dottree\, n)=\sum_{(n_1,n_2): n_1,n_1\geq 0,n_1+n_2=n}
\dottree \, n_1\otimes \dottree \, n_2$ and hence the reduced co--product is given by
$\bar\Delta (\dottree\, n)=\sum_{(n_1,n_2): n_1,n_1\geq 1,n_1+n_2=n}
\dottree \, n_1\otimes \dottree \, n_2$ whence $\check\O(1)$ is nilpotent.

If we take planar trees, there are no automorphisms and we obtain the first Hopf algebra
of planted planar labeled forests. Notice that in the quotient $[|]=[||\dots|]=[1]$ which says that there is only one empty forest.
If we are in the non--planar case, we obtain a Hopf algebra of rooted forests, with labeled leaves.
These structures are also discussed in \cite{HopfPart1}, and in \cite{foissyCR1},\cite{foissyCR2} and \cite{KreimerEbrahimi-Fard}.

\subsubsection{Algebra over the operad description for Connes--Kreimer}
\label{Moealg}
If one considers  algebras over the operad $\O$, then for a given algebra $(\rho,V)$,
$\rho(\dottree)\in Hom(V,V)$ is a ``marked'' endomorphism. This is the basis of the   constructions of \cite{MoerdijkKreimer}.
One can also add more extra morphisms, say $\dottree \, c$ for $c\in C$ where $C$ is some indexing set of colors. This was considered in \cite{vanderLaan}.
In general one can include such marked morphisms into Feynman categories (see \cite{feynman}[2.7]) as morphisms of $\emptyset\to \ast_{[1]}$.

\subsubsection{Unlabeled and symmetric version}
In the non--planar case, we have the action of the symmetric groups as $Aut(n)$. The bi--algebra on the co--invariants and the Hopf quotient of Theorem \ref{Bisothm}   yield
the same results as the constructions \cite[\S\ref{P1-operadpar}]{HopfPart1} in the symmetric case.
The result is the commutative Hopf algebra of rooted forests with unlabeled tails.

The action of the automorphisms is free and hence there is also the reduced version of the co-- and Hopf algebras.

\subsubsection{No tail version}
For this particular operad, there is the construction of forgetting tails and we can use the construction of \cite[\S\ref{P1-freeampsec}]{HopfPart1}.
In this case, we obtain the Hopf algebras of planted planar forests without tails or
the commutative Hopf algebra of rooted forests, which is called $\H_{CK}$. On the Feynman category level, this construction is done using universal operations of \S \ref{universalpar} applied to the decorated Feynman categories, see \S\ref{decopar}, $\Surj_{dec \, \O}$ and $\Surj_{<,dec \, \O}$ for the (non--$\Sigma$) operad of leaf labeled trees.

\subsection{Colored operads and their dual co--operads}
\label{coloroppar}
Colored operads are partial operads, where the compositions are allowed if the colors match. More precisely, fix a set of colors $C$ then, a colored operad is a collection $\O(c_1\kdk c_n;c)$ with $c,c_i\in C$ and there is a composition
$\gamma:\O(c_1\kdk c_n;c)\otimes \O(c^1_1\kdk c_1^{m_1};c_1)\odo
\O(c_n^1\kdk c_n^{m_n};c_n)\to \O(c^1_1\kdk c_1^{m_1} \kdk c_n^1\kdk c_n^{m_n};c)$.
\begin{rmk}
\label{cooprmk}
The dual of a colored operad is a co--operad. Indeed, one only decomposes into factors that are {\it a priori} composable.
\end{rmk}
In the Feynman category terms, cf.\ \cite[\S2.5]{feynman}, these are $\ops$ for a Feynman category whose vertices are rooted corollas together with a morphisms of the flags to $C$. This is technically a decoration, see \S\ref{decopar}. One then restricts to those morphisms whose underlying ghost graphs have the property that both flags of any ghost edge have the same color, see \S\ref{graphexpar}. Coloring is a form of decoration and restriction as discussed in \cite[\S6.4]{decorated}.
Such a colored operad also furnishes an (enriched) Feynman category whose vertices are
$c\in C$ and whose basic morphisms are given by the $\O(c_1\kdk c_n;c):
\amalg_{i=1}^n c_i\to c$. The $c_i$ are called input colors and $c$ is the output color.

\begin{prop}
The strict  Feynman categories based on colored operads as above
are  non--negative with respect to length, if $\O(\emptyset,c)=\emptyset$
and are non--positive w.r.t.\ length if $\O(c_1\kdk c_n,c)=\emptyset$ for $n>0$.

The construction of bi--algebras and conditions for Hopf algebras coincide
in both formulations under this translation to the bi--algebras and Hopf algebras obtained from the dual co--operads.

\qed
\end{prop}
This includes the examples of Goncharov and Baues in the form discussed in\cite[\S\ref{P1-overlapex}]{HopfPart1}.

\begin{rmk}
If the co--operads are not in $\Set$ the construction and statement are analogous, see \S\ref{hyppar} below.
\end{rmk}

\subsubsection{Bi-- and Hopf algebras from categories, sequences and Goncharov's Hopf algebra}

\begin{prop}
\label{catprop}
Every category defines a colored operad and
thus we obtain an associated bi--algebra and possibly a Hopf algebra from any category.

This recovers the Hopf algebra of Goncharov's and Baues' construction when considering a complete groupoid.
\end{prop}

\begin{proof}
Consider $X_n=N_n(\C)$ the simplicial object given by the nerve of a category.
Let $C=N_1(\C)=Mor(\C)$ be the set of colors. Then there is a colored operad defined by
$\O(\phi_1\kdk \phi_n,\phi)=\{
X_0\stackrel{\phi_1}{\to} \dots\stackrel{\phi_n}{ \to} X_n\in N_n(\C):\phi=\phi_n\circ\cdots\circ\phi_n\}$.
If $X_0\stackrel{\phi_1}{\to} \dots\stackrel{\phi_n}{ \to} X_n$
is an $n$ simplex and $X_{i-1}=Y_0\stackrel{\psi_1}{\to} \dots \stackrel{\psi_m}{\to}
Y_m=X_i$ is an $m$ simplex,
with $\psi_m\circ\dots\circ\psi_1=\phi_{i}$, then we can compose to
$$X_0\stackrel{\phi_1}{\to} \dots \stackrel{\phi_{i-1}}{\to}X_{i-1}=Y_0\stackrel{\psi_1}{\to} \dots \stackrel{\psi_m}{\to} Y_m=X_{i+1}\stackrel{\phi_{i+1}}{ \to}
\dots\stackrel{\phi_n}{ \to} X_n$$

If the underlying category is a complete groupoid, so that there is exactly one morphism per pair of objects, then any $n$--simplex can simply be replaced by the word $X_0\cdots X_n$
of its sources and targets.
\end{proof}

Notice that in the complete groupoid case $\V=\{X_0X_1\}$ is the set of words of length $2$ not $1$.
This explains the constructions of Goncharov involving multiple zeta values, but also polylogarithms \cite{Gont}, and the subsequent construction of Brown.
This matches our discussion in \cite[\S\ref{P1-simplicialpar}]{HopfPart1}  and  \S\ref{brownpar}.

\subsubsection{Marking angles by morphisms}
\label{anglepar}
Considering the simplicial object given by the nerve of a category $N_{\bullet}(\C)$ yields a particularly nice example of the duality between marking angles vs.\ marking tails. An $n$--simplex  $X_0\stackrel{\phi_1}{\to} {X_1}\dots \stackrel{\phi_n}{\to} X_n$ naturally gives rise to a decorated corolla,
where the angles are decorated by the objects and the leaves are decorated by the morphisms, viz.\ the colors, see Figure \ref{angle}. The operation that the corolla represents is the composition of all of the morphisms to get a morphism $\phi=\phi_n\circ \dots \circ \phi_0:X_0\to X_n$, viz.\ the output color.
\begin{figure}
    \centering
    \includegraphics[width=0.8\textwidth]{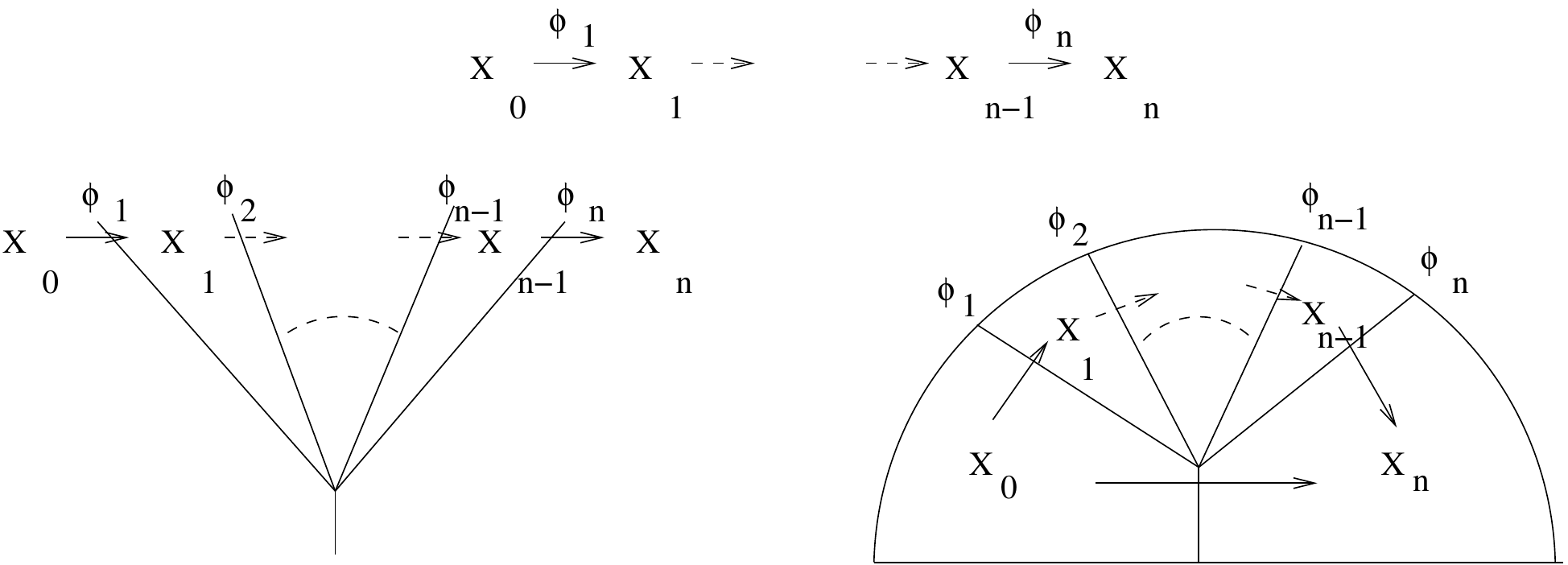}
    \caption{Marking a corolla by a simplex in $N_{\bullet}(\C)$. The morphisms decorate the ends of the tree, while the objects decorate the angles which correspond to the marks on the half circle}
    \label{angle}
\end{figure}
If there is a single morphism between any  two objects
either one of the markings, tail or angle, will suffice to give a simplex. In the general case, one actually needs both the markings.
The angle/tail duality is related to Joyal duality; see \cite[Appendix \ref{P1-Joyalapp}]{HopfPart1} and \S\ref{simpfeypar} below.

\subsection{Graph examples}
\label{graphexpar}
The basic graph Feynman category is $\GG=(\Crl,\Agg,\imath)$, defined in detail in \cite[\S2.1]{feynman}, see also Appendix \ref{graphsec}.
The notion of graph that is used is that of \cite{BorMan}. The BM--graphs from a category, and $\Agg$ is the full subcategory whose objects are aggregates of corollas.
A corolla is a graph with one vertex and no edges, and an aggregate is a disjoint union of these. $\Crl$ is the groupoid of corollas and their isomorphisms, and $\imath$ is inclusion.
To each BM--morphism $\phi:X\to Y$ between two aggregates $X$ and $Y$, one can associate a ghost graph $\gh(\phi)$, see Appendix \S\ref{ghostgraphpar}. A morphism $\phi$ is roughly a graph $\gh(\phi)$, together with an identification of the vertices of $\gh(\phi)$ with the source aggregate and an identification of $\gh(\phi)/E_{\gh(\phi)}$ with the target aggregate, see \cite[\S2.1]{feynman} and the appendix for details.  Different varieties of graph based Feynman categories are then given by restricting or decorating graphs in  a manner respected by composition (see the appendix and the examples in \S5).
A first new example  is that of collections of 1-PI graphs, which we call the Broadhurst--Connes--Kreimer Feynman category.

Without going into all the details, we wish to note the following facts, cf.\ \cite[\S2.1, \S5 and Appendix A]{feynman}.
\begin{enumerate}
\item The morphisms of $\Agg$ are generated by
(a) isomorphisms,
(b) simple edge contractions,
(c) simple loop contractions,
(d) simple mergers.

A simple edge contraction glues two flags from two different corollas together to form an edge and then contracts the edge leaving a corolla. A simple loop contraction does the same with the exception that the two flags come from the same corolla. A simple merger identifies two distinct corollas by identifying their vertices and keeping all flags. The ghost graph keeps track of which flags have been glued together to form edges that are subsequently contracted.
\item The subcategory generated by only the first three classes defines the wide subcategory $\Agg^{ctd}$ of $\Agg$ and the Feynman category $\GG^{ctd}=(\Crl,\Agg^{ctd}, \imath)$. The ghost graphs of morphisms in $(\Agg\downarrow \Crl)$ are connected.
\item A ghost graph does not define a morphism uniquely, but the isomorphisms class $[\phi]$ for $\phi\in \Agg^{ctd}$ is fixed by the ghost graph $\gh(\phi)$. In $\Agg$ the same is true for the morphisms in $(\Agg\downarrow \Crl)$. The ghost graph also fixes the source of a morphism and the target up to isomorphism.
\item Composition of morphisms corresponds to graph insertion. In particular in $\Agg^{ctd}$, $\gh(\phi\circ\psi)=\gh(\phi)\circ\gh(\psi)$ where $\gh(\phi)$ has connected components corresponding to the vertices of $\gh(\psi)$:
$\gh(\phi)=\amalg_{v\in V(\gh(\psi))}\gh_v(\phi)$. The insertion inserts $\gh_v(\phi)$ into the vertex $v$ of $\gh(\psi)$ -- using extra data provided by the morphisms to identify the flags aka.\ half--edges adjacent to $v$ with the tails aka.\  external legs of $\gh_v(\phi)$.
An example is given in Figure \ref{phi3fig}.
\item For a basic morphism in $\Agg$, i.e.\ one whose target is a corolla, the ghost graph determines the isomorphism class. In $\Agg^{ctd}$ the isomorphism class of any morphism is determined by its ghost graph and vice--versa.
\item If $\phi=\phi_0\circ\phi_1$ then (a) $\gh(\phi)=\gh(\phi_0)\circ\gh(\phi_1)$ as above, but also  (b) $\gh(\phi_1)\subset \gh(\phi)$  is (not necessarily connected) subgraph  and
$\gh(\phi_0)\simeq\gh(\phi)/\gh(\phi_1)$.
The corresponding factorization of a morphism in $(\F\downarrow \V)$ is
\begin{equation}
    \xymatrix{
    X\ar[r]^{\phi}\ar[d]_{\phi_1}&\ast\\
    Y\ar[ur]_{\phi_0}&\\
    }   \quad
\raisebox{-5mm}{\tabular{r} and on\\iso classes
    \endtabular }
         \xymatrix{
    X\ar[r]^{\gh(\phi)}\ar[d]_{\gh(\phi_1)}&\ast\\
    Y\ar[ur]_{\gh(\phi_0)=\gh(\phi)/\gh(\phi_1)}&\\
    }      \end{equation}
where $\gh(\phi_1)a$ is a subgraph, $\gh(\phi)/\gh(\phi_0)$ is sometimes called the co--graph and $*$ is the  residue in the physics nomenclature;
see Figure \ref{phi3fig} for an example.
\end{enumerate}

\begin{figure}
    \centering
    \includegraphics[width=.8\textwidth]{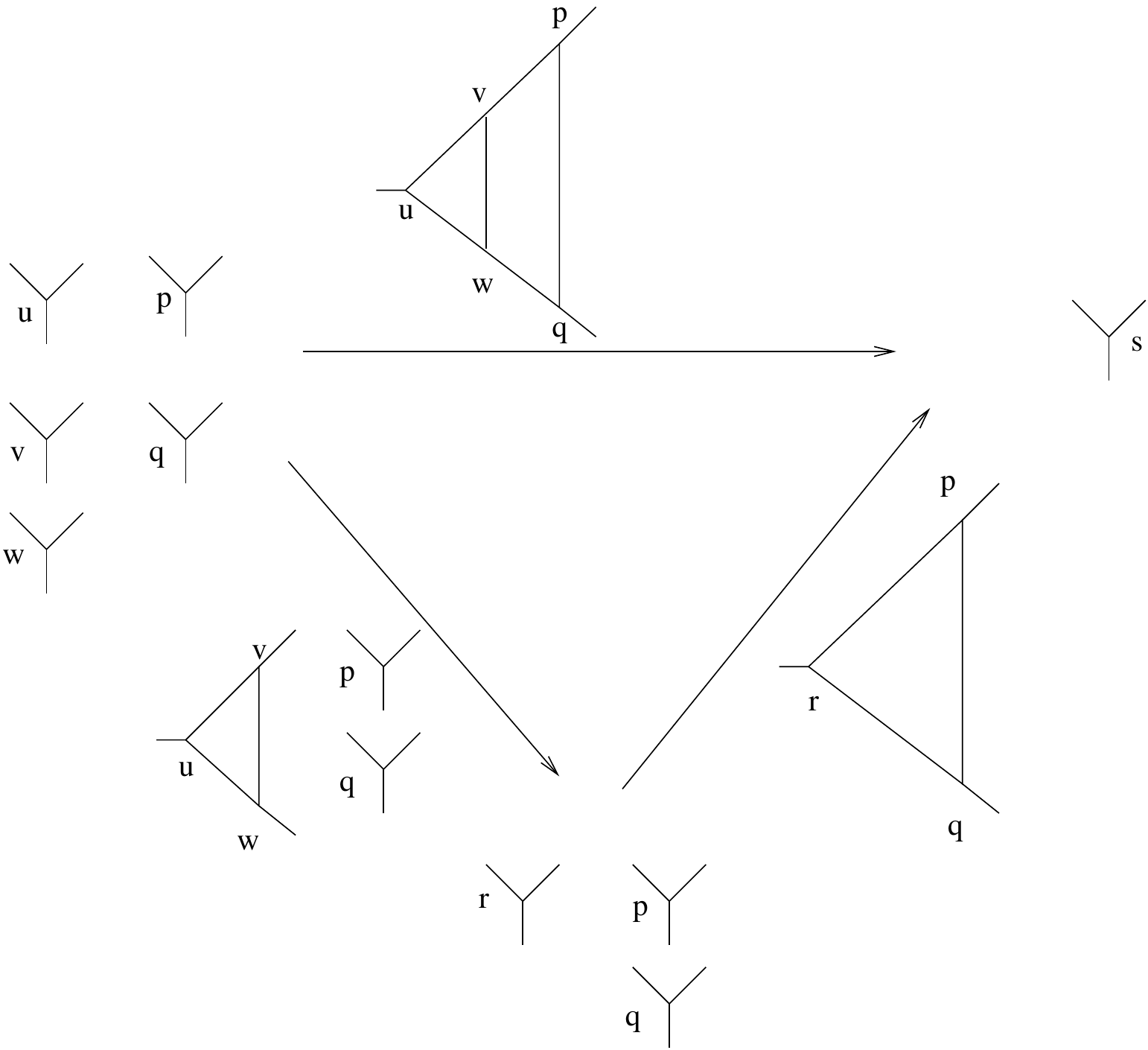}
    \caption{\label{phi3fig}An example of a factorization
in three--valent graphs aka.\ $\phi^3$. Alternatively the top graph $\gh$ results from inserting the left graph $\gh_1$, which has three components, into the right graph according  $\gh_0$ to the vertex map $\{u,v,w\}\mapsto r, p\mapsto p, q\mapsto q$, viz.\ $\gh=\gh_0\circ\gh_1$, or the left graph is a subgraph of the top graph $\gh_1\subset \gh$ and the right graph $\gh_0$ is the quotient graph. $\gh_0=\gh/\gh_1$}
  \end{figure}
\begin{lem}
In $\Agg^{ctd}$ the action of $Aut(Y)$ on $Hom(X,Y)$ is free.
\end{lem}
\begin{proof}
We use the terminology and formalism of Appendix \ref{graphsec}.
A morphism is given by $\phi=(\phi_V,\phi^F,\imath_\phi)$ the action of $\sigma=(\sigma_V,\sigma^F,id)$ with both $\sigma^F$ and $\sigma_V$ bijections. Now $(\sigma\circ\phi)^F= \phi^F\circ\sigma^F$, which already implies the result as $\sigma^F$ is an injection.
\end{proof}

\begin{cor} In $\Agg^{ctd}$ the action on the middle space is a free action on the decompositions.
\end{cor}

\begin{prop}
On isomorphism classes $\gh$ in $\Agg^{ctd}$.
\begin{equation}
\label{graphcoprodeq}
\Delta^{iso}(\gh)=\sum_{\gh_1\subset \gh}\gh/\gh_1\otimes \gh_1=\sum_{\gh_1\subset\gh}\gh_0\otimes \gh_1
\end{equation}
Here $\gh$ is the isomorphism class $\gh=[\phi]=\gh(\phi)$ and $\gh_1=\gh(\phi_1)$ is a subgraph, which corresponds to the isomorphism class of a decomposition $[(\phi_0,\phi_1)]$ where then necessarily $\gh(\phi_0)=\gh(\phi)/\gh_1$. Moreover if $\gh$ is connected, so is $\gh_0$. --- both are isomorphism classes in $(\Agg^{ctd}\downarrow \Crl)$.

\end{prop}
\begin{proof}
Given $\phi$ its isomorphism type is fixed by $\gh(\phi)$. We can choose a representative for $\phi$. The claim is that the factorizations up to the action on the middle space are given precisely by the subgraphs. Indeed, given any subgraph, there is surely a factorization. We have to show that there is exactly one term per sub--graph. For this, we ``enumerate everything''.  That is the flags, vertices, ghost edges etc. to fix the morphism. For a given subgraph there is a putative morphism, whose source is fixed and whose target is fixed up to isomorphism. This ambiguity is exactly compensated by the action on the middle space. This actions is free, on the decompositions and does not change the subgraph and hence every subgraph appears exactly once in the sum.
\end{proof}

Note that the multiplicities of the graphs appearing on the right side can be higher than one as the same graph may appear in several ways yielding different subgraphs, but isomorphic quotient graphs.

\begin{ex}
\label{subgraphex}
We consider the morphism of Figure \ref{graphcoprodfig}.
Each edge leads to a factorization. One such factorization is given in Figure \ref{graphcompfig}. If we write $\phi=\phi_0\circ\phi_1$, we note  that $im(\phi_1^F)=\{1,1',2,2'\}$. If $(\hat \phi_0,\hat{\phi}_1)$ is the decomposition with respect to the other edge $\{2,2'\}$, then $im(\phi_1^F)=\{1,1',3,3'\}$. Since this invariant under the $Aut (*_{1,2,3,4})$ action  $(\hat \phi_0,\hat{\phi}_1)$ and $(\phi_1,\phi_0)$ are  not equivalent under this action. But the abstract one edge graphs are the same.  $\gh(\hat{\phi_i})=\gh(\phi_i):i=0,1$. To be clear, different subgraphs, same underlying graph.
\end{ex}
\begin{figure}
    \centering
    \includegraphics[width=0.8\textwidth]{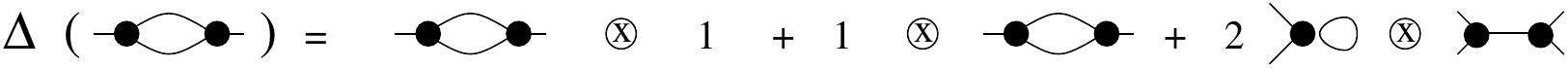}
    \caption{The co--product of a graph. The factor of 2 is there, since there are two distinct subgraphs ---given by the two distinct edges--- which give rise to two factorizations whose abstract graphs coincide}
    \label{graphcoprodfig}
\end{figure}

\begin{figure}
    \centering
    \includegraphics{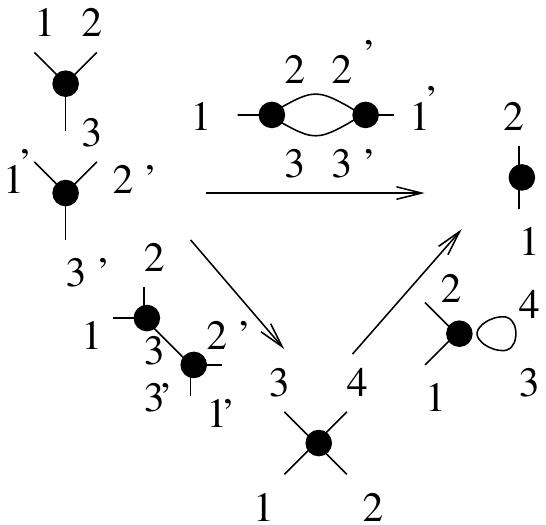}
    \caption{One decomposition. To fix $\phi$ we specify $\phi^F(1)=1,\phi^F(2)=1'$, to fix $\phi_1$, we set $\phi_1^F(1)=1,\phi_1^F(2)=1,\phi_1^F(3)=1',\phi_1^F(4)=2'$ and to fix $\phi_0$ we fix $\phi_0^F(1)=1,\phi_0^F(2)=2$. There is no choice for the vertex maps and the involution is the one given by the ghost graph.
    }
    \label{graphcompfig}
\end{figure}

\subsubsection{Graph based Feynman categories and Connes--Kreimer Hopf algebras}

If we look at the Feynman category $\GG=(\Crl,\Agg,\imath)$ then we obtain the core Hopf algebra
of graphs of Connes and Kreimer \cite{CK}.
The standard ``refined'' grading is as follows. Usually there will be no mergers involved, and edge contractions and loop contractions are assigned degree $1$. The co-radical grading is by word length in the elementary morphisms, that is the grading above,  which coincides with the number of edges.

There are several restrictions and decoration that one can put on the graphs to obtain sub--categories indexed over the category $\GG$. Here indexing means that there is a functor surjective on objects, cf.\ \cite[\S1.2.7]{feynman}.
Decoration is used in the technical sense described below \S\ref{decopar}; see \cite[\S6.4]{decorated} for standard decorations of graphs.

The key thing is that the extra structures and restrictions respect the concatenation of morphisms, which boils down to plugging graphs into vertices. Examples of this type  furnish bi-- and Hopf algebras of
 of
 modular graphs, non--\SSigma{} modular graphs, trees, planar trees, etc..

\subsubsection{1--PI graph version} A not so standard example, at least for mathematicians,  are 1--PI graphs.
Recall that a  connected 1--PI graph is a connected graph that stays connected, when one severs any edge and in general a 1--PI graph is a graph whose every component is 1--PI.
A nice way to write this is as follows \cite{brown}: Let $b_1(\Gamma)$ be the first Betti number of the  graph $\Gamma$. Then a graph is 1--PI if for any proper subgraph
$\gamma\subsetneq \Gamma$: $b_1(\gamma)<b_1(\Gamma)$. This means that 1-PI for non-connected graphs any edge cut
 decreases the first Betti (or loop) number by one.

It is easy to see that the property of being 1--PI is preserved under composition in $\GG$, namely, blowing up a vertex of a 1-PI graphs into a 1-PI graph leaves the defining property (namely connectivity) invariant. Hence, we obtain a bi--algebra of 1--PI graphs. It is almost connected and after amputation, one obtains the Hopf algebra used in physics.

A decorated version of this is Brown's Hopf algebra of motic graphs, see below \S \ref{brownpar}.

\subsection{Decoration: $\FF_{dec\O}$}
\label{decopar}
This type of modification was defined in \cite{decorated} and further analyzed in the set--based case in \cite{ddec}.
It gives a new Feynman category $\FF_{dec\O}$
from a pair $(\FF,\O)$ of a Feynman category $\FF$ and a strong monoidal functor $\O:\F\to \C$.
The objects of $\FF_{dec\O}$ are pairs $(X,a_X), a_X\in \O(X)$ ($a_X\in Hom_{\E}(\unit,\O(X))$ in the general enriched case).
The morphisms  from $(X,a_X)$ to $(Y,a_Y)$ are those $\phi\in Hom_{\F}(X,Y)$ for which $\O(\phi)(a_X)=a_Y$.
For a morphism $\phi$, we let $s(\phi)$ and $t(\phi)$ be the source and target of $\phi$ .

\begin{lem}
The morphism of $\FFdeco$ are pairs $(\phi,a_{s(\phi)}), a_{s(\phi)}\in \O(s(\phi))$.
If $\FF$ is decomposition finite, then so is $\FFdeco$. If $\FF$ is Hopf, then so is $\FFdeco$.
\end{lem}

\begin{proof}
By descriptions, any morphism $(X,a_X)\to (Y,a_Y)$ is a lift of a morphism
$\phi:X\to Y$. Such a lift exists if $a_Y=\O(a_X)$.
Thus fixing $\phi:X\to Y$ and $a_X\in \O(X)$, there is a unique morphism $(\phi,a_X):(X,a_X)\to (Y,\O(\phi)(a_X))$ and these are all the morphisms.
Since the source and $\phi$ fix the target:
\begin{equation}
\Delta((\phi,a_X))= \sum_{(\phi_0,\phi_1):\phi=\phi_0\circ\phi_1}(\phi_0,\O(\phi_1)(a_X)) \otimes (\phi_1,a_X)
\end{equation}
This equation also shows that the Hopf property is preserved.
\end{proof}

\subsubsection{Brown's motic Hopf algebras}
\label{brownpar}
In \cite{brown} a generalization of 1--PI graphs is given.
In this case there are the decorations of (ghost) edges
of the morphisms by masses and the momenta; that is, maps $m:E(\Gamma)\to \R$ and $q:T(\Gamma)\to \R^d\cup \{\emptyset\}$.
Notice that these are decorations in the technical sense of \cite{decorated} as well. For this, the decoration operad  is $\O(*_S)=\{S\mapsto \R^d\amalg \R\}$, so that each flag is either decorated by a momentum, or a mass. As a functor, under edge/loop contractions the decoration on the contracted flags is simply forgotten. This gives a decoration of all the flags of the ghost graph. This is not the end result, but we further to restrict to those morphisms whose ghost graphs have the {\it same} decoration for any two flags that make up a ghost edge, which is the standard procedure, cf.\ \cite[\S6.4]{decorated}. This results in the ghost edges being decorated by masses.
The masses carry over onto the new edges upon insertion.
Note that the flags that carry momenta are never glued

A subgraph $\gamma$ of a graph $\Gamma$
is called momentum and {\it mass spanning (m.m.)} if it contains all the tails and all the edges with non--zero mass. This means that as a ghost graph its target has corollas, whose flags are labeled with $0$ mass except possibly one corolla whose flags are labeled with all the external momenta.
A graph $\Gamma$ is called {\em motic} if for any m.m.\ subgraph $\gamma$: $b_1(\gamma)<b_1(\Gamma)$. This condition invented by Brown generalizes 1--PI.
It is again stable under composition, i.e.\ gluing graphs into vertices as can be readily verified, see  \cite[Theorem 3.6]{brown}.

After taking the quotient and amputating all tails marked by momenta, we see that the one vertex ghost graph becomes identified with the empty graph and
we obtain the Hopf algebra structure of \cite[Theorem 4.2]{brown}.

\subsection{Simplicial structures and Feynman categories}
\label{simpfeypar}
In this section, we consolidate and expand the construction of \cite[\S\ref{P1-simplicialpar}]{HopfPart1} in the setting of Feynman categories.

\subsubsection{The Feynman category $\FFinSet$ and variations}

The basic non--trivial Feynman category with trivial $\V$, is $\FFinSet=(\triv,\FinSet,\imath)$  where $\FinSet$, the category of finite sets and set maps with monoidal structure given by the disjoint union $\amalg$. The functor $\imath$ is given by sending $*$ to the atom $\{*\}$.
The equivalence between $\SS$ and $Iso(\FinSet)$ is clear as $\SS$ is the skeleton of $Iso(FinSet)$. Condition (iii) holds as well. Given any morphisms $S\to T$ between finite sets, we can decompose it using fibers as.
\begin{equation}
\label{finseteq}
\xymatrix
{
S \ar[rr]^{f}\ar[d]_{=}&& T\ar[d]^{=} \\
 \amalg_{t\in T} f^{-1}(t)\ar[rr]^{\amalg f_t}&&\amalg_{t\in T} \{*\}
 }
    \end{equation}
where $f_t$ is the unique map $f^{-1}(t)\to \{*\}$. Note that this map exists even if $f^{-1}(t)=\emptyset$.
This shows the condition (ii), since any isomorphisms of this decomposition must preserve the fibers.

$\FFinSet$ has the Feynman subcategories
$\Surj=(\triv,\surj,\imath)$, where the maps are restricted to be surjections and $\Inj=(\triv,\inj,\imath)$
where the maps are restricted to be injections.
This means that none of the fibers are empty or all of the fibers are empty ,respectively.

In the non--$\Sigma$ case, a basic example is $\FFinSet_<=(\triv,\FinSet_<,\imath)$, where $\FinSet_<$ is the category of ordered finite sets with order preserving maps  has as $\F$  the category of and with $\amalg$ as monoidal structure. The order of $S\amalg T$ is lexicographic, $S$ before $T$.
The functor $\imath$ is given by sending $*$ to the atom $\{*\}$.
 Viewing an order on $S$ as a bijection to $\{1,\dots,|S|\}$, we see that ${\bf N}_0$, the set $\N_0$ viewed as a discrete category (that is with only identity morphisms), is the skeleton of $Iso(\FinSet_<)$. The diagram \eqref{finseteq} translates to this situation and we obtain a non--$\Sigma$ Feynman category. The skeleton of Feynman category is the strict Feynman category $(\triv, \Delta_+,\imath)$, where $\Delta$ is the augmented simplicial category and  $\imath(*)=[0])$.
Restricting to order preserving surjections and injections, we obtain the Feynman subcategories $\Surj_<=(\triv,OS,\imath)$ and $\Inj_<=(\triv,OI,\imath)$.
We can also restrict the skeleton of $\FinSet_<$ given by $\Delta_+$ and the subcategories of order preserving surjections and injections.
See Tables \ref{table1} and \ref{table2}. In $\Delta_+$ the image of $*^{\otimes n}$ under $\imath^{\otimes}$ will be the set $\underline{n}$ with its natural order.

\begin{table}
\begin{tabular}{l|l|l}
$\FF$&$\F$ &definition\\
\hline
$\FFinSet$&$\FinSet$&Finite sets and set maps\\
$\Surj$&$\mathcal{S}urj$&Finite sets and surjections\\
$\Inj$&$\mathcal{I}nj$&Finite sets and injections\\
\end{tabular}
\caption{\label{table1} Set based Feynman categories Feynman categories. $\V=\underline{*}$ is trivial.}
\end{table}
\begin{table}
\begin{tabular}{l|l|l}
non-$\Sigma$ $\FF$&$\F$&definition\\
\hline
$\FFinSet_<$&$\FinSet_<$&Finite sets and order preserving maps. \\
$\Surj_<,$&$OS$&Ordered finite sets and ordered preserving surjections\\
$\Inj_<$&$OI$&Ordered finite sets and order preserving injections\\
$\Delta_+$&$\Delta_+$&Augmented Simplicial category, Skeleton of $\FinSet_<$\\
$\Int$&$OI_{*,*}$&Subcategory of $\Delta_+$ of double base--point preserving \\
&&injections
\end{tabular}
\caption{\label{table2} Set based non-$\Sigma$ Feynman categories. $\V=\underline{*}$ is trivial.}

\end{table}

\begin{ex}[Bi-- and Hopf--algebra structures]

$\FinSet$ and $\FinSet_<$ are not decomposition finite, but the restrictions to injections and surjections in the skeletal version are.
The bi--algebra structure on surjections is as follows: the basic morphisms are surjections $\pi_n:n\ta 1$ which can  be alternatively viewed
as corollas with $n$ inputs. In the non--sigma case, $\V$ is discrete and $\B=\B^{iso}$. We get
\begin{equation}
\Delta(\pi_n)=
\hspace{-10pt}
\sum_{1\leq k\leq n, f: (n,<)\ta (k,<)}
\hspace{-10pt}
\pi_k \otimes f=
\hspace{-12pt}
\sum_{1\leq k\leq n, (n_1,\dots,n_k):n_1\geq 1,\sum n_i=n} \hspace{-12pt}
\pi_k\otimes (\pi_{n_1}\odo \pi_{n_k})
\end{equation}
since an order preserving surjection is uniquely determined by the cardinalities of its ordered set of fibers.
In the Hopf algebra, we get
\begin{equation}
\Delta^{\H}(\pi_n)=\pi_n\otimes 1+1\otimes \pi_n +\sum_{1< k<n, (n_1,\dots,n_k):n_i>1,1<\sum n_i<n}
\pi_k\otimes (\pi_{n_1}\odo \pi_{n_k})
\end{equation}
as in the quotient $[id_1]=[1\ta 1]=1$ as well as its products.
This reproduces the example of corollas
\cite[Example \ref{P1-corrolladeltaex}]{HopfPart1}.

For the case of $\Surj$, we can use a skeleton for the isomorphism classes. The bi--algebra is then
\begin{equation}
\Delta([\pi_n])=
\hspace{-4pt}
\sum_{1\leq k\leq n, [f]: f:\ta (k,<)}
\hspace{-4pt}
[\pi_k] \otimes [f]=
\hspace{-10pt}
\sum_{1\leq k\leq n, \{n_1,\dots,n_k\}:n_1\geq 1,\sum n_i=n}
\hspace{-10pt}
\pi_k\otimes [\pi_{n_1}]\cdots[\pi_{n_k}]
\end{equation}

\begin{equation}
\Delta^{\H}(\pi_n)=[\pi_n]\otimes 1+1\otimes [\pi_n] +\sum_{1< k<n, \{n_1,\dots,n_k\}:n_i>1,1<\sum n_i<n}
[\pi_k]\otimes [\pi_{n_1}]\cdots [\pi_{n_k}]
\end{equation}

Note that this gives the same multiplicities as in Example \cite[\ref{P1-symmetryex}]{HopfPart1}.

\end{ex}

\subsubsection{The Feynman category of simplices, Intervals and  the Joyal dual of $\Surj_<$}
\label{joyal1par}
 As stated previously, there is a very interesting and useful contravariant
duality \cite{joyal} of subcategories of $\simpcat_+$
between  $\Delta$ and the  category $\Delta_{*,*}$,
 which are the endpoint preserving morphisms in $\Delta_+$.
It maps surjections $OS$ in $\simpcat$ to double base point preserving injections $OI_{*,*}$.

Thus the category $\Int^{op}$ is also a non--$\Sigma$ Feynman category with trivial $\V$.
One has to be careful with the monoidal structure: while in $\Delta$ the monoidal structure is disjoint union of small categories, for which
$[n]\otimes[m]=[n+m+1]$,  with unit $\emptyset=[-1]$.
The monoidal structure on $\Delta_{*,*}$ is the one defined e.g.\  in  Definition \cite[\ref{P1-doublebasedef}]{HopfPart1}, whose unit is $[0]$. We will denote
this tensor product by $\otb$, so that $[n]\otb[m]=[n+m]$ by identifying $n$ and $0$.

 $\Int=(\underline{*},OI_{*,*},\otb),\imath)$ is also a subcategory of the non--$\Sigma$ Feynman $\Inj_<$ category. The underlying objects of $\F$ are the natural numbers. To each $n$ we associate $[n]$, technically $\imath(*)=[1]$. For the morphisms, we have the identity $id_{[1]}$ in $Hom([1],[1])$, and one can check that indeed $id_{[1]}^{\otb n}=id_{[n]}$.

To get injections in $\Delta_+$, we only need to add one morphism:
$p:[0]\to [1]$ which we will call special. This generates all injections, cf.\ \cite[\S 2.10.3]{feynman}.
Any double--base point preserving injection from $[n+1]$ to $[m+1]$ in $\Delta_+$
is then represented by a tensor product of identities and special maps for the tensor product $\otimes$. This can be used to give a representation of the Feynman category $\Int$ in terms of generators and relations  in the sense of \cite[\S5]{feynman}. In particular, any double base point preserving injection can be written as $id\otimes p^{n_1-1}\otimes id \otimes p^{n_2-1}\odo p^{n_d-1} \otimes id:[d]\to [N]$, where $N=\sum_{i=1}^d n_i$ is the operadic degree, the length is $N-1$.
Let us introduce the notation $(1;0^{n_1-1},1,0^{n_2-1}\kdk 1, 0^{n_d-1};1)$ for this morphism. Where we think of $0^{n-1}=0,0\kdk 0$ as $n-1$ occurrences of $0$ indicating the elements in the target that are not hit.

Just as  surjections are generated by the unique maps $\underline{n}\ta \underline{1}$, so too are, dually,  the double base point preserving injections generated by the unique maps $[1]\to [n]\in Hom_{*,*}([1],[n])$.  These are the basic morphisms. In the notation above the unique double base point preserving injection $[1]\to [n]$ is $(1;0^{n-1};1)=(1;0 \kdk 0;1)$ with $n-1$ copies of $0$. It is given by $id\otimes p^{\otimes n-1}\otimes id$.
For example: $(1;0^{n-1},1)\otb (1;0^{m-1},1)=(1;0^{n-1},1,0^{m-1};1)=id\otimes p^{\otimes n-1} \otimes id  \otimes p^{\otimes m-1}\otimes id:[1]\otb[1]=[2]\to [n]\otb[m]=[n+m]$ is the morphism that sends $0\mapsto 0, 1\mapsto n, 2\mapsto n+m$.

In general
\begin{multline*}
(1;0^{n_1-1},1,0^{n_2-1}\kdk 1, 0^{n_d-1};1)=\\
 (1;0^{n_1-1};1)\otb (1;0^{n_2-1};1)\otb \cdots \otb (1; 0^{n_d-1};1)
\end{multline*}

The factorizations dual to the surjections $\underline{n}\ta\underline{k}\ta \underline{1}$, i.e.\ $[0]\to [k] \to [n]$ yields the co--product
\begin{multline}
\Delta(1;0^{n-1};1)=
\!\!\!\!\!\!\!\!\!
\sum_{\substack{k\geq 0\\(n_1\kdk n_k):\sum n_i=n)}}
\!\!\!\!\!\!\!\!\!\!\!\!
(1;0^{k-1};1)\otimes (1; 0^{n_1-1},1 \kdk 0^{n_k-1};1)=\\
\!\!\!\!\!\!\!\!\!
\sum_{\substack{k\geq 0\\(n_1\kdk n_k):\sum n_i=n)}}
\!\!\!\!\!\!\!\!\!\!\!\!(1;0^{k-1};1)\otimes \left(
(1;0^{n_1-1};1)\otb(1;0^{n_2-1};1)\otb\cdots\otb
(1;0^{n_k-1};1)\right)
\end{multline}

The Hopf quotient is then given by setting $id_1=(1;\emptyset;1)=1=id_\unit$.

\begin{rmk}
In terms of \cite[\S\ref{P1-gencooppar}]{HopfPart1} a multiplication is given by sending free tensor product $\boxtimes$ to $\otb$ ---and evaluating.
See \S\ref{picturepar}  for  pictorial representations. This corresponds to the equivalence in axiom \eqref{morcond} for Feynman categories by picking a functor from the free monoidal category realizing the equivalence. Identifying $\boxtimes$ with $\otimes$ explains the appearance of (op)-lax monoidal functors, see \S\ref{ncpar} and \cite[Proposition
\ref{P1-oplax}]{HopfPart1}.
\end{rmk}
\begin{rmk}
\label{depthrmk}
Note that the depth is the number of $1$s. Except for the interpretation as a lax monoidal functor, it is not clear how this is exactly related to the multi--zeta values and will be a field of further study.
A different encoding would be to use the symbol $(0;1\kdk n-1;n)$ for the unique double base point preserving injection $[1]\to [n]$. Then the formula becomes.

\begin{multline}
\Delta(0;1\kdk n-1;n)=
\!\!\!\!\!\!\!\!\!
\sum_{\substack{k\geq 0\\(j_1\kdk j_k):\sum j_i=n)}}
\!\!\!\!\!\!\!\!\!\!\!\!
(0;1,\dots, k-1;k)\otimes\\
\left(
(0;1\kdk j_1-1;j_1)\otb(0;1\kdk j_2-1;j_2)\otb\cdots\otb
(0;1\kdk j_k-1;j_k)\right)
\end{multline}

This is the basic structure of Goncharov's co--product, see e.g.\ \cite[\eqref{P1-gontcoprod}]{HopfPart1}, which only needs one more step of decoration, see \S\ref{decopar}, and especially \S\ref{seqpar}.
In this particular case, these are angle markings, see \S\ref{anglepar}. Further connections are given in  Example \ref{coopex} and \S\ref{semisimpex}.
\end{rmk}

\subsubsection{Pictorial representation}
\label{picturepar}

Pictorially, the surjection is naturally depicted by a corolla while
the injection is nicely captured by drawing an
injection as a half circle.  The use of half circles goes back to Goncharov, albeit he did not associate them to double base point preserving injections.
Joyal duality can then be seen
by superimposing the two graphical images. The superposition goes back to \cite{gangl}. The connection to Joyal duality is new.
 This duality is also that of dual graphs
on bordered surfaces.
This is summarized in Figure \ref{figure1}. Notice that in this duality,
the elements of $[n]$ correspond to the angles of the corolla and the elements of $\underline{n}$ label the leaves of the corolla.

This also explains the adding and subtraction of $1$ in the formulas for Joyal duality \cite[\eqref{P1-joyaleq}]{HopfPart1}.

\begin{figure}
    \centering
    \includegraphics[width=.3\textwidth]{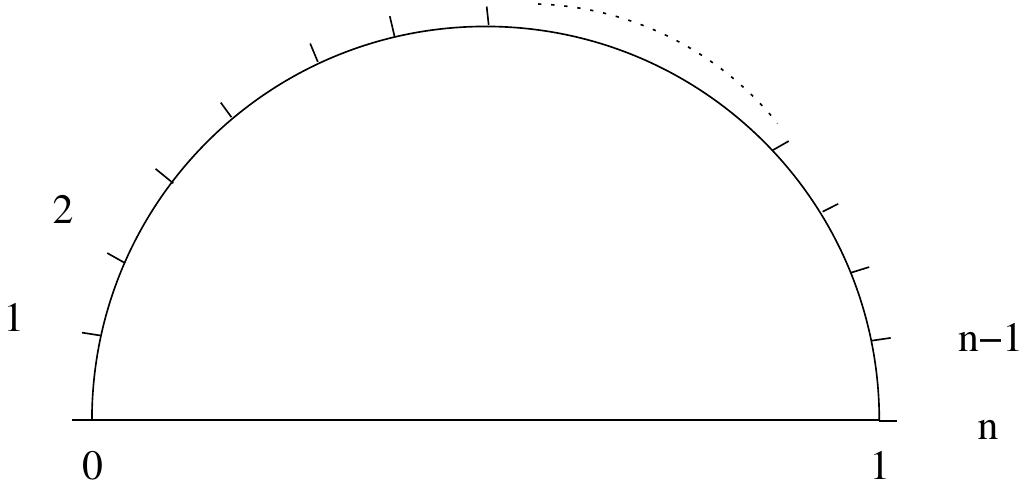}
    \includegraphics[width=.3\textwidth]{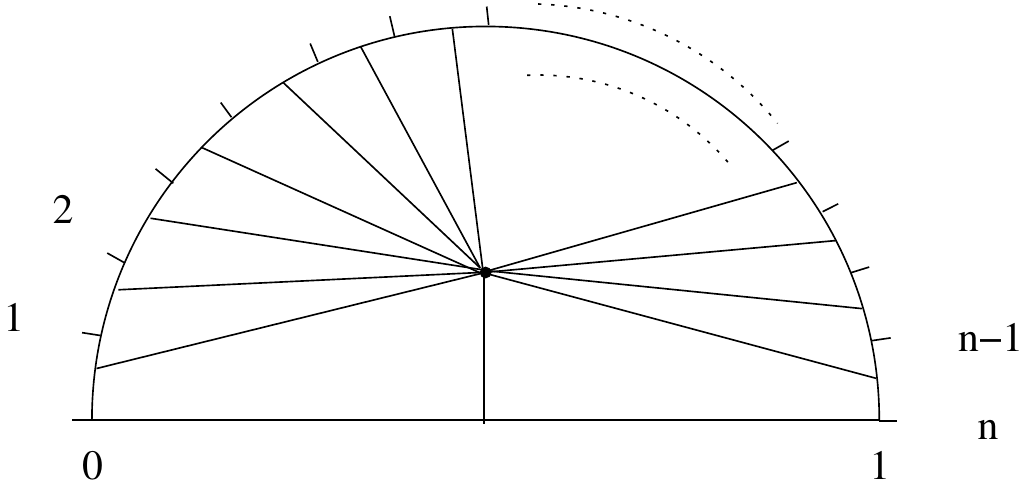}
    \includegraphics[width=.3\textwidth]{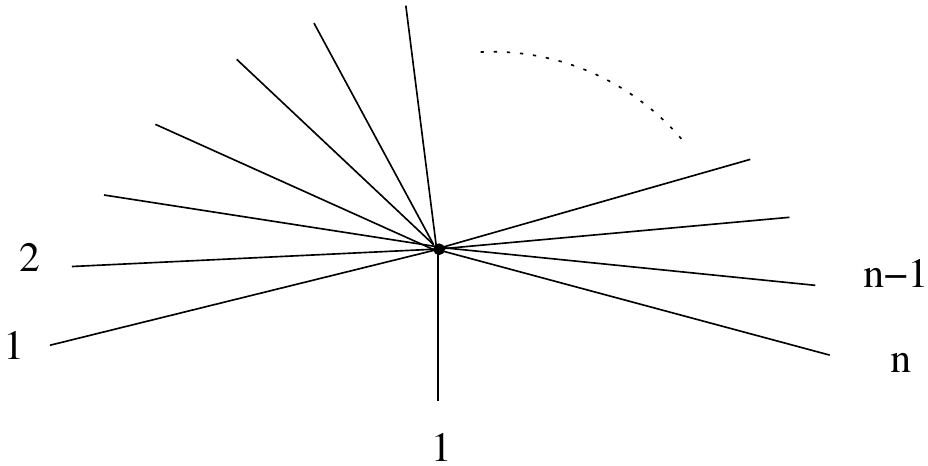}
    \caption{The interval injection $[1]\to [n]$  on the left, the surjection $\underline{n}\to \underline{1}$ on the right and
    and Joyal duality in the middle. Here reading the morphism upwards yields the double base point preserving injection, while reading it downward the surjection.}
    \label{figure1}
\end{figure}

For general surjections, the picture is the a forest of corollas and
a collection of half circles. The composition then is given by composing corollas to corollas
and by gluing on the half circles to the half circles by identifying the beginning and endpoints. This is exactly the map of combining  simplicial strings.
The prevalent picture for this  in the literature on multi--zetas and polylogs is by adding line segments as the base for the arc segments. This is pictured in Figure \ref{figure3}. The composition is then given by contracting the internal edges or dually erasing the internal lines. This is depicted in Figure \ref{figure2}.

\begin{figure}
\begin{tabular}{c}
     \includegraphics[height=.4in]{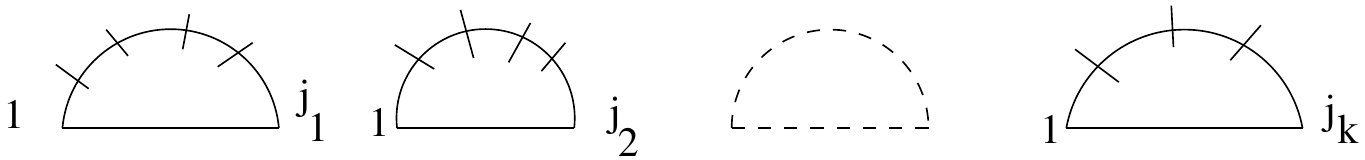}\\
  \includegraphics[height=.4in]{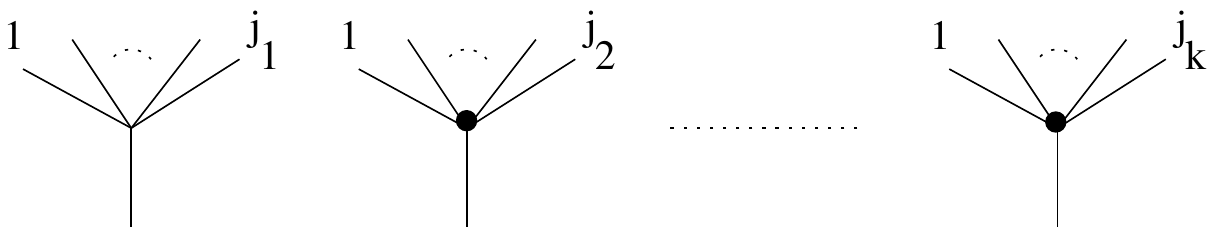}
      \end{tabular}
 \raisebox{-.4in} {\includegraphics[height=.8in]{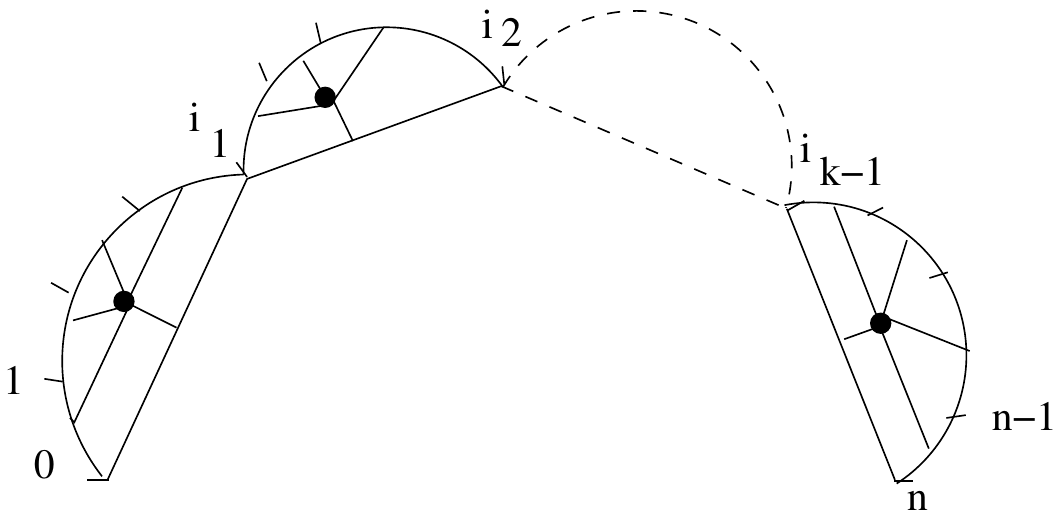}}

    \caption{The first step of the composition is to assemble  a collection of half discs or a forest into one morphism.
    This is pictured on the right.
    The $j$ and $i$ are related by $i_l=j_1+\dots j_k$. Notice that in the half disc assembly is glued at the $i_l$ essentially repeating them, while the forest assembly does not repeat.
    This also corresponds to an iterated cup product.}
    \label{figure3}
\end{figure}

\begin{figure}
    \centering
   \includegraphics[width=.3\textwidth]{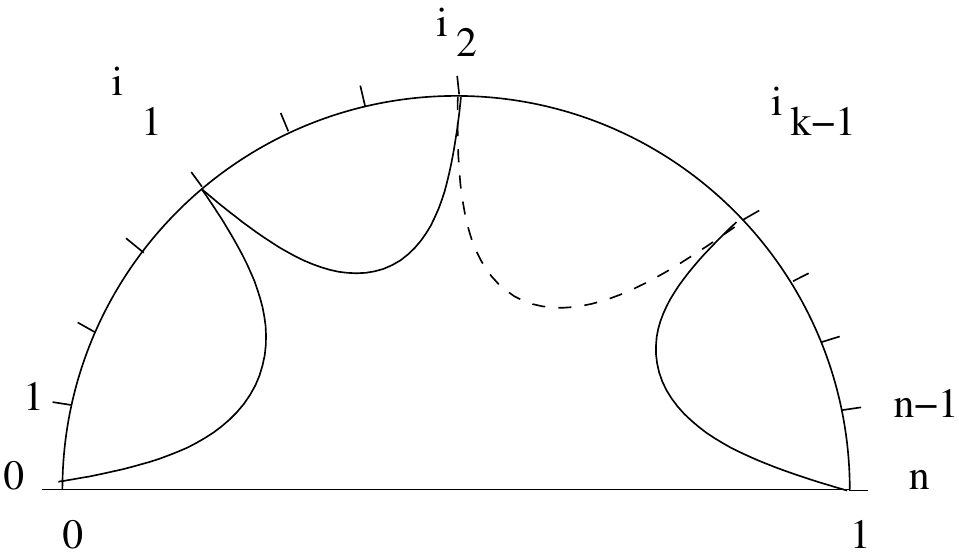}
    \includegraphics[width=.3\textwidth]{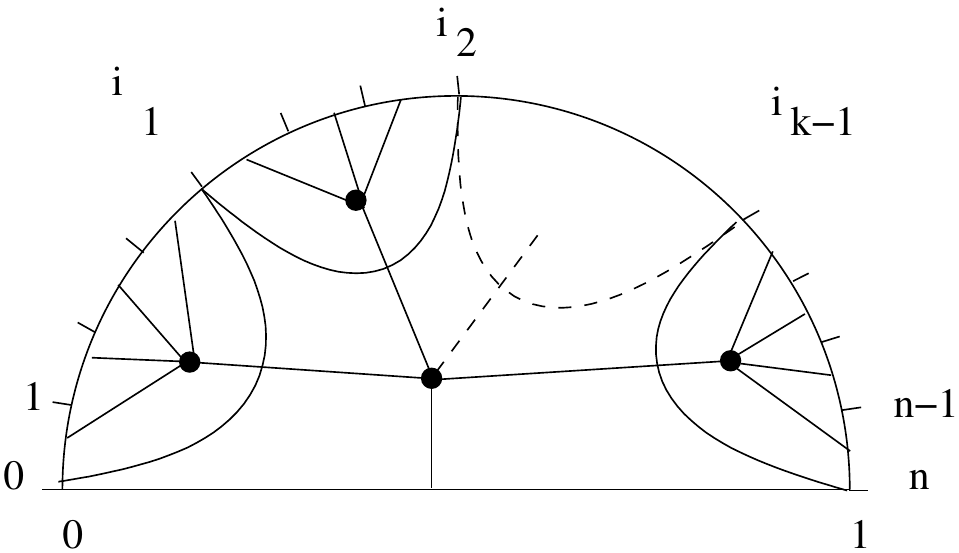}
    \includegraphics[width=.3\textwidth]{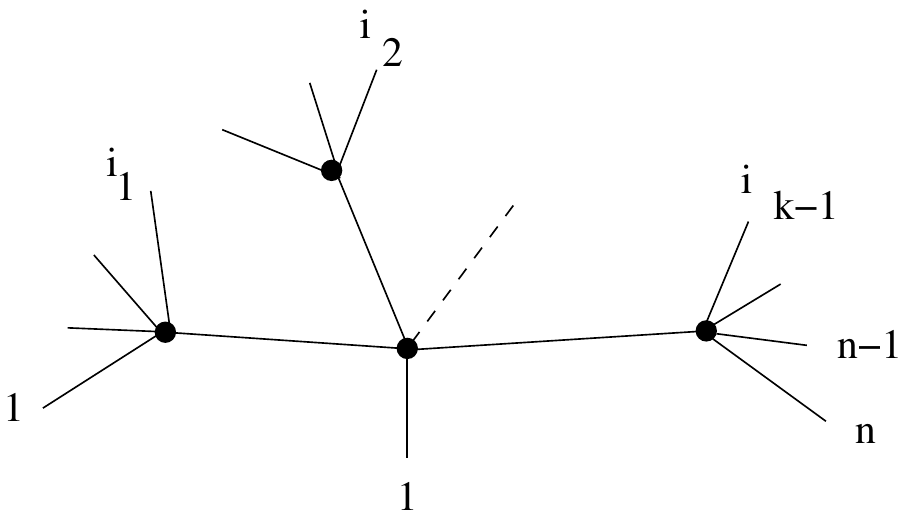}
    \caption{The second step of composition. For  half circles on the left, where we have deformed the half circles such that the outer boundary is now a half circle,
    corollas on the right and the duality in the middle.
    is done in Figure \ref{figure3}.
    The result of the composition is after the third step, which erases the inner curves or segments and in the corolla picture contracts the edges. The result is in Figure \ref{figure1}.}
    \label{figure2}
\end{figure}

We have chosen here the traditional way of using half circles. Another equivalent way would be to use polygons with a fixed base side. Finally, if one includes
both the tree and the half circle, one can modify the picture into a perhaps more  pleasing aesthetic by deforming
the line segments into arcs as is done in \cite[\S\ref{P1-simplicialpar}]{HopfPart1}, where also
one  explicit composition is given in all details, see \cite[Figure \ref{P1-fig:2347}]{HopfPart1}.

\subsubsection{Joyal duality in formulas}
\label{joyal2par} In this formulation Joyal duality is also easy to grasp. A double base point preserving injection is given by the symbol $(1;0^{n_1-1},1 \kdk 0^{n_d-1};1)=(1;w;1):[d]\to [N]$ as above. Where $1$ stands for $id$, $0$ for $p$ and $w$ is a word in these letters.
Now, the word $w$ in the middle is uniquely fixed by knowing the $n_i$. Vice--versa, given the $n_i$ there is a unique order preserving surjection $[N-1]\ta [d-1]$ whose fibers have cardinalities $n_1,\dots n_d$, that is $\pi_{n_1}\amalg \cdots\amalg \pi_{n_d}$.
This gives half of Joyal duality $OI_{*,*}([n+1],[m+1])\simeq OS([m],[n])$, where the bijections are natural. One can think of mapping the intervals in Figures \ref{figure1} and $\ref{figure2}$
surjectively from the top to the bottom.

To get the other direction note that any injection is given uniquely by a word $w$ as above. This will be a morphism $[d-2]\to [N-2]$. The corresponding surjection is a map $[N-1]\to [d-1]$. Now
since $OS=OS_{*.*}$ since all order preserving surjections have to preserve the base points, we have the second part of Joyal duality given by $OI(n,m)\simeq OS_{*,*}(m+1,n+1)$.
We also see the different monoidal structures. In the surjections, the monoidal structure is just $\amalg$;
for the half--circles, intervals, dually this means that they have to be joined at the base points,  see Figure \ref{figure2}.

\begin{rmk}
Using this logic, we also see that $OI(n,m)\simeq OS_{*,*}(m+1,n+1)= OS(m+1,n+1) \simeq OI_{*,*}(n+2,m+2)$, where the bijections are natural.
This is just the isomorphism which sends $w$ to $(1;w;1)$.
In $OI$, we just have the concatenation of words: $w_1\otimes w_2=w_1w_2$. Thus to get the right monoidal structure on $OI_{*,*}$,
we have to use $\otb: (1;w_1;1)\otb (1;w_2;1)= (1;w_1w_2;1)$.
Dually, we see that when combining the words $w_1w_2$ if there are occurrences of $0$ in the middle they will add as $0^{n_d-1}0^{m_1-1}=0^{n_d-1+m_1-1}$ which means that the two surjections will be merged using $\otb$
\end{rmk}

\subsubsection{Decorating with sequences}
\label{seqpar}
Consider the  Feynman category $\Delta_+$ and fix a set $S$. The contravariant functor $Seq:\Delta_+^{op}\to \Set$:
$[n]\to Hom([n],S)$ associates to $[n]$ the set of sequences $\{(a_0\kdk a_n):a_i\in S\}$ in $S$. Injections act as restrictions and surjections as repetitions.
The usual tensor product which takes the ordered sets $([n],[m])$
to the ordered set $[n+m+1]$ concatenates two sequences.
$(a_0\kdk a_n)\amalg(b_0\kdk a_m)=(a_0 \kdk a_n, b_0\kdk b_n)$ thus making $Seq$ into a monoidal functor.
For the Feynman category version, we can consider $Seq:\Delta_+\to \Set^{op}$.
In the decorated version, we have objects $([n],(a_0,\dots,a_n))$ which one can view as an interval with $n-1$ marked internal points (only their order matters), where the $i$--th point, counting both  internal and boundary points, is marked by $a_i$.

Restricting to $\Inj^{op}_{*,*}\simeq \Surj$, we see that alternatively, $Seq:\Surj\to \Set$.  In this setting is more natural to set  the image of $[n]$ to be $\underline{n}=\{1,\dots,n\}$. Now,  the decoration of $\underline{n}$ is by $(a_0,\dots,a_{n})$, that is $n+1$ elements, which we can take as an angle decorations.
The morphism $\pi_n:=\underline{n}\ta \underline{1}$ dual to $(1;0^{n-1};1):[1]\to [n]$ sends a decoration $(a_0\kdk a_n)$ to $(a_0,a_1)$, that is the two outer angle markings.
The graphical depiction of the morphism $\pi_n$ is a planar corolla as previously discussed, and the decoration by $(a_0,\dots,a_n)$ then naturally is carried by the angles, see Figure \ref{anglepar}.

 This gives rise to the colored operad structure of \S\ref{coloroppar} in the context of Goncharov, see also \cite[Example \ref{P1-overlapex}]{HopfPart1} as the decorations need to match and the category splits into connected components whose  final objects are $(\underline{1},(a_0,a_1))$. The monoidal structure in this setting  on $\underline{1}$ is addition while the monoidal structure on the decorations is $\otb$ due to the use of Joyal duality.

\begin{rmk}
 One can view the tensor product $\otb$ as a partial product, whose dual co--product  is the reason for the op--lax structure, namely the dual to the partial multiplication given by
 $\otb$.
\begin{equation}
 \Delta((a_0,\dots,a_n))=\sum_{i=1}^n(a_0\kdk a_i)\otb(a_{i}\kdk a_n)
\end{equation}
which is also the co--derivation discussed in \cite[\S\ref{P1-freeinfsec}]{HopfPart1} and an instance of the Alexander--Whitney map; see the next paragraph.
\end{rmk}

\subsubsection{Sequences as (Semi)--simplicial objects}
\label{semisimplicialpar}

In general, we can decorate $\Surj_{<}^{\boxtimes}$ with the semi--simplicial set $X_{\bullet}$, and then regard the decorated $\Surj_{<,dec X_\bullet}^{\boxtimes}$.
By definition, the objects will be $(\underline{n}_1\boxtimes\dots\boxtimes \underline{n}_{k},x_1 \odo x_k \in X_{n_1}\odo X_{n_k})$. Using the $B_+$ operator given by $\boxtimes\mapsto \otimes$ and the Alexander Whitney map, we re--obtain the simplicial results of \cite[\S\ref{P1-simplicialpar}]{HopfPart1}.

In order to read off the structure for Baues, we see that under the tensor product, we are looking at the tensor algebra on the simplicial objects $C_\bullet$, which is the underlying space of the bar--transform, when we regard everything as graded and use the usual shift $B(C_\bullet)=T C_\bullet[1]$.

Such a transition to the tensor algebra is also known as second quantization, cf.\ e.g.\ \cite{sq}.

\begin{ex} The decoration above can be viewed as a decoration by (semi)--simple objects. For this, we just consider $S$ to be the vertex set of an abstract simplicial complex $\mathscr S$. Then the sequences are simply the ordered simplices of $S$. Their linearization is
$C_*^{\rm ord}(S)$ the ordered simplicial chain complex. In this setting, we have a different tensor product. It corresponds to the tensor product of chain complexes, so that $(a_0\kdk a_n)\otimes (b_0\kdk b_m)\in C^{ord}_n(\mathscr{S})\otimes C^{ord}_m(\mathscr{S)}$.
This gives rise to the construction of Goncharov if we regard the $C_n$ as ungraded objects and use Joyal duality as in the previous paragraph.
In this context, the shuffle product \cite[\eqref{P1-shufflecond}]{HopfPart1} appears naturally, as the Eilenberg--Zilber map $C_n(\mathscr{S})\otimes C_m(\mathscr{S)}\to C_{n+m}(\mathscr{S})$.
\end{ex}
\subsubsection{Bi-- and Hopf--algebra from the decoration by the algebra of co--chains}
As $\Surj$-$\opcat_\C$ are algebras in $\C$, we can decorate by any algebra.

 Given a semi--simplicial set $X_\bullet$ then
$C^*(X_\bullet)$ can be made into a functor from $\Surj_<$, since it is an algebra. Namely, we assign to each $n$
the set $C^*(X_\bullet)^{\otimes n}\simeq C^*(X_{\bullet}^{\times n})$
and to the unique map $n\to 1$ the iterated cup product $\cup^{n-1}$.
After decorating, the objects become collections of co--chains,
and there is a unique map with source an $n$--collection of co--chains
and target a single co--chain, which is the iterated cup product.
Thus, one can identify the morphisms of this type with the objects.
Furthermore, the set of morphisms  possesses a natural structure of Abelian group. Dualizing this Abelian group, we get
the co--operad structure on $C_*(X_{\bullet})$ and the co--operad structure with multiplication on $C_*(X_{\bullet})^{\otimes}$ that coincides with the one considered in \cite[\S\ref{P1-simplicialpar}]{HopfPart1}.

The bi--algebra is almost connected if the 1--skeleton of $X_{\bullet}$ is connected. And after quotienting we obtain the same Hopf algebra structure from both constructions.

\subsubsection{Decorating with the bar/cobar complex}
\label{cobarpar}
Given an algebra $A$, we can decorate $\Surj_<$ directly.
Alternatively, we can decorate $\Surj_<$ with $BA$ as an $op$ decoration.
$OS\to \C^{op}$. Conversely given a co--algebra $C$, we can decorate with the algebra $\Omega C$. This leads to the construction of Baues.

\subsubsection{Relation to $\cup_i$ products}
\label{cupsec}
It is here that we find the similarity to the $\cup_i$ products also noticed by JDS Jones.
Namely, in order to apply $\cup^{n-1}$ to a simplex, we first use the Joyal dual map
$[1]\to [n]$ on the simplex. This is the map that is also used for the $\cup_i$ product.
The only difference is that instead of using $n$ co--chains, one only uses two.
To formalize this one needs a surjection that is not in $\simpcat$, but uses a permutation, and hence lives in $S\simpcat_+$. Here the surjection $\Surj$ gives rise
to what is alternatively called the sequence operad. Joyal duality is then the fact that one uses sequences and overlapping sequences.
The pictorial realizations and associated representations can be found in \cite{hoch2} and \cite{postnikov}. This is also related to the notion of discs in Joyal \cite{joyal}.
This connection will be investigated in the future.

In the Hopf algebra situation, we see that the terms of the iterated $\cup_1$ product coincide with  the second factor of the co--product $\Delta$. Compare Figure \ref{figure3}.

\subsection{Non--connected and free Feynman categories, simplicial objects and strings}
\label{ncpar}
Given a Feynman category $\FF$ there are two  associated Feynman categories  $\FF^{\boxtimes}, \FF^{nc}$ (nc stands for non--connected), which
have the properties
\begin{equation}
Fun_{\otimes}(\F^{\boxtimes},\C)=Fun(\FF,\C) \text{ and } Fun_{\otimes}(\F^{nc},\C)=Fun_{lax-\otimes}(\F,\C)
\end{equation}
see \cite[\S3.1,3.2]{feynman}.

\begin{rmk} [Co--operads with multiplication as an example of a $B_+$ operator]
\label{B+parII}
Using \S\ref{B+par} in the particular case of $\Surj_{<,\O}$,  $\mu=B_+:\Coop^{nc}\to \Coop$ is precisely which satisfies the compatibility equations for
a co--operad with multiplication and the conditions for
the unit and co--unit. This allows us to understand the constructions of \cite[\S\ref{P1-gencooppar}]{HopfPart1} which become natural in this definition.
\end{rmk}

\subsubsection{Simplicial objects and links to the simplicial construction of \cite{HopfPart1}}
\label{semisimpex}

By definition a simplicial object in $\C$ is(1) a functor $X_{\bullet}:\Delta^{op}\to \C$, and rewriting this, we see
that this is equivalent either (2) to a functor $X_{\bullet}^{op}:\Delta\to \C^{op}$ or (3) to a functor $X_{\bullet}^{Joy}: \Delta_{*,*}\to \C$.
The second and third descriptions open this up for a description in terms of Feynman categories and
our constructions of \cite[\S\ref{P1-simplicialpar}]{HopfPart1} mostly work with the last interpretation.

For (2) and (3) notice that in this interpretation $X^{op}_{\bullet}$ can be extended to a functor from $\Delta_+$, but it is not monoidal. However, it does give rise
to a functor $X^{op, \boxtimes}_{\bullet}\in\Delta_+^{\boxtimes}\text{-}\opcat_{\C^{op}}$,
or an element in $X^{Joy, \boxtimes}_{\bullet}\in \Delta_{*,*}^{\boxtimes}\text{-}\opcat_\C$.

In particular, the relevant constructions are on semi-simplicial objects in $\C$ which again described as
(1) a functor $X_{\bullet}:\Surj_<^{op}\to \C$, (2) to a functor $X_{\bullet}^{op}:\Surj_<\to \C^{op}$, equivalently  $X^{op, \boxtimes}_{\bullet}\in\Surj_<^{\boxtimes}\text{-}\opcat_{\C^{op}}$, or (3)  a functor $X_{\bullet}^{Joy}: \Delta_{*,*}\to \C$, equivalently an element $X^{Joy, \boxtimes}_{\bullet}\in  \Delta_{*,*}^{\boxtimes}\text{-}\opcat_\C$.

There is one more level of sophistication given by
\cite[Proposition \ref{P1-oplax}]{HopfPart1} which  one can rephrase as:
\begin{equation}
\Delta_+^\boxtimes=\Omega \Delta \text{ and }  \Delta_{*,*}^{nc}=\Omega \Delta
\end{equation}
which identifies simplicial strings as the free, receptively n.c.\ construction  by  using that in the correspondence $Fun(\F^{op},\C)\stackrel{1-1}{\leftrightarrow}Fun(\F,\C^{op})$ an oplax monoidal functors map to lax monoidal functors. What is intriguing is that although in (i) the original tensor product $\otimes$ is basically forgotten, in (ii) the dual tensor product
already is weakly respected by the functor and hence Joyal duality furnishes an intermediate step.
That is one only has to add the
 the oplax monoidal structure \cite[\S\ref{P1-stringsec}]{HopfPart1}, induced by the Alexander--Whitney map $X_{p+q}\to X_p\times X_q$, which is also represented in the monoidal structure of Joyal duality, as explained above, see also \S\ref{seqpar} for a concrete example.

The cubical realization of this using the functors $L$ of
\cite[ \S\ref{P1-stringsec}]{HopfPart1}.
in the more general context of $\FF^\otimes$ and $\FF^{nc}$ will be the subject of further investigation.

\subsection{Enrichment and operad based Feynman categories}
\label{enrichedpar}
\subsubsection{Enrichments, plus construction and hyper category $\FF^{hyp}$}
\label{hyppar}

The first construction is the plus construction $\FF^+$ and its quotient $\FF^{hyp}$ and its equivalent reduced version $\FF^{hyp,rd}$, see \cite{feynman}.
The main result of \cite[Lemma 4.5]{feynman} says that for any Feynman category $\FF$ there exists a Feynman category $\FF^{hyp}$ and the set of monoidal functors $\O:\F^{hyp}\to \E$ is in 1--1 correspondence
with indexed enrichments $\F_{\O}$ of $\F$ over $\E$.

For such an enrichment, one has $Obj(\F_\O)=Obj(\F)$ and
\begin{equation}
Hom_{\F_{\O}}(X,Y)=\coprod_{\phi\in Hom_{\F}(X,Y)}\O(\phi)
\end{equation}
And the additional condition that if $\phi$ is an isomorphism, then
$\O(\phi)\simeq \unit_\E$
 This generalizes the notion of hyper--operads of \cite{GKmodular}, whence the superscript $hyp$.

The compositions in $\F$ then give rise to compositions in $\F_\O$ for instance for the composition $\phi=\phi_1\circ\phi_0$, we get:
\begin{equation}
\O(\phi_1)\otimes \O(\phi_0)\to \O(\phi)
\end{equation}

The extra condition guarantees that one does not have to enlarge $\V$.
A slightly less strict restriction is that one regards $\O:\F^+\to \E$ which is suitably split. This is the $+gcp$ construction of \cite{Frep}.
Then there is $\FF_\O=(\V_\O,\F_\O,\imath)$ with $\F_\O$ as defined as above is still a Feynman category.
In this case one can enlarge $\V$ to $V_\O$ to include any invertible generators. In all of the cases $\FF_\O$ is a weak Feynman category \cite[Definition 1.9.1]{feynman}.
One possible enrichment is given for any $\O:\F^+\to \E$, such that $\O(\phi)$ is free of rank one.  This is called the free enrichment.

\subsubsection{Bootstrap}
There is the following nice observation. The simplest Feynman category is given by $\FF_{triv}= (\V=\triv,\F=\V^{\otimes},\imath)$ and $\FF_{triv}^{+}=\Surj^<$ \cite{Frep}. The underlying category are finite sets with surjections and orders on the fibers. This is indexed over $\Surj$ by forgetting the order on the fibers..
Going further, $\Surj^{+}=\operads$, the Feynman category for operads. Going back $\operads_{\V}$ gives $\Surj_{\O=\text{\it leaf labeled trees}}=\FF_{CK}$.
Decorating by simplicial sets, we obtain the three original examples from these constructions.
More details can be found in \cite{Frep}.

\subsubsection{Bi-- and Hopf algebras in the enriched case}
The bi-- and Hopf algebras in the enriched case use the formulation of the hereditary condition in the enriched setting. We refer the reader to
\cite[\S4]{feynman} for the rather technical details.
In the enriched setting, we will already postulate that the Hom spaces are Abelian groups. This means that the category $\E$ over which $\F$ is enriched, has a faithful functor to the category of Abelian groups. In this case, we say $\F_\O$ is $\Ab$ enriched over $\E$. We also assume that $\E$ has internal Homs and regard it as enriched over itself. A basic example is $\E=dg\Vect$.
Assume that $sk(\F)$ is small, $\F$ is strict.
In this case, we set $\B=\bigoplus_{X,Y}Hom^{\vee}_{sk(\F_\O})(X,Y)$, where $\vee$ is the dual in $\E$ given by $\check V=\inthom(V,\unit)$
and define the multiplication on $\B$ by $\otimes$. The unit is again $\id_\unit$.
For the co--multiplication $\Delta$, we take the dual of the composition $\circ$
\begin{equation}
\circ:\inthom_{\F_O}(Y,Z)\otimes \inthom_{\F_O}(X,Z)\to \inthom_{\F_O}(X,Y)
\end{equation}
as a morphism in $\E$.
\begin{equation}
\Delta: \inthom^{\vee}_{\F_O}(X,Y)\to \inthom^{\vee}_{\F_O}(Y,Z)\otimes \inthom^{\vee}_{\F_O}(X,Z)
\end{equation}
Again it is clear that $\eps(\phi)=1$ if $\phi=id_X$ and $\eps(\phi)=0$ if $\phi$ is not in a component $\unit$ corresponding to $id_X$ is a co--unit.
Similarly to \S\ref{feynmanpar}, assuming that the we can define $\B^{iso}$ by using co--invariants, assuming that these exist.

\begin{thm}
Let $\FF_\O$ be an indexed enriched Feynman category or more generally a weak Feynman category $\Ab$ enriched over a co--complete $\E$, which is enriched over $\Ab$, and $\F$ is factorization finite,
 then $\B^{iso}$ is a bi--algebra in $\E$. In the non--$\Sigma$ case, already $\B$ is a bi--algebra.
\end{thm}

\begin{proof}
The co--associativity and well--definedness of $\Delta$ follows from the condition the underlying $\F$ is factorization finite.
The hereditary condition (ii) is replaced by a co--end formula which can be written as, cf.\  \cite[Proposition 1.8.8,\S4]{feynman}:
\begin{multline}
\label{dayeq}
Hom_{\F}(\imath^{\otimes}\, \cdot\, ,X\otimes Y)
=\\
\int^{Z,Z'}Hom_{\F}(\imath^{\otimes} Z, X)\otimes
Hom_{\F}(\imath^{\otimes} Z' , Y)\otimes Hom_{\V^{\otimes}}( \,\cdot\,,Z\otimes Z')
\end{multline}
This formula precisely states that the space of morphisms into a product coincides with  the product of the space of morphisms, up to natural isomorphisms changing the intermediate $Z\otimes Z'$.

\begin{equation}
\xymatrix{W\ar[rr]^\phi\ar[dr]_\simeq&&X\otimes Y\\
&Z\otimes Z'\ar[ur]_{\phi_1\otimes \phi_2}&\\
}
\end{equation}
This directly implies that the bi--algebra equation holds on the level of isomorphism classes.

In the non--$\Sigma$ case, the isomorphism between $W$ and $Z\otimes Z'$ must be a product as well, as $Hom_{\V^\otimes}(W,Z\otimes Z')=Hom_{\V^\otimes}(W,Z)\otimes Hom_{\V^\otimes}(W,Z')$
so that the bi--algebra equation already holds on the level of morphism spaces.

\end{proof}

Again, define $\I=\la [id_X]-[id_Y]\ra$ and $\H=\B^{iso}/\I$, then $\H$ is a bi--algebra which may or not be Hopf.
\begin{df}
We call $\F_\O$ as above Hopf, if $\H$ has an antipode.
\end{df}
The discussion of criteria is analogous to that of the non--enriched case, by lifting all the notions from $\F$ to $\F_\O$. This is straightforward and will be omitted here.

\begin{ex}

The relevant example is that $\Surj^{hyp,rd}\simeq\operads_0$ that is the Feynman category for  operads without $\O(0)$ and whose $\O(1)$ is reduced.
Thus any such operad, that is a strong monoidal functor $\O:\operads_0\to \E$ gives rise to a Feynman category $\Surj_{\O}$
whose morphisms are determined by
\begin{equation}
Hom_{\Surj_{\O}}(n,1)=\O(n)
\end{equation}

In particular, if $f:S\twoheadrightarrow T$ then
$\O(f)=\bigotimes_{t\in T} \O(f^{-1}(t))$ since $f$ decomposes as one--comma generators
$f_t:f^{-1}(t)\twoheadrightarrow \{t\}$.

\end{ex}
\begin{rmk}
For operads with not necessarily reduced $\O(1)$, one can use the $\Surj^{+}=\operads$, and restrict to those functors
whose $\O(1)$ is split unital. See also \S\ref{trivVpar}
and \cite[\S3.4]{Frep}.
\end{rmk}

\subsubsection{Bi-- and Hopf algebras}
For concreteness, we will provide the details for the framework of twisted Feynman categories, in the specific case $\Surj_\O$.
In this language, the diagrams \cite[\eqref{P1-coprodsqeq}]{HopfPart1} identify certain summands in the co--product and on the coinvariants one is left with the channels.
Indeed in $\Surj$ decomposing $\pi_S:S\twoheadrightarrow \{*\}$ yields the sum
$S\stackrel{f}{\ta} T\stackrel{\pi_T}{\ta} \{*\}$. This is a typical morphism in $\Surj'$ from $\pi_S$ to $\pi_T$.

The composition operation on the twisted $\Surj_\O$:
$\gamma_f:\O(f)\otimes \O(T)\to \O(S)$, corresponding to the composition $\pi_T\circ f=\pi_S$ cf.\ \ref{hyppar}.
Dually, there is one summand of this type $\check\gamma_f$ in the co-product. We identify two such summands in the co--product under the action of the automorphism groups. This corresponds to the diagrams \cite[\eqref{P1-square}]{HopfPart1} which are the isomorphisms in $\Surj'$. Effectively, this means that fixing the size of $S$ and $T$ there is only one channel per partition of $S=S_1\amalg \dots \amalg S_k$ into fibers of $f$.

If one would like to include $\O(1)$ has more invertible elements, one has to enlarge
$\Surj$ by choosing the appropriate $\V$. In the case of Cartesian $\E$ this is   $Hom_{\V'}(1,1)=\O(1)^{\times}$.
This gives rise to extra isomorphisms and/or a $K$--collection, see \cite[2.6.4]{feynman}.
This means in particular that any operad gives rise to an enriched Feynman category whose morphisms are this operad.
The dual of the morphisms are then co--operads and the co--operadic and Feynman categorical construction coincide.

The non-\SSigma{} case is similar. For this one uses $\Surj_<$ and then obtains enrichments by non--\SSigma{} operads.
Thus again the co--operadic methods apply and yield the same results as the Feynman category constructions.
In this case, we see that $\B$ is the free tensor algebra on the basic morphisms, that is $\B=T\check{\O}(n)$ as in
\cite[\S\ref{P1-subsect:free}]{HopfPart1} and we obtain the following theorem, recovering all of \cite[\S\ref{P1-subsect:free}]{HopfPart1}.

\begin{thm} In both the symmetric case $\Surj_\O$ and non--symmetric case
$\Surj_{<\O}$, we obtain unital, co--unital bi--algebras $\B^{iso}$ respectively $\B$. If the quotient by the ideal $\I=\la id_1-id_\unit\ra$ is connected, we obtain a Hopf algebra. The latter is the case if there (a) there is no $\O(0)$ or (b) there is no $\O(i):i>1$,
 and $(\O(1),\id_1,\eps)$ is connected.
\qed
\end{thm}

\subsubsection{Enrichment over $\C^{op}$ and opposite Feynman category.}
Notice that we can regard functors $\FF\to \C^{\op}$ as co--operads.
In particular if we have a functor $\FF^{hyp}\to \C^{op}$, we get a Feynman category $\FF_{\O}$ enriched
over $\C^{op}$. This means that $\FF_{\O}^{op}$ is enriched over $\C$.

\begin{ex}
\label{coopex}
In particular,  $\O:\Surj^{hyp}=\operads_{0}\to \C^{op}$ a reduced co--operad in $\C$.
Then twisting with $\O$ gives us $\Surj_{<,\O}$ which is enriched in $\C^{op}$.
Taking the opposite we get $\Surj_{<,\O}^{op}$. The underlying category is $\Int$ enriched  by $\check\O$, where $\check\O$ is the co--operad in $\C$ corresponding to the operad in $\C^{op}$.
This means that the objects are the natural numbers $n$ and the morphisms are $Hom(1,n)=\check \O(n)$.
This is the enrichment in which the unique map in $Hom_{\Int}([1],[n])$ is assigned $\check\O(n)$ in the overlying enriched category $(\Int)_{\Coop}$.
\end{ex}

Putting all the pieces together then yields the following:

\begin{thm} Given a co--operad $\check{\O}$ that is given by a
functor  $\O:\operads_{0}\to \C^{op}$. Let $\B_{\check{\O}^{nc}}$ be the bi--algebra of
\cite[Example \ref{P1-subsect:free}]{HopfPart1}. And let $\B_{\Surj_{<,\O}^{op}}$
be the bi--algebra of the Feynman category discussed above then these two bi--algebra coincide.

Moreover if  $\Surj_{<,\O}$ is almost connected, the so is $\check{\O}$ and the corresponding Hopf algebras coincide.
\qed
\end{thm}

\begin{rmk}
This extends to split unital operads as split functors from $\Surj^{+gcp}=\operads$, and also to operads with $\O(0)$  (see \cite{Frep} for these notions).
\end{rmk}
This is another explanation of the relation between Joyal duality and the dual co--operad structure to a colored operad structure.

\subsection{Universal operations}
\label{universalpar}
It is shown that $\FF_{\V}$, which is given by $\F_{\V}=\colim_{\V}\imath$,  yields a Feynman category with trivial groupoid $\V_\V\simeq \underline{*}$.  This generalizes the Meta--Operad structure of \cite{del}. The result is again a Feynman category whose morphisms define an operad and
hence the free Abelian group yields a co--operad.

Moreover in many situations,  the morphisms of the category are
weakly generated \cite[\S6.4]{feynman} by a simple Feynman category obtained by ``forgetting tails''.
The action is then via a foliation operator as introduced in \cite{del}. In fact
there is a poly--simplicial structure  here, see also \cite{BataninBerger}.
In order to establish this, we recall that any operad under the equivalence established in \cite{feynman}[Example 4.12] can be thought of either an
enrichment of the Feynman category of sets and surjections or as a functor from the Feynman
category for operads to a target category, see also \S\ref{enrichedpar}. As the latter, we obtain universal operations through colimits, see paragraph \S 6 of \cite{feynman}. On the other hand, we obtain the colimits, in the same form as here,
via the construction in paragraph \S \ref{feynmanpar} below.

\begin{ex}
For the operad of leaf labeled trees, one can effectively amputate the tails using this construction. One obtains the
co--operad dual to the pre--Lie operad \cite{CL,del}.
That is $\H_{amp}$ is realized naturally from
a weakly generating sub--operad.
\end{ex}

\section{Summary and outlook}
\label{summarypar}

\subsection{Constructions}
We have shown that one can construct Bi--algebras that under checkable conditions yield Hopf algebras in the following related constructions, all of which exist in a symmetric and a non-$\Sigma$ version.

\begin{enumerate}
\renewcommand{\theenumi}{\roman{enumi}}
\item From a locally finite (unital) operad.
\item From a locally finite co--operad.
\item From a locally finite co--operad with multiplication.
\item From a simplicial object.
\item From a suitable Feynman category.
\item From a suitable Feynman category with a $B_+$ operator.
\end{enumerate}

Here the transition from (i) to (ii) is dualization. The construction (iii) replaces the free product with a chosen compatible one. Construction (i) and (ii)and (iv) are the special cases of (v) that appear as enriched Feynman categories, in particular enrichments of the Feynman categories of surjections or ordered surjections. The construction (iii) is a special case of the nc construction together with a $B_+$ operator. The construction (iv) can be seen as a special case of (i) and (ii), but there is an additional structure coming from the simplicial category and Joyal duality.

We also gave criteria when these constructions are functorial.
Furthermore, there are infinitesimal versions, which yield Brown's derivations in the (co)--operad case and are related to the generators for the Feynman categories and hence to master equations, cf.\ \cite{feynman,KWZ}.

\subsubsection{Main Results}
The main upshot is that in all these cases and the classical examples the co--algebra structure is simply the dualization of a partial product structure provided by concatenation in a category.
Furthermore, the bi--algebra equation in a general monoidal category is non--trivial and the conditions for Feynman categories are a sufficient condition
for it to hold.
The Hopf algebras of interest are connected and they are quotients of the natural bi--algebras. The quotient effectively identifies all the objects of mentioned categories.
\subsubsection{Further results and constructions}
Further results and constructions concern deformations, co--module structures, derivations/ infinitesimal structures and a detailed analysis of Joyal duality and its consequences among others.

\subsection{Connes--Kreimer}
There are several types of Connes--Kreimer Hopf algebras which appear as special examples. The tree--type Hopf algebras stem from the construction (i) while the graph--type algebras are examples of (iii).

\subsubsection{CK--forests}
The CK--forests in the planar and non--planar version can be viewed as coming from construction (i) for the (non-$\Sigma$) operads of
 leaf--labeled and leaf--labeled planar trees.
These are alternatively constructed using set--based Feynman categories with trivial $\V$, which can be thought of as indexed enrichments.
The amputated versions can be thought of as co--limits, either over a semi--simplicial system of maps, or via the universal operations in Feynman categories.

\subsubsection{Decorated/motic versions}
Using decorations and restrictions, one can obtain other versions, such as the motic versions from Brown, a 1-PI version and more generally colored and weighted versions.
\subsubsection{CK-graphs}
The full graph algebra is the basic example coming from a graphical Feynman category, i.e.\ one that is indexed over the Feynman category $\GG$, which is a full subcategory of the Borisov--Manin category of graphs. A main ingredient is that the ghost graph of a morphism fixes its isomorphism class.

Restricting and decorating allows one to give the ``core'' versions and the ``renormalization'' versions.

\subsection{Goncharov/Baues}
The Hopf algebra of Goncharov and its graded analogue that of Baues can be analyzed in the settings (i), (ii) and (v). In terms of (i) one is using a colored operad. There is an additional structure provided by Joyal duality, which we discussed and which links the constructions to lax--monoidal functors and the nc--construction. This duality also gives rise to the colored operad structure and explains the corolla vs. semi--circle representations.
Furthermore, using the cup product, there is a direct link to the decoration by an algebra.

We also found re--interpretations of the additional structures and restrictions of Goncharov and Baues.

\subsubsection{Goncharov multiple zeta values and polylogarithms}
In terms of (iv) taking the contractible groupoid on $0,1$ we obtain the construction of $\H_{Gon}$ for the multi--zeta values. If we take that with objects $z_i$, we obtain Goncharov's Hopf algebra for polylogarithms \cite{Gont}.

\subsubsection{Baues} This is the case of a general simplicial set, which however is
1-connected. We note that since we are dealing with graded objects, one has to specify that one is in the usual monoidal category of graded $\Z$--modules whose tensor product is given by the Koszul or super sign.
The 1--connectedness is needed for the bi--algebra quotient to be Hopf.
To obtain the connection to double loop spaces, we furthermore need 2--connectedness.

\subsection{Simplicial}
In general, in the simplicial setting, we provided a bi--algebra structure which is Hopf if the simplicial set is 1--connected.
We could explain these constructions on several levels.
\begin{enumerate}
\item as derived from the fact that simplices form an operad.
\item through monoidal and lax--monoidal functors.
\item using Joyal duality.
\item using the fact that the simplicial category is a Feynman category. \item As an operadic enriched Feynman category.
\item As a decorated Feynman using the $\cup$ product as an algebra structure.
This gives the relationship to the iterated $\cup$ product. The symmetric version also give the relationship to iterated $\cup_i$ products.
\end{enumerate}

\subsection{Outlook}
We expect these results to be the basis of further work.
There will be a closer analysis of the role of the $B_+$ operator and its use inside the theory of Feynman categories as well as its Hopf-theoretic nature \cite{B+paper}.
It will also play a role in the truncation/blow--up of moduli spaces and outer space cells \cite{ddec} its sequel and \cite{KaufmannZuniga}.
There are further applications to the theory of Feynman categories, theoretical physics, number theory and algebraic geometry along the basic examples of this paper and {\em loc.\ cit.}. In particular, we will analyze and build upon the combinatorial invariants and analysis of Feynman graphs as put forth by the Kreimer group. Here the next steps are applying our general cubical structures \cite{feynman,ddec} to the understanding of  the Cutkosky rules.

\appendix

 \section{Graph Glossary}
\label{graphsec}

\subsection{The category of graphs}

Interesting examples of Feynman categories used in operad--like  theories are indexed over a Feynman category built from graphs.
It is important to note that although we will first introduce a category of graphs $\Graphs$, the relevant
Feynman category is given by a full subcategory $\Agg$  whose
objects are disjoint unions or aggregates of corollas. The corollas themselves
play the role of $\asts$.

Before giving more examples in terms of graphs it will be useful to recall some terminology.
A very useful presentation is given in \cite{BorMan}, slightly modified in \cite{feynman}, which we follow here.

\subsubsection{Abstract graphs}
An abstract graph $\G$ is a quadruple $\G=(V_{\G},F_{\G},$ $i_{\G},\del_{\G})$
of a finite set of vertices $V_{\G}$, a finite
set of half edges or flags $F_{\G}$,
an involution on flags $i_{\G}\colon F_{\G}\to F_{\G}; i_{\Gamma}^2=id$ and
a map $\del_{\G} \colon F_{\G}\to V_{\G}$.
We will omit the subscript $\G$ if no confusion arises.

Since the map $i$ is an involution, it has orbits of order one or two.
We will call the flags in an orbit of order one {\em tails} and denote the set of tails by $T_{\Gamma}$.
We will call an orbit of order two an {\em edge} and denote the set of edges by $E_{\Gamma}$. The flags of
an edge are its elements.
The function $\del$ gives the vertex a flag is incident to.
It is clear that the set of vertices and edges form a 1-dimensional CW complex.
The realization of a graph is the realization of this CW complex.

A graph is (simply) connected if and only if its realization is.
Notice that the graphs do not need to be connected. Lone vertices, that
is, vertices with no incident flags, are also possible.

We also allow the empty graph $\egr$, that is, the unique graph with $V=\varnothing$.
 It will serve as the monoidal unit.
\begin{ex}
A graph with one vertex and no edges is called a {\em corolla}.
Such a graph only has tails.
For any set $S$ the corolla $\crl_{p,S}$ is the unique graph with $V=\{p\}$ a singleton and $F=S$.

We fix the short hand notation $*_S$ for the corolla with $V=\{\ast\}$ and $F=S$.
\end{ex}

Given a vertex $v$ of a graph, we set
$F_v=\del^{-1}(v)$ and call it {\em the
flags incident to $v$}. This set naturally gives rise to a corolla.
The {\em tails} at $v$ is the subset of tails of $F_v$.

As remarked above, $F_{v}$ defines a corolla $\crl_{v}=\crl_{\{v\},F_v}$.

\begin{rmk}
The way things are set up, we are talking about (finite) sets, so
changing the sets even by bijection changes the graphs.
\end{rmk}

\begin{rmk}
\label{disjointrmk}
 As the graphs do not need to be connected, given two graphs $\G$ and $\G'$
we can form their disjoint union:\\
$$\G\sqcup\G'=(F_{\Gamma}\sqcup F_{\Gamma'},V_{\Gamma}\sqcup V_{\G'},
i_{\Gamma}\sqcup i_{\G'}, \del_{\G}\sqcup \del_{\G'})$$

One actually needs to be a bit careful about how disjoint unions are defined.
Although one tends to think that the disjoint union $X\sqcup Y$
is strictly symmetric, this is not the case. This becomes apparent
if $X\cap Y\neq\emptyset$. Of course
there is a bijection $X\sqcup Y \stackrel{1-1}{\longleftrightarrow}Y\sqcup X$.
Thus the categories here are symmetric monoidal. It is also but not strict symmetric monoidal, since there
technically $X\amalg (Y\amalg Z)$ is not equal to $(X\amalg Y)\amalg Z)$.
This is important, since we  consider functors into other not necessarily strict monoidal categories.

Using Mac Lane's theorem it is, however, possible to make a technical construction that makes the monoidal
structure (on both sides) into a strict symmetric  monoidal structure
\end{rmk}

\begin{ex}
An {\em aggregate of corollas} or aggregate for short is a finite disjoint union of corollas, that is,
a graph with no edges.

Notice that if one looks at $X=\bigsqcup_{v\in I} \ast_{S_v}$ for some finite index set $I$ and some finite sets of flags $S_v$,
 then the set of flags is automatically the disjoint
union of the sets $S_v$. We will just say
just say $s\in F_X$ if $s$ is in some $S_v$.
\end{ex}
\subsubsection{Category structure; Morphisms of Graphs}

\begin{df}\cite{BorMan}\label{grmor}
Given two graphs $\G$ and $\G'$, consider a triple $(\phi^F,\phi_V, i_{\phi})$
where
\begin{itemize}
\item [(i)]
$\phi^F\colon F_{\G'}\hookrightarrow  F_{\G}$ is an injection,
\item [(ii)] $\phi_V\colon V_{\G}\twoheadrightarrow V_{\G'}$ and $i_{\phi}$ is a surjection and
\item [(iii)] $i_{\phi}$
is a fixed point free involution on the tails of $\G$ not in the image of $\phi^F$.
\end{itemize}

One calls the edges and flags that are
not in the image of $\phi$ the contracted edges
and flags. The orbits of $i_{\phi}$ are called \emph{ghost edges} and denoted by $E_{ghost}(\phi)$.

Such a triple is {\it a morphism of graphs} $\phi\colon \G\to \G'$ if

\begin{enumerate}
\item The involutions are compatible:
\begin{enumerate}
\item An edge of $\G$
is either a subset of the image of $\phi^F$ or not contained in it.
\item
If an edge is in the image of $\phi^F$ then its pre--image
is also an edge.
\end{enumerate}
\item $\phi^F$ and $\phi_V$ are compatible with the maps $\del$:
 \begin{enumerate}
\item Compatibility with $\del$ on the image of  $\phi^F$:\\
\quad If $f=\phi^F(f')$ then
$\phi_V(\del f)=\del f'$
\item Compatibility with $\del$ on the complement of the image of  $\phi^F$:\\
 The two vertices of a ghost edge in $\G$ map to
the same vertex in $\G'$ under $\phi_V$.
\end{enumerate}

 \end{enumerate}

If the image of an edge under $\phi^F$ is not an edge,
we say that $\phi$ grafts the
two flags.

The composition $\phi'\circ \phi\colon\G\to \G''$
of two morphisms $\phi\colon\G\to \G'$ and $\phi'\colon\G'\to \G''$
is defined to be  $(\phi^F\circ \phi^{\prime F},\phi'_V\circ \phi_V,i)$
where $i$ is defined by its orbits viz.\ the ghost edges.
Both maps $\phi^{F}$ and $\phi^{\prime F}$ are injective,
so that the complement of their concatenation is in bijection
with the disjoint union of the complements of the two maps.
We take $i$ to be the involution whose
orbits are the union of the ghost edges of $\phi$ and $\phi'$
under this identification.
\end{df}

\begin{rmk}
A {\em na\"ive morphism} of graphs $\psi\colon\G\to \G'$
is given by a pair of maps  $(\psi_F\colon F_{\G}\to F_{\G'},\psi_V\colon V_{\G}\to V_{\G'})$
compatible with the maps $i$ and $\del$ in the obvious fashion.
This notion is good to define subgraphs and automorphisms.

It turns out that this data {\em is not enough} to capture all the needed aspects
for composing along graphs.
For instance it is not possible to contract edges with such a map or graft
two flags into one edge. The basic operations of composition in an operad
viewed in graphs is however exactly grafting two flags and then contracting.

For this and other more subtle aspects one needs the more involved definition above
which we will use.
\end{rmk}
\begin{df}
We let $\Graphs$ be the category whose objects are abstract graphs and
whose morphisms are the morphisms described in Definition \ref{grmor}.
We consider it to be a monoidal category with monoidal product $\sqcup$ (see
Remark \ref{disjointrmk}).
\end{df}

\subsubsection{Decomposition of morphisms}
Given a morphism $\phi\colon X\to Y$ where $X=\bigsqcup_{w\in V_X} \ast_w$ and $Y=\bigsqcup_{v\in V_Y}\ast_v$
are two aggregates, we can decompose $\phi=\bigsqcup \phi_v$ with $\phi_v\colon X_v\to \ast_v$ where $X_v$ is the sub--aggregate $\bigsqcup_{\phi_V(w)=v}*_w$, and $\bigsqcup_v X_v=X$. Here $(\phi_v)_V$ is the restriction of $\phi_V$ to $V_{X_v}$. Likewise   $\phi_v^F$ is the restriction of
 $\phi^F$ to $(\phi^{F})^{-1}(F_{X_v}\cap \phi^F(F_Y))$. This is still injective.
Finally $i_{\phi_v}$ is the restriction of $i_{\phi}$ to $F_{X_v}\setminus \phi^F(F_Y)$.
These restrictions are possible due to the condition (2) above.

\subsubsection{Ghost graph of a morphism}
\label{ghostgraphpar}
The following definition introduced in  \cite{feynman} is essential.
The underlying ghost graph of
a morphism of graphs $\phi\colon\G\to \G'$ is the graph
 $\gh(\phi)=(V(\Gamma),F_{\Gamma},\hat \imath_{\phi})$ where $\hat \imath_{\phi}$ is $i_{\phi}$ on the complement of
$\phi^F(\Gamma')$ and identity on the image of  flags of  $\Gamma'$ under $\phi^F$.
The edges of $\gh(\phi)$ are called the ghost edges of $\phi$.

\subsection{Extra structures}
\subsubsection{Glossary}
This section is intended as a reference section.

Recall that  an order of a finite set $S$ is a bijection $S\to \{1,\dots ,|S|\}$.
Thus the group $\SS_{|S|}=Aut\{1,\dots,n\}$ acts on all orders.
  An  orientation of  a finite set $S$ is an
equivalence class of orders, where two orders are equivalent if they are
obtained from each other by an even permutation.

All the following definitions in Table \ref{treetable} are standard.

\begin{table}
\begin{tabular}{l|l}
A tree& is
a connected, simply connected graph.\\
{A directed  graph $\G$}& is a graph together with  a map $F_{\G}\to \{in,out\}$\\
&such that the two flags of each edge are mapped\\
&to different values. \\

A rooted tree& is a directed tree such that each vertex has exactly\\& one ``out'' flag.\\
A {ribbon or fat graph} &is a graph together with a cyclic order on each of \\&the
sets $F_{v}$.  \\
A planar graph& is a ribbon graph that can be embedded
into the\\
& plane such that the induced cyclic orders of the \\
&sets $F_v$ from the orientation of the plane  \\
&coincide with the chosen cyclic orders.\\

A planted planar tree&is a rooted planar tree together with a \\
&linear order
on the set of flags incident to the root.\\
An oriented graph& is a graph with an orientation on the set of its edges.\\
An ordered graph& is a graph with an order on the set of its edges.\\
A $\gamma$ labeled graph&is a graph together with a  map $\gamma:V_{\Gamma}\to \N_0$.\\
A b/w graph&is a graph $\G$ with a map $V_{\G}\to \{black,white\}$.\\
A bipartite graph& is a b/w graph whose edges connect only \\
&black to white vertices. \\
A $c$ colored graph& for a set $c$ is a graph $\G$ together with a map $F_{\G}\to c$\\
&s.t.\ each edge has flags of the same color.\\
A connected 1--PI graph& is a connected graph that stays connected, \\
&when one severs any edge.\\
A 1--PI graph&is a graph whose every component is 1--PI.
\end{tabular}
\caption{\label{treetable} Nomenclature for Graphs}
\end{table}

\subsubsection{Remarks and language}\mbox{}
\begin{enumerate}

\item Planar means that the graph can be and up to isotopy is embedded into the plane.

\item  In a directed graph one speaks about the ``in'' and the ``out''
edges, flags or tails at a vertex. For the edges this means the one flag of the edges
 is an ``in'' flag at the vertex. In pictorial versions the direction
is indicated by an arrow. A flag is an ``in'' flag if the arrow points to the vertex.

\item A rooted tree is  taken to be a tree with a marked vertex. Note that necessarily a rooted tree as described above has exactly  one ``out'' tail. The unique vertex whose ``out'' flag is not a part of an edge is the root vertex.
The usual picture is obtained by deleting this unique ``out'' tail.

\item A planted planar tree induces a linear order on all sets $F_v$, by declaring the first
flag to be the unique outgoing one.
Moreover, there is a natural order on the edges, vertices and flags
given by its planar embedding.
\end{enumerate}

\subsubsection{Category of directed/ordered/oriented graphs.}
\label{ordsec}
\begin{enumerate}

\item
Define the category of directed graphs $\Graphs^{dir}$ to be the category whose
objects are directed graphs. Morphisms are morphisms $\phi$ of the underlying graphs,
which additionally satisfy that $\phi^F$ preserves orientation of the flags and the $i_{\phi}$ also
only has orbits consisting of one ``in'' and one ``out'' flag, that is the ghost graph is also directed.

\item

The category of edge ordered graphs $\Graphs^{or}$ has as objects graphs with an order on the edges.
A morphism is a morphism together  with an order $ord$  on  all of the edges of the ghost graph.

The composition of orders on the ghost edges is as follows.
$(\phi,ord)\circ \bigsqcup_{v\in V}(\phi_v,ord_v):=(\phi \circ \bigsqcup_{v\in V}\phi_v,ord\circ\bigsqcup_{v\in V}ord_v)$ where the order on the set of all ghost edges, that is $E_{ghost}(\phi)\sqcup \bigsqcup_vE_{ghost}(\phi_v)$,
is given by first enumerating  the elements of $E_{ghost}(\phi_v)$ in the
order $ord_v$ where the order of the sets  $E(\phi_v)$ is given by
the order on $V$, i.e. given by the explicit  ordering of the tensor product in  $Y=\bigsqcup_v \ast_v$.\footnote{Now we are working with ordered
tensor products. Alternatively one can just index the outer order
by the set $V$ by using \cite{TannakaDel}}
 and then enumerating the edges of $E_{ghost}(\phi)$ in their order $ord$.

\item The oriented version $\Graphs^{or}$ is then obtained by passing from orders to equivalence classes.

\end{enumerate}

\subsubsection{Basic Feynman categories/operads}

The Feynman category $\GG=(\Crl,\Agg,\imath)$: $\Crl$ is the groupoid of corollas with isomorphisms.
$\Agg$ is the full subcategory of graphs whose objects are aggregates and $\imath$ is the inclusion.
$\GG^{ctd}=(\Crl,\Agg^{ctd},\imath)$ is the sub--Feynman category whose basic morphisms have connected ghost graphs.
$\CCyclic$ is the sub--Feynman category whose basic morphisms have trees as ghost graphs.

$\operads$ is the restriction of a decorated Feynman category. The decoration of $\CCyclic$ is by assigning the flags of a vertex $in$ and $out$.
And the restriction is that there is only one $out$ per vertex and the ghost graph is a directed graph.

\subsubsection{Non--$\Sigma$ versions/planar structures}
\label{planarsec}
Although it is hard to write down a consistent theory of planar graphs with planar morphisms, if not impossible,
there does exist a planar version of special subcategory of $\Graphs$.

We let $\Crl^{pl}$ have as objects planar corollas --- which simply means that there is a cyclic order on
the flags --- and as morphisms isomorphisms of these, that is isomorphisms of graphs, which preserve the cyclic order.
 The automorphisms of a corolla $\ast_S$ are then isomorphic to $C_{|S|}$, the cyclic group of order $|S|$.
Let $\CCyclic^{\neg \Sigma}$ be the  subcategory of aggregates of planar
corollas whose morphisms are morphisms of the underlying
 corollas, for which the ghost graphs in their planar structure
 induced by the source  is compatible
 with the planar structure on the target via $\phi^F$. For this we use the fact that the tails of a planar tree have a cyclic order.
 This is the Feynman category for non--$\Sigma$ cyclic operads. This is also a decorated Feynman category $\CCyclic^{\neg\Sigma}=\CCyclic_{dec \, \CycAss}$,
 where $\CycAss$ is the cyclic assocative operad, cf.\ \cite{decorated}.

Adding a direction one arrives an  $\Crl^{pl,dir}$ the groupoid of  planar corollas with one output and let $\operads^{\neg\Sigma}$ be the corresponding Feynman category.
This is again a decoration $\operads^{\neg\Sigma}=\operads_{dec \; Assoc}$, see {\it loc.\ cit} and is the Feynman category for non--$\Sigma$ operads.

\subsection{Flag killing and leaf operators; insertion operations}
\label{insertionsec}
\subsubsection{Killing tails}
We define the operator $\kill$, which removes all tails from a graph. Technically, $\kill(\G)=(V_{\G},F_{\G}\setminus
T_{\G},\del_{\Gamma}|_{F_{\G}\setminus T_{\G}},
\imath_{\G}|_{F_{\G}\setminus T_{\G}})$.

\subsubsection{Adding tails}
Inversely, we define the formal expression
$\leaf$ which associates to each $\G$ without tails the formal sum
$$\sum_n\sum_{\G':\kill(\G')=\G;F(\G')=F(\G')\sqcup \underline{n}}\G'$$, that is all possible
additions of tails where these tails are a standard set, to avoid isomorphic duplication.
To make this well defined, we can consider the series as a power series in $t$:
$$\leaf(\G)= \sum_n\sum_{\G':\kill(\G')=\G;F(\G')=F(\G')\sqcup \bar n}\G't^{n}$$

This is the foliage operator of \cite{KS,del} which was rediscovered in \cite{BBM}.

\subsubsection{Insertion}
Given graphs, $\G$,$\G'$, a vertex $v\in V_{\G}$ and an isomorphism $\phi$: $F_v\mapsto T_{\G'}$
 we define $\G\circ_v\Gamma'$ to be the graph obtained by deleting $v$ and identifying the flags of $v$ with
the tails of $\G'$ via $\phi$. Notice that if $\G$ and $\G'$ are ghost graphs
of a morphism then it is just the composition of ghost graphs, with the morphisms at the other vertices being the identity.
\subsubsection{Unlabeled insertion} If we are considering graphs with unlabeled tails, that is, classes $[\G]$ and $[\G']$
of coinvariants under the action of permutation of tails. The insertion naturally lifts as
$[\G]\circ[\G']:=[\sum_{\phi} \G\circ_v\G']$ where $\phi$ runs through all the possible isomorphisms of two fixed lifts.

\subsubsection{No--tail insertion}
\label{nolabcompsec}
If $\G$ and $\G'$ are graphs without tails and $v$ a vertex of $v$, then we
define $\G\circ_v\G'=\G\circ_v {\rm coeff}(\leaf(\G'),t^{|F_v|})$, the (formal) sum of graphs where $\phi$ is
one fixed identification of $F_v$ with $\overline{|F_v|}$.
In other words one deletes $v$ and grafts all the tails to all possible positions on $\G'$.
Alternatively one can sum over all $\del:F_{\G}\sqcup F_{\G'}\to V_{\G}\setminus v \sqcup V_{\G'}$ where
$\del$ is $\del_{G}$ when restricted to $F_w,w\in V_{\G}$ and $\del_{\G'}$ when restricted to $F_{v'}, v'\in V_{\G'}$.

\subsubsection{Compatibility}
Let $\G$ and $\G'$ be two graphs without flags, then for any vertex $v$ of $\G$
 $\leaf(\G\circ_v\G')=\leaf(\G)\circ_v\leaf(\G')$.

\subsection{Graphs with tails and without tails}
\label{kreimerapp}
There are two equivalent pictures one can use for the (co--)operad structure underlying the Connes--Kreimer Hopf algebra of rooted trees. One can either work with tails that are flags, or with tail vertices. These two concepts are of course equivalent in the setting where if one allows flag tails, disallows vertices with  valence  one and  vice--versa if one disallows tails, one allows  one--valenced vertices called tail vertices. In \cite{CK} graphs without tails are used.
Here we collect some combinatorial facts which represent this equivalence as a useful dictionary.

There are the obvious two maps which either add a vertex at each the end of each tail, or, in the other direction, simply delete each valence one vertex and its unique incident flag, but what  is relevant for the Connes--Kreimer example is another set of  maps. The first takes a graph with no flag tails to the tree  which to every  vertex, we {\em add} a tail, we will denote this map by $\sharp$ and we add one extra (outgoing) flag to the  root, which will be called the root flag.

The second map $\flat$ simply deletes all tails. We see that $\flat\circ \sharp=id$. But $\flat$ is not the double sided inverse, since $\sharp\circ \flat$ replaces any number of tails at a given vertex by one tail.
It is the identity on the image of $\sharp$, which we call single tail graphs.

Notice that $\sharp$ is well defined on leaf labeled trees by just transferring the labels as sets. Likewise $\flat$ is well defined on single tail trees
again by transferring the labels. This means that each vertex will be labeled.

There are the following degenerate graphs which are allowed in the two setups: the empty graph $\emptyset$ and the graph with one flag and no vertices $|$. We declare that
\begin{equation}
\emptyset^{\sharp}=| \text{  and vice--versa  } |^{\flat}=\emptyset
\end{equation}

\subsubsection{Planted vs. rooted}

A planted tree is a rooted tree whose root has valence $1$. One can plant a rooted tree $\tau$ to obtain a new planted rooted tree $\tau^{\downarrow}$, by adding a new vertex which will be the root of $\tau^{\downarrow}$ and adding one edge between the new vertex and the old root. Vice--versa, given a planted rooted tree $\tau$, we let $\tau^{\uparrow}$
be the uprooted tree that is obtained from $\tau$ by deleting the root vertex and its unique incident edge, while declaring the other vertex of that edge to be the root.

\subsection{Operad structures on rooted/planted trees}
There are several operad structures on leaf--labeled trees which appear.

For rooted  trees without tails and labeled  vertices, we define
\begin{enumerate}
\item $\tau\circ_i\tau'$ is the tree where the $i$-th vertex of $\tau$ is identified with the root of $\tau'$. The root of the resulting tree being the image of the root of $\tau$.

\item  $\tau\circ^+_i\tau'$ is the tree where the $i$-th vertex of $\tau$ is joined to the root of $\tau'$ by a new edge, with the root of the resulting tree is then the image of the root of $\tau$.
\end{enumerate}
It is actually the second operad structure that underlies the Connes-Kreimer Hopf algebra.

One can now easily check that
\begin{equation}
\tau\circ^{+}_i\tau'=\tau\circ_i\tau^{\prime \downarrow}=(\tau^{\downarrow}\circ_i\tau^{\prime \downarrow})^{\uparrow}
\end{equation}

These constructions also allow us to relate the compositions of trees with and without tails as follows
\begin{equation}
(\tau^{\sharp}\circ_i\tau^{\prime \sharp})^{\flat}=\tau\circ^+_i\tau'
\end{equation}
where the $\circ_i$ operation on the left is the one connecting the $i$th flag to the root flag.

\subsubsection{Planar case: marking angles}
In the case of planar trees, we have to redefine $\sharp$ by adding a flag to every {\em angle} of a planar tree. The labels are then not on the vertices, but rather the angles.
The analogous equations hold as above. Notice that to give a root to a planar tree actually means to specify a vertex and an angle on it. Planting it connects a new vertex into that angle.

This angle marking is directly to the angle marking in Joyal duality, see below and Figures \ref{figure1} and \ref{angle}.
This also explains the appearance of angle markings in \cite{Gont}.

\bibliography{hopfbib}
\bibliographystyle{halpha}
\end{document}